\documentclass[amscd, 11pt]{amsart}
\usepackage{url}
\usepackage{mathrsfs} 
\usepackage[utf8]{inputenc}
\usepackage{comment}

\RequirePackage{xparse,xstring,etoolbox}
\RequirePackage{kvoptions}
\RequirePackage{mathtools,amssymb,amsthm} 

\usepackage{thmtools}
\declaretheorem[style=plain, parent=section]{theorem} 
\declaretheorem[style=plain, sibling=theorem]{lemma}
\declaretheorem[style=plain, sibling=theorem]{proposition}
\declaretheorem[style=plain, parent=theorem]{corollary} 
\declaretheorem[style=plain, sibling=theorem]{conjecture}
\declaretheorem[style=definition, sibling=theorem]{definition}
\declaretheorem[style=definition, sibling=theorem]{example}

\declaretheorem[style=remark, sibling=theorem]{remark}
\declaretheorem[style=definition, sibling=theorem]{notation}

\SetMathAlphabet\mathrm{normal}{OT1}{\rmdefault}{m}{n}
\SetMathAlphabet\mathsf{bold}{OT1}{\sfdefault}{b}{n}
\RequirePackage{microtype}
\RequirePackage[nodisplayskipstretch, onehalfspacing]{setspace}

\RequirePackage{geometry}
\RequirePackage[dvipsnames]{xcolor}

\RequirePackage{tikz}
\usetikzlibrary{calc, cd, math, positioning}

\RequirePackage[shortlabels, inline]{enumitem}

\RequirePackage[colorlinks]{hyperref} 
\definecolor{urlcolor}{rgb}{0.9, 0, 0.5}
\definecolor{linkcolor}{rgb}{0, 0.2, 0.8}
\definecolor{citecolor}{rgb}{0, 0.6, 0.3}
\hypersetup{urlcolor=urlcolor, linkcolor=linkcolor, citecolor=citecolor}

\NewDocumentCommand{\hl}{m}{\textbf{\textcolor{\mrk@hlcolor}{#1}}}
\NewDocumentCommand{\ol}{}{\overline}
\NewDocumentCommand{\ul}{}{\underline}
\NewDocumentCommand{\wh}{}{\widehat}
\NewDocumentCommand{\wt}{}{\widetilde}

\NewDocumentCommand{\xr}{m}{\xrightarrow{#1}}

\RenewDocumentCommand{\aa}{}{\mathbb{A}}

\NewDocumentCommand{\ff}{}{\mathbb{F}}

\NewDocumentCommand{\gggg}{}{\gg}
\let\gggg\gg
\RenewDocumentCommand{\gg}{}{\mathbb{G}}

\RenewDocumentCommand{\ll}{}{\mathbb{L}}

\NewDocumentCommand{\pp}{}{\mathbb{P}}
\NewDocumentCommand{\qq}{}{\mathbb{Q}}
\NewDocumentCommand{\rr}{}{\mathbb{R}}

\NewDocumentCommand{\zz}{}{\mathbb{Z}}

\NewDocumentCommand{\eee}{}{\mathcal{E}}
\NewDocumentCommand{\fff}{}{\mathcal{F}}

\NewDocumentCommand{\iii}{}{\mathcal{I}}

\RenewDocumentCommand{\lll}{}{\mathcal{L}}

\NewDocumentCommand{\ooo}{}{\mathcal{O}}

\NewDocumentCommand{\rrr}{}{\mathcal{R}}

\NewDocumentCommand{\uuu}{}{\mathcal{U}}

\DeclareMathOperator{\im}{im}

\DeclareMathOperator{\id}{id}

\RequirePackage{tikz-cd} 

\DeclareMathOperator{\spec}{Spec}

\DeclareMathOperator{\divv}{div}

\RenewDocumentCommand{\hom}{m}{\textnormal{Hom}_{#1}}

\NewDocumentCommand\vvec{m m}{\ensuremath{\begin{pmatrix}#1\\#2\end{pmatrix}}}

\NewDocumentCommand\mf{m}{\ensuremath{\mathfrak{#1}}}
\NewDocumentCommand\msf{m}{\ensuremath{\mathsf{#1}}}
\NewDocumentCommand\mcal{m}{\ensuremath{\mathcal{#1}}}

\NewDocumentCommand\into{}{\ensuremath{\hookrightarrow}}
\NewDocumentCommand\onto{}{\ensuremath{\twoheadrightarrow}}

\makeatletter
\newcommand{\colim@}[2]{%
  \vtop{\m@th\ialign{##\cr
    \hfil$#1\operator@font colim$\hfil\cr
    \noalign{\nointerlineskip\kern1.5\ex@}#2\cr
    \noalign{\nointerlineskip\kern-\ex@}\cr}}%
}
\newcommand{\colim}{%
  \mathop{\mathpalette\colim@{\rightarrowfill@\textstyle}}\nmlimits@
}
\makeatother

\newtheorem{thmx}{Theorem}

\newcommand{\La}{\Lambda}

\newcommand{\BZ}{\mathbb{Z}}
\newcommand{\la}{\lambda}
\newcommand{\fr}{\msf{Fr}}
\newcommand{\psl}{\msf{PSL}}
\newcommand{\st}{\mathbf{St}}
\newcommand{\uzero}{\uuu_0(\mf{g})}
\newcommand{\kt}{K_0^{\wt{T}\times \wt{T}}(X)}
\newcommand{\ktq}{K_0^{\wt{T}\times \wt{T}}(X)_{\mathbb{Q}}}
\newcommand{\todd}{\textnormal{td}_X}
\makeatletter

\title{Decomposition of Frobenius pushforwards of line bundles on wonderful compactifications}

\author{Merrick Cai}
\address{Department of Mathematics
Massachusetts Institute of Technology
\newline
77 Massachusetts Avenue,
Cambridge, MA 02139,
USA
}
\email{mercai@mit.edu, merrickcai@gmail.com}

\author{Vasily Krylov}
\address{Department of Mathematics
Massachusetts Institute of Technology
\newline
77 Massachusetts Avenue,
Cambridge, MA 02139,
USA;
\newline National Research University Higher School of Economics, Russian Federation\newline
Department of Mathematics, 6 Usacheva st., Moscow 119048;
}
\email{krvas@mit.edu, krylovasya@gmail.com}

\date{September 2022}

\begin{document}

\maketitle
\begin{abstract}
De Concini-Procesi introduced varieties known as wonderful compactifications, which are smooth projective compactifications of semisimple adjoint groups $G$. We study the Frobenius pushforwards of invertible sheaves on the wonderful compactifications, and in particular its decomposition into locally free subsheaves. We give necessary and sufficient conditions for a specific line bundle to be a direct summand of the Frobenius pushforward of another line bundle, formulated in terms of the weight lattice of $\wt{G}$, the universal cover of $G$ (identified with the Picard group of the wonderful compactification). In the case of $G=\psl_n$, we offer lower bounds on the multiplicities (as direct summands) for those line bundles satisfying the sufficient conditions. We also decompose Frobenius pushforwards of line bundles into a direct sum of vector subbundles, whose ranks are determined by invariants on the weight lattice of $G$. We study a particular block which decomposes as a direct sum of line bundles, and identify the line bundles which appear in this block. Finally, we present two approaches to compute the class of the Frobenius pushforward of line bundles on wonderful compactifications in the rational Grothendieck group and in the rational Chow group.
\end{abstract}

\tableofcontents

\section{Introduction}
Let $G$ be a semi-simple adjoint group over an algebraically closed field $\ff$ of characteristic $p > 0$. Let $X \supset G$ be the De Concini-Procesi wonderful compactification of $G$. The Picard group of  $X$ can be identified with the weight lattice $\Lambda$ of $\widetilde{G}$ (universal cover of $G$). Let $\fr \colon X \rightarrow X$ be the (absolute) Frobenius morphism. Pick $\mu \in \La$ and consider the corresponding line  bundle $\mathcal{L}=\mathcal{O}_X(\mu)$. Smoothness of $X$ implies that $\fr_* \mathcal{O}_X(\mu)$ is the vector bundle of rank $p^{\operatorname{dim}G}$ on $X$. 

In this paper we investigate vector bundles $\fr_* \mathcal{O}_X(\mu)$, $\mu \in \Lambda$. Note that the  Krull-Schmidt theorem holds in  the category of vector bundles on $X$  (see \cite[Theorem 3]{atiyah1956krull}) so  
the decomposition of $\fr_*\mathcal{L}$ (or more generally any vector bundle on $X$) into a direct sum of indecomposable vector subbundles is unique. The natural question is to describe inecomposable summands of $\fr_*\mathcal{L}$ and their multiplicities.  In the case of toric varieties, it is known that Frobenius pushforwards of line bundles decompose into a direct sum of line bundles, and these can be calculated explicitly (this is proven in \cite{thomsen2000frobenius}, see also \cite{bogvad1998splitting}, \cite{achinger2010note} for the alternative approach). Beyond toric varieties, the answer is known in few cases: quadrics, due to \cite{achinger2012frobenius}; Grassmannian $\operatorname{Gr}(2,n)$, due to \cite{raedschelders2019frobenius}; and certain flag varieties, due to \cite{samokhin2014frobenius} and \cite{https://doi.org/10.48550/arxiv.1705.10187}. 

\begin{remark}
There are also some general conjectures 
in this direction that are known to be true for partial flags and due to  Bezrukavnikov,  Mirkovi\'c and Rumynin: namely the so called ``$p$-uniformity'' property holds in these cases and is conjectured to be true in general (see \cite[Section 1.5]{raedschelders2019frobenius} for the details).
\end{remark}

In general, the problem is not well studied. In our case, we cannot hope for as nice a result as in the toric varieties case: due to a result by P. Achinger, such total splittings into line bundles occur if and only if the scheme is a (smooth) toric variety. However, due to the extensive structure of the wonderful compactifications, there is much to say about the Frobenius pushforwards of line bundles, and we investigate this in our paper. One of our main results (Theorem \ref{thmC:decomp_of_frob_push_as_vector_subbundles}) actually holds for any smooth $G \times G$-variety $X$ that contains $G$ as an open subset. In this Theorem we decompose $\fr_* \mathcal{L}$ in the direct sum of (not necessary indecomposable) vector bundles and compute their dimensions. It turns out that the dimensions of these summands are independent on $p$ (for $p \gggg 0$) so this result may be considered as an evidence of the so-called ``$p$-uniformity'' (see \cite{raedschelders2019frobenius}). Let us mention that the decomposition above is constructed as follows: we consider the natural action of $G_1 \times G_1$ on $\fr_* \mathcal{L}$ (here $G_1 \subset G$ is the Frobenius kernel), this action corresponds to the action of $\uzero \otimes \uzero$ on $\fr_* \mathcal{L}$ and we can then decompose $\fr_* \mathcal{L}$ into the direct sum of vector bundles via primitive idempotents of $\uzero \otimes \uzero$.

\begin{remark}
Note that the decomposition of $\fr_* \mathcal{L}$ above is completely analogous to the decomposition of the Frobenius push forwards of line bundles on toric varieties constructed in \cite{bogvad1998splitting} (in the case of toric varieties summands will be line bundles).
\end{remark}


\subsection{Main results}

Let $X$ denote the wonderful compactification. The partial order $\succeq$ is defined in Definition~\ref{def:succeq}. The first result gives necessary and sufficient combinatorial conditions on $\textnormal{Pic}(X)\cong \Lambda$, the weight lattice of $\wt{G}$, for a line bundle to be a direct summand of the Frobenius pushforward of another line bundle.
\begin{thmx}
\label{thmA:line_bundles_direct_summands}
Let $\mu,\lambda\in\Lambda$. For $\ooo_X(\mu)$ to be a direct summand of $\fr_*\ooo_X(\lambda)$, it is \textit{necessary} that $\mu$ and $\lambda$ satisfy the condition
\[ K_X\succeq \lambda-p\mu \succeq 0,\]
and it is \textit{sufficient} that $\lambda-p\mu$ can be written as $\sum_{i=1}^{\ell}(a_i\omega_i+b_i\alpha_i)$ for $0\le a_i\le 2(p-1)$ and $0\le b_i\le p-1$.
\end{thmx}
\noindent
Theorem~\ref{thmA:line_bundles_direct_summands} can be found in Proposition~\ref{conditions_on_Homs} (combined with Corollary~\ref{cor:condition_on_line_subbundles}) and Theorem~\ref{frobenius_splitting_as_lattice_points}.

When $G=\psl_n$, the sufficiency condition in Theorem~\ref{thmA:line_bundles_direct_summands} gives rise to a larger lower bound on the multiplicity of that line bundle as a direct summand. If $D=\sum a_i D_i$ is an effective divisor, such that $a_i\ge 0$ and $D_i$ are prime divisors, recall that $D'\subset D$ is an effective sub-divisor if $D'=\sum b_i D_i$ for $0\le b_i\le a_i$, as defined in Definition~\ref{def:sub_divisor}.
\begin{thmx}
\label{thmB:PSL_n_high_multiplicity_of_line_bundle}
Let $G=\psl_n$. Let $S(\lambda)$ denote the number of distinct effective subdivisors of $(p-1){\wt{K}_X}$ (see Notation~\ref{notation:tilde_K_X}) whose class is $\lambda\in \textnormal{Pic}(X)$. Then the multiplicity $m(\mu,\lambda)\ge S(\lambda-p\mu)$.
\end{thmx}
\noindent Theorem~\ref{thmB:PSL_n_high_multiplicity_of_line_bundle} can be found in Theorem~\ref{thm:psl_n_number_of_splittings}.

Our third result decomposes the Frobenius pushforward of a line bundle as a direct sum of vector subbundles, where the ranks can be expressed in terms of invariants largely independent of $p$. We define $a_\lambda$ to be the size of the linkage class of $\lambda\in\Lambda/p\Lambda$, as in Notation~\ref{notation:a_lambda}. We define $d_\lambda$ to be the multiplicity of $L_\lambda$ in $\Delta_\lambda$, the baby Verma module, as in Definition~\ref{definition:d_lambda}.
\begin{thmx}
\label{thmC:decomp_of_frob_push_as_vector_subbundles}
Let $\lll$ be a line bundle on $X$. We have an abstract decomposition
\[\fr_*\lll \cong\bigoplus_{\lambda\in \Lambda_p }\bigoplus_{\mu\sim \lambda} \bigoplus_{\substack{1\le i\le \dim L_\mu\\ 1\le j\le \dim L_\lambda}}\fff_{\mu,\lambda}^{i,j}, \]
where $\fff_{\mu,\lambda}^{i,j}$ is a vector bundle of rank $a_\lambda d_\lambda d_\mu$, and these can be chosen to be vector subbundles (so that the isomorphism is in fact an equality).

In particular, the ranks of the summands are uniformly bounded by $(\max_{\lambda} d_\lambda)^2\cdot |W|$, and for $p\gggg 0$ this is independent of $p$ (see Theorem~\ref{thm:behavior_of_d_lambda}).
\end{thmx}
\noindent
Theorem~\ref{thmC:decomp_of_frob_push_as_vector_subbundles} can be found in Theorem~\ref{thm:decomp_line_bundle_into_vector_bundles}.

The final result describes which of the $\fff_{\mu,\lambda}^{i,j}$ are line bundles, and determines them explicitly. Let $\st$ denote the Steinberg representation (see Definition~\ref{def:steinberg_module}). Let $\rho$ denote one-half the sum of positive roots (see Notation~\ref{notation:rho}).
\begin{thmx}
\label{thmD:st_st_component_description}
Let $\lambda\in \Lambda$, and let $\mu$ be maximal with respect to $\succeq$ among all weights $\gamma$ such that $\lambda-p\gamma\ge (p-1)\rho$. Letting $\pi_{(p-1)\rho}$ be the projection onto the block $(p-1)\rho$, we have the isomorphism of $(\ooo_X,G_1\times G_1)$-modules
\[\pi_{(p-1)\rho}\fr_*\ooo_X(\lambda)\cong \st\otimes  \st \otimes \ooo_X(\mu).\]
In particular, in the notation of Theorem~\ref{thm:decomp_line_bundle_into_vector_bundles},
\[\fff_{(p-1)\rho, (p-1)\rho}^{i,j}\cong \ooo_X(\mu),\]
and $\mu=\lambda=(p-1)\rho$ are the only pairs $(\mu,\lambda)$ for which $\fff_{\mu,\lambda}^{i,j}$ are line bundles.
\end{thmx}
\noindent
Theorem~\ref{thmD:st_st_component_description} can be found in Corollary~\ref{cor:steinberg_splits_into_line_bundles_general} and Theorem~\ref{thm:determine_line_bundles_for_p-1_rho_block}.

\subsection{Structure of paper}

In \S\ref{subsection:setup}, we discuss the general setup and hypotheses for the paper. We review the Frobenius morphism in \S\ref{subsection:frobenius_morphism}. We then briefly recall the relevant parts of modular representation theory in \S\ref{subsection:representation_theory_of_wt_G} and representation theory of Frobenius kernels in \S\ref{subsection:representation_theory_of_G_1}. In \S\ref{subsection:the_wonderful_compactification}, \S\ref{subsection:structure_of_X}, and \S\ref{subsection:divisors_and_line_bundles_on_X}, we describe the construction, structure, and line bundles on the wonderful compactification. Discussion on the Vinberg monoid and its connection to the wonderful compactification can be found in Appendix~\ref{section:appendix:vinberg_monoid}. We defer standard algebraic results to Appendix~\ref{subsection:assorted_results}.

In \S\ref{section:line_bundles_as_direct_summands}, we study when we can split off a line bundle as a direct summand from the Frobenius pushforward of another line bundle. In \S\ref{subsection:constraints_on_line_bundles_as_direct_summands}, we investigate the constraints that splitting off a line bundle would entail, and prove Proposition~\ref{conditions_on_Homs}, which is the necessary condition in Theorem~\ref{thmA:line_bundles_direct_summands}. In \S\ref{subsection:frobenius_splitting}, we use the method of Frobenius splitting to prove Theorem~\ref{frobenius_splitting_as_lattice_points}, which is the sufficient condition in Theorem~\ref{thmA:line_bundles_direct_summands}. When this sufficient condition is satisfied and $G=\psl_n$, we are able to obtain lower bounds on the multiplicities by showing that each distinct effective divisor (with certain constraints) which maps to a specific element of $\Lambda$ will provide a mutually compatible Frobenius splitting, thus obtaining a distinct summand. This is the content of Theorem~\ref{thmB:PSL_n_high_multiplicity_of_line_bundle}, and is proven in Theorem~\ref{thm:psl_n_number_of_splittings}.

In \S\ref{section:vector_subbundles_via_irreps}, we use the irreducibility of certain $G_1\times G_1$-representations and the $\wt{G}\times\wt{G}$-equivariant structure on line bundles to prove that the structure sheaf raised to a multiplicity is a vector subbundle of the Frobenius pushforward of specific line bundles. The main result is Theorem~\ref{thm:embedding_of_L_otimes_L}.

In \S\ref{section:splitting_via_idempotents}, we utilize the $G_1\times G_1$-action to apply primitive idempotents and decompose the Frobenius pushforward of a line bundle into vector subbundles. We outline the main strategy in \S\ref{subsection:the_main_strategy}, particularly in Proposition~\ref{prop:independence_of_line_bundles}. In \S\ref{subsection:Central_idempotents}, we apply this strategy to decompose $\fr_*\lll$, the Frobenius pushforward of a line bundle on the wonderful compactification, into a direct sum of vector subbundles of known rank which are also $G_1\times G_1$-modules. In \S\ref{subsection:primitive_idempotents}, we apply this strategy and decompose $\fr_*\lll$ into a direct sum of vector subbundles (which are of smaller rank but lose the $G_1\times G_1$-module structure), and express the ranks in terms of the multiplicities of irreducible modules for $\uzero$ inside their projective covers. The general theory is discussed in Appendix~\ref{subsection:appendix:theory_of_idempotents}. In \S\ref{subsection:computation_of_dimensions}, we express the multiplicities in terms of certain invariants ($a_\lambda$ and $d_\lambda$), proving Theorem~\ref{thmC:decomp_of_frob_push_as_vector_subbundles}, which is shown in Theorem~\ref{thm:decomp_line_bundle_into_vector_bundles}. We then investigate the vector subbundles which have rank $1$, proving Theorem~\ref{thmD:st_st_component_description} in Corollary~\ref{cor:steinberg_splits_into_line_bundles_general} and Theorem~\ref{thm:determine_line_bundles_for_p-1_rho_block}, where we also describe exactly which line bundles these subbundles are. Discussion of $d_\lambda$ values is given in Appendix~\ref{section:appendix:decomposition_numbers}, and in particular, we apply Theorem~\ref{thmC:decomp_of_frob_push_as_vector_subbundles} to root systems $A_1$, $A_2$, $A_3$, $B_2$, and $G_2$, which we can describe explicitly and consequently in \S\ref{subsection:concrete_applications}, we give the explicit numbers for the ranks of the vector subbundles in the decomposition in Theorem~\ref{thm:decomp_line_bundle_into_vector_bundles}.

In \S\ref{section:class_in_K_theory}, we aim to describe the class of the Frobenius pushforward of a line bundle in the Grothendieck group of locally free sheaves on the wonderful compactification. Our first approach, in \S\ref{subsection:grothendieck_group_K_0(X)}, involves computing the class in the rational $\wt{T}\times \wt{T}$-equivariant Grothendieck group, resulting in Theorem~\ref{thm:compute_class_of_pushforward}. Our second approach, in \S\ref{subsection:chern_character}, involves computing the Chern character of $\fr_*\lll$ in the rational Chow group, which is isomorphic to the rational Grothendieck group (as $\qq$-algebras) by the Chern character. We present a formula involving the Chern character of $\lll$ and the Todd genus.

\subsection{Acknowledgements}
The authors would like to thank Roman Bezrukavnikov for suggesting the problem and for many useful discussions and ideas. The authors are grateful to the MIT SPUR program and to its advisors Ankur Moitra and David Jerison, through which this research was carried out. The authors also wish to thank Dylan Pentland and Andrei Ionov for explaining many aspects of modular representation theory, which greatly improved our understanding.

\section{Preliminaries}

\subsection{Setup}
\label{subsection:setup}

Throughout, we always work over $\ff$, an algebraically closed field of characteristic $p$, and we assume that all schemes are defined over $\ff_p\subset \ff$. Let $\fr$ denote the absolute Frobenius morphism. Note that since all schemes we work with are defined over $\ff_p$, and since $\ff$ is algebraically closed, the results are just as applicable to the relative Frobenius morphism as well. We always denote by $G$ to be a semisimple adjoint algebraic group over $\ff$, and assume $\textnormal{char }\ff=p\nmid h$, where $h$ is the Coxeter number of $\mf{g}$ (the Lie algebra of $G$), so that in particular $G$ is smooth over $\ff$. (This is the condition that $p$ is ``very good," which is discussed for example in \cite[\S1]{mirkovic1999centers}.) It will be enough for our purposes to consider only simple (adjoint) $G$, because such semisimple adjoint groups are just products of simple adjoint groups. The wonderful compactification of a semisimple adjoint group is again a product of the wonderful compactifications of simple adjoint groups as well.

Let $\wt{G}$ be the simply connected cover of $G$, fixing a maximal torus $\wt{T}$, Borel subgroups $\wt{B}$ and $\wt{B}^{-1}$, and unipotent subgroups $\wt{U}$ and $\wt{U}^-$ inside $\wt{G}$. Fix $\wt{Z}\subset \wt{G}$ to be the center such that $\wt{G}/\wt{Z}\cong G$, canonically identified with the kernel of the map $\wt{G}\to G$; by the assumption on $p$, we have that $\wt{Z}$ is a finite reduced subgroup scheme of $\wt{G}$. Checking the image of these inside $G$, we obtain maximal torus $T$, Borel subgroups $B,B^-$, and unipotent subgroups $U,U^-$ in $G$. Let $G_1$ denote the Frobenius kernel of $G$, i.e. the kernel of the map $\fr:G\to G$. Note that $G_1\cong (\wt{G})_1$.
We let $\mf{g}$ be the Lie algebra of $\wt{G}$ (hence $G$ as well) over $\ff$, and from our conventions above (fixing subgroups inside $\wt{G}$), obtain a Cartan subalgebra $\mf{h}\subset \mf{g}$. We let $\Phi$ denote the roots of $\mf{g}$ and $\rrr $ the root lattice, with $\rrr^+$ the submonoid generated by $\Phi^+$. Let $\Lambda$ denote the weight lattice of $\mf{g}$, with $\Lambda^+$ the submonoid of dominant weights. We also obtain a polarization of $\Phi$, i.e. a choice of positive roots $\Phi^+$, and simple roots $\Delta\subset \Phi^+$. From the choice of $\Delta$, we fix fundamental weights $\{\omega_i\}_{i=1}^{\textnormal{rk }\mf{g}}$, dual to the simple coroots. 

For convenience, we'll make the following notations.
\begin{notation}
Throughout the paper, let $\ell\coloneqq |\Delta|$.
\end{notation}
\begin{notation}
\label{notation:rho}
We denote 
\[\rho\coloneqq \frac{1}{2}\sum_{\alpha\in\Phi^+}\alpha=\sum_{i}\omega_i.\]
\end{notation}
\begin{notation}
We say that $\lambda_1\ge \lambda_2$ if $\lambda_1-\lambda_2\in \rrr^+$. (This is the usual partial ordering on a weight lattice.)
\end{notation}

Fix $W$ to be the Weyl group of $\mf{g}$, and $w_0$ for the longest element of $W$.
\begin{notation}
\label{notation:s_i_simple_reflection}
Denote $s_i\in W$ to be the simple reflections associated with the simple roots $\alpha_i\in\Delta$.
\end{notation}
Then $W$ acts on $\Lambda$ in two ways: first, we have the direct action $w\lambda$ in the usual way, and then we have the \textit{dot action}, given by
\[ w\cdot\lambda = w(\lambda+\rho)-\rho.\]
The dot action of $W$ effectively shifts the action of $W$ on $\Lambda$ by $-\rho$. We also have the affine Weyl group $(W^{\textnormal{aff}},\cdot_p)$ where $W^{\textnormal{aff}}=W\ltimes \rrr $. The group $W^{\textnormal{aff}}$ should be identified as controlling the representation theory of $\wt{G}$: $W^{\textnormal{aff}}$ acts on $\Lambda$, and its orbits give the block decomposition of $\msf{Rep}~\wt{G}$, see Proposition~\ref{linkage_principle_G}. The action is given by $(W,\cdot)$ with the dot action, and $\rrr $ acts by $\gamma:\bullet\mapsto \bullet+p\gamma$ for $\gamma\in\rrr $. On the other hand, the representation theory of $\msf{Rep}~G_1$ is controlled by $\Lambda/p\Lambda$. In this case, the $(W,\cdot)$-orbits of $\Lambda/p\Lambda$ give the block decomposition of $\msf{Rep}~G_1$, see \ref{linkage_principle_U(g)}.

\subsection{Frobenius morphism}
\label{subsection:frobenius_morphism}
Consider first an $\ff$-algebra $R$. Then the \textit{Frobenius map} $F:R\to R$ is defined by the $p$th power map:
\[ F:R\to R,\quad r\mapsto r^p.\]
Note that $F$ is \textit{not} an $\ff$-algebra homomorphism, as it is not $\ff$-linear. The image of $F$ is a subalgebra of $R$, denoted by $R^p$.

For any $R$-module $M$, we can push $M$ forward to a module $M'$ under Frobenius, which we'll denote by $\fr_* M$. In particular, $\fr_* M$ has the same underlying abelian group as $M$, but $R$ acts by $\fr(R)$. For $r\in R$ and $m\in \fr_* M$ corresponding to $m'\in M$,
\[ r\cdot m\leftrightarrow \fr(r)\cdot m'=r^p\cdot m'.\]
By abuse of notation, we will also say that $M$ is a module over the ring $\fr_* R$, to indicate the same action.

\begin{definition}
The \textit{(absolute) Frobenius morphism} $\fr:X\to X$ on any scheme $X$ over $\ff$ is defined as the identity morphism on the underlying topological space, and the $p$th power map on the structure sheaf $\ooo_X$.
\end{definition}
Concretely, $\fr$ does not do anything to the topological space, but for any open set $U\subset X$, we have the induced map
\[ \fr_X^\# : \ooo_X(U)\to \ooo_X(U),\quad s\mapsto s^p.\]

In particular, note that for any sheaf $\fff$ of $\ooo_X$-modules, the sheaf $\fr_*\fff$ is the same as $\fff$ as sheaves of abelian groups, but now the $\ooo_X$-module structure on $\fr_*\fff$ is given by $r\cdot s=r^p\cdot s$.

There are two important properties of the Frobenius morphism that we will use. The first is that the Frobenius morphism is flat, and as a consequence, for any locally free sheaf $\fff$ (i.e., vector bundle) on a smooth scheme, then $\fr_*\fff$ is also locally free (i.e., a vector bundle). The second: for any line bundle $\lll$ on a scheme $X$, then $\fr^*\lll\cong \lll^{\otimes p}$; see \cite[Lemma~1.2.6]{brion2007frobenius}. The projection formula then gives the immediate consequence
\[ \lll\otimes \fr_*\lll'\cong \fr_*\left(\lll'\otimes \lll^{\otimes p}\right),\quad \lll,\lll'\in\textnormal{Pic}(X).\]

Throughout our paper, $\fr$ always means the absolute Frobenius morphism. If $R=\ff[x_1,\dots,x_m]/I$, then $R^{(1)}=\ff[x_1^p,\dots,x_m^p]/^I{(p)}$, where $I^{(p)}=(f^{(p)}\mid f\in I)$ is the ideal generated in $\ff[x_1^p,\dots,x_m^p]$ by $f^{(p)}$, the image of $f\in I$ after sending each $x_i\mapsto x_i^p$ and extending $\ff$-linearly to a map $R\to R^{(1)}$. This map is clearly $\ff$-linear, hence glues to a morphism of $\ff$-schemes which we call the \textit{relative Frobenius morphism}; it is no longer an endomorphism but a map $X\to X^{(1)}$ from $X$ to the \textit{Frobenius twist} of $X$. However, because all schemes here will be defined over $\ff_p$, the schemes $X^{(1)}$ and $X$ are identified, and thus all results apply to the relative Frobenius morphism as well (with the appropriate statements).

\subsection{Representation theory of $\wt{G}$}
\label{subsection:representation_theory_of_wt_G}

We will denote by $\msf{Rep}~\wt{G}$, the category of finite dimensional rational representations of $\wt{G}$.
For any character $\lambda\in\Lambda$, we denote the one-dimensional $\wt{T}$-representation associated to $\lambda$ by $\ff_\lambda$. 
We also extend the $\wt{T}$-action to a $\wt{B}$-action on $\ff_\lambda$ via the natural surjection $\wt{B}\onto\wt{T}$, denoted with the same symbol.

\begin{definition}
Let $\lambda\in\Lambda$. The \textit{dual Weyl module} $M_\lambda$ is the $\wt{G}$-module
\[M_\lambda\coloneqq \textnormal{Ind}_{\wt{B}}^{\wt{G}}\hspace{1mm}\ff_{w_0\lambda}=\Gamma(\wt{G}/\wt{B}, \wt{G}\times^{\wt{B}}\ff_{w_0\lambda})=\Gamma(\wt{G}/\wt{B},\ooo_{\wt{G}/\wt{B}}(-w_0\lambda)).\]
\end{definition}

\begin{definition}
Let $\lambda\in \Lambda$. The \textit{Weyl module} $W_\lambda$ is the $\wt{G}$-module 
\[W_\lambda\coloneqq (\textnormal{Ind}_{\wt{B}}^{\wt{G}}\hspace{1mm}\ff_{-\lambda})^*=\Gamma(\wt{G}/\wt{B}, \wt{G}\times^{\wt{B}}\ff_{-\lambda})^*=\Gamma(\wt{G}/\wt{B},\ooo_{\wt{G}/\wt{B}}(w_0\lambda))^*.\]
\end{definition}

\begin{remark}
One can compute the dimension of $W_\lambda$ via the Weyl character formula, proven in \cite{weyl1925theorie}, \cite{weyl1926theoriea}, and \cite{weyl1926theorieb}.
\end{remark}

The irreducible representations of $\wt{G}$ are parametrized by $\Lambda^+$ by \cite[Corollary~II.2.6]{jantzen2003representations}.

\begin{definition}
For each $\lambda\in\Lambda^+$, define $L_\lambda$ to be the unique irreducible representation of $\wt{G}$ with highest weight $\lambda$. Note that $L_\lambda$ is the unique irreducible quotient of $W_\lambda$ as well as the unique irreducible subrepresentation of $M_\lambda$.
\end{definition}
\begin{example}
\label{example:weyl_module_for_SL2}
For $\wt{G}=\msf{SL}_2$, we have $\wt{G}/\wt{B}\cong \pp^1$. Let $\omega$ denote the fundamental weight. We have that \[\Gamma(\pp^1, \textnormal{Ind}_{\wt{B}}^{\msf{SL_2}}\hspace{1mm}\ff_{-\lambda})\cong \Gamma(\pp^1, \ooo_{\pp^1}(n\omega))\cong \text{Sym}^{n}(\ff^2).\] 
As a result, the Weyl module for the weight $n\omega$ is
\[ W_{n\omega}=\text{Sym}^{n}(\ff^2)^*.\]
For $0\le n<p$, $W_{n\omega}$ is both irreducible and self-dual, and therefore we have 
\[ L_{n\omega}=W_{n\omega}\cong \text{Sym}^{n}(\ff^2).\]
\end{example}

Of particular interest is when $\lambda=(p-1)\rho$.

\begin{definition}
\label{def:steinberg_module}
We denote $L_{(p-1)\rho}$ by $\st$, the \textit{Steinberg module}. 
\end{definition}

\begin{remark}
Note that $\st\otimes \st$ is the Steinberg module for $\wt{G}\times \wt{G}$.
\end{remark}
\begin{proposition}
The Steinberg module $\st$ has dimension $p^{\dim \wt{G}/\wt{B}}$ and is an irreducible $G_1$-module (and hence irreducible as a $\wt{G}$-module). The Steinberg module satisfies 
\[ \st=L_{(p-1)\rho}\cong M_{(p-1)\rho}\cong W_{(p-1)\rho}.\] Furthermore, it is projective and injective in $\msf{Rep}~\wt{G} $, hence self-dual. 
\end{proposition}
\begin{proof}
See \cite[Proposition~II.10.2]{jantzen2003representations} and \cite[Theorem~8.3]{steinberg1963representations}.
\end{proof}
\begin{example}
As in Example~\ref{example:weyl_module_for_SL2}, for $\wt{G}=\msf{SL}_2$, we find that
\[ \st \cong \text{Sym}^{p-1}(\ff^2).\]
\end{example}
\begin{remark}
The dimensions of the irreducible representations $L_\lambda$ can be computed generically for $p\gggg 0$ using the Weyl character formula and the affine Kazhdan-Lusztig polynomials. In Remark~\ref{remark:compute_dim_L_lambda}, we discuss this procedure. In the remark, it is only of interest for a weights corresponding to irreducible representations for $G_1$, but these lift to the irreducible modules for $\wt{G}$ of the same weight (see \cite{jantzen2003representations} or \S\ref{subsection:representation_theory_of_G_1}). However, the procedure in general holds for all $\lambda\in\Lambda$, and we are really specializing the procedure for irreducible $\wt{G}$-representations to irreducible $G_1$-representations.
\end{remark}

Representations of $\wt{G}$ in fact always decompose into \textit{blocks}:

\begin{proposition}[Linkage Principle, \cite{andersen1980strong}]
\label{linkage_principle_G}
We have a decomposition
\[ 
\msf{Rep}~\wt{G} =\bigoplus_{\lambda\in\Lambda/(W^{\textnormal{aff}},\cdot_p)}\msf{Rep}_\lambda(\wt{G}),\] where $W^{\textnormal{aff}}=W\ltimes \rrr $.
\end{proposition}
\noindent
This corresponds to the existence of central idempotents in $\ff[\wt{G}]$.

\subsection{Representation theory of $G_{1}$}
\label{subsection:representation_theory_of_G_1}

Our objective in this subsection is to recall the representation theory of $G_1$, the Frobenius kernel. To begin, we first define our object of interest.

\begin{definition}
The \textit{Frobenius kernel} of a connected semisimple algebraic group $G$ is the subgroup scheme $G_1\subseteq G$ which is the kernel of 
\[\fr: G\to G^{(1)}.\]
\end{definition}

The Frobenius kernel $G_1$ is a finite subgroup scheme, which is set-theoretically just a single point, the identity. However, scheme theoretically, it has length $p^{\dim G}$. Note that since $\wt{G}\to G$ is \'etale so the Frobenius kernels of $\wt{G}$ and $G$ coincide. The algebra of functions is given by 
\[ \ff[G_1]=\ff[G]/I^{(p)},\quad I^{(p)}=(f^p\textnormal{ for }f\in I),\]
where $I$ is the ideal cutting out the identity in $G$. Note that by the above discussion, this definition should coincide for all $G'$ with Lie algebra $\mf{g}$; one can check this is true, i.e. $\ff[G]/I^{(p)}\cong \ff[\wt{G}]/\wt{I}^{(p)}$.

For more detail on the structure of $G_1$, see \cite{jantzen2003representations}. However, for our purposes, it is enough to understand the following. The irreducible representations are exactly $L_\lambda$ for $\lambda\in \Lambda_p$, as defined below. In particular, they are indexed by $\Lambda/p\Lambda$.

\begin{definition}
\label{definition:Lambda_p}
We define the $p$\textit{-restricted weights} by 
\[\Lambda_p\coloneqq \{\lambda\in \Lambda \mid 0\le\langle \lambda, \alpha^\vee\rangle<p \text{ for all }\alpha\in\Delta\}. \]
In particular, when $\Lambda\cong \zz\{\omega_1,\omega_2,\dots,\omega_n\}$, then
\[\Lambda_p = \{a_1\omega_1+a_2\omega_2+\dots+a_n\omega_n\mid 0\le a_1,a_2,\dots,a_n<p\}.\]
\end{definition}

In particular, note that $\Lambda_p\subset \Lambda^+$, and that $-\rho\not\in \Lambda_p$. The goal of $\Lambda_p$ is to describe a class of representatives for $\Lambda/p\Lambda$ using dominant weights inside $\Lambda$.

\begin{example}
Let $G=\psl_2$, so that $\Lambda\cong \zz\omega_1\oplus \zz\omega_2$. Then $\Lambda/p\Lambda\cong \ff_p\omega_1\oplus \ff_p\omega_2$, while the $p$-restricted weights are given by 
\[ \Lambda_p= \{a_1\omega_1+a_2\omega_2\mid 0\le a_1,a_2<p\}.\]
\end{example}

It turns out the lifting $\Lambda/p\Lambda\to \Lambda$, given by the lifting of the irreducible $G_1$-modules to irreducible $\wt{G}$-modules, is precisely given by sending $\Lambda/p\Lambda\ni [\lambda]$ to its representative in $\Lambda_p$. By Steinberg's tensor product theorem, we can understand the restriction of $L_\lambda$ for any $\lambda\in\Lambda$ as a $\wt{G}$-module to $G_1$-module; it is only the $\lambda\in\Lambda^+$ that restrict to irreducible $G_1$-modules. Writing $\lambda=\lambda_0+p\lambda_1+p^2\lambda_2+\dots$ for $\lambda_i\in\Lambda_p$, we have that
\[ L_\lambda\cong L_{\lambda_0}\otimes L_{\lambda_1}^{(1)}\otimes L_{\lambda_2}^{(2)}\otimes \dotsm\]
as $G_1$-modules, where the twists indicate Frobenius twists of the representations. But since $G_1$ is the Frobenius kernel of $\wt{G}$, $L_{\lambda_i}^{(i)}$ is simply a trivial representation of dimension $\dim L_{\lambda_i}$ for all $i>0$.

Now note that since $G_1$ is a finite group scheme, it is a finite-dimensional Hopf algebra, hence
\[ \{G_1\textnormal{--modules}\}\leftrightarrow \{\ff[G_1]\textnormal{--comodules}\}\leftrightarrow \{\ff[G_1]^*\textnormal{--modules}\}.\]

To understand what $\ff[G_1]^*$ is, we first need to introduce the reduced enveloping algebra of $\mf{g}$.

\begin{definition}
The \textit{reduced enveloping algebra} $\uzero$ of $\mf{g}$ is defined to be
\[ \uzero\coloneqq \uuu(\mf{g})/(X^p-X^{[p]}),\]
where $(X^p-X^{[p]})$ is the ideal generated by the $p$-center.
\end{definition}

The ideal in the quotient is called the $p$-center, and is generated by elements of the form $X^p-X^{[p]}$, where $X^p$ is the usual $p$th power in $\uuu(\mf{g})$, and $X^{[p]}$ is the $p$th power as matrices, by embedding $\mf{g}\into \mf{gl}(\mf{g})$. (One can check that $X^p-X^{[p]}$ is indeed in the center of $\uuu(\mf{g})$, which justifies the name.) We always have a functor of ``differentiation" giving a map
\[ \msf{Rep}~\wt{G}\to\msf{Rep}~\mf{g},\]
which is an isomorphism in characteristic zero. In characteristic $p$, the relationship between $\msf{Rep}~\wt{G}$ and $\msf{Rep}~\mf{g}=\msf{Rep}~\uuu(\mf{g})$ is not as strong as in characteristic $0$. The way to remedy this is to understand the \textit{distributions} on $\wt{G}$ (or on $G$): the distribution algebras will essentially play the role of $\uuu(\mf{g})$ in characteristic $0$. However, in characteristic $p$, we actually have that 
\[ \textnormal{Dist}_1(G) =\textnormal{Dist}_1(\wt{G})\cong \textnormal{Dist}(G_1)\cong \uzero,\]
which is another explanation for why $G_1$ and $\uzero$ are deeply related. (See \cite[\S I.7]{jantzen2003representations} for details about distribution algebras.) 
It is easy to see that 
\[ \textnormal{Dist}_1(G)\cong \ff[G_1]^*,\] so we have that
\[ \ff[G_1]^*\cong \uzero.\] Therefore, in characteristic $p$, the differentiation map induces the equivalence
\[ \msf{Rep}~G_1\cong \msf{Rep}~\uzero.\]

Now $G_1$ is a group, so there is a $G_1\times G_1$ action given by left and right multiplication. From our identification above, this is equivalent to the standard $\uzero\otimes \uzero$ action on $\uzero^*$ by left and right multiplication (see the proof of Proposition~\ref{prop:independence_of_line_bundles} for an explanation).

\begin{lemma}
\label{lemma:u_0*=u_0}
As $\uuu_0(\mf{g})\otimes \uuu_0(\mf{g})$-modules, $\uuu_0(\mf{g})\cong \uuu_0(\mf{g})^*$, where $\uuu_0(\mf{g})$ and $\uzero^*$ are equipped with the standard bimodule structure (i.e., left and right multiplication).
\end{lemma}
\begin{proof}
By \cite{berkson1964u}, $\uuu_0(\mf{g})$ is a Frobenius algebra, and by \cite[Corollary~1]{humphreys1978symmetry} (note that \textit{loc. cit.}, the notation is $u_1$, which is identified with $\uuu_0(\mf{g})$, see \cite[Appendix~U]{humphreys2006ordinary}), $\uuu_0(\mf{g})$ is a symmetric algebra, as defined \cite[\S66]{curtis1966representation}. Therefore $\uzero$ satisfies the assumptions of Lemma~\ref{lemma:frobenius_symmetric_algebra_isomorphic_to_dual}, which implies the statement.
\end{proof}
\begin{remark}
There is a generalization of $\uzero$, given by $p$-characters $\chi\in \mf{g}^*$, which are also finite-dimensional quotients of $\uuu(\mf{g})$. In these cases, a nondegenerate associative bilinear form (i.e., a bilinear form satisfying the Frobenius property) is constructed in \cite{friedlander1988modular}, extending the form constructed for $\uzero$ as described in \cite{berkson1964u}. Furthermore, $\uuu_\chi(\mf{g})$ is even symmetric as above; see \cite[5.4]{strade1988modular} and \cite[Proposition~1.2]{friedlander1988modular}.
\end{remark}

Similar to $\wt{G}$, the category of representations of $G_1$ again decomposes via blocks, which are now by $W$-orbits in $\Lambda/p\Lambda$ (compared to $W^{\textnormal{aff}}$-orbits in $\Lambda$).

\begin{proposition}[Linkage Principle, \cite{Kac1976CoadjointAO}]
\label{linkage_principle_U(g)}
We have a decomposition 
\[\msf{Rep}~G_1 =\bigoplus_{\lambda\in (\Lambda/p\Lambda)/(W,\cdot)}\msf{Rep}_\lambda(G_1).\] We have an equivalence of categories $\msf{Rep}~G_1 \cong \msf{Rep}~\mcal{U}_0(\mf{g})$ (corresponding to the same block decomposition). This decomposition is again induced by central idempotents in $\uzero$, in bijection with blocks.
\end{proposition}

\subsection{The wonderful compactification}
\label{subsection:the_wonderful_compactification}
The main object of study in our paper is a variety known as the \textit{wonderful compactification} of $G$, constructed by De Concini and Procesi in \cite{concini1983complete}, and extended by Strickland to arbitrary characteristic in \cite{stricland1987vanishing}. The wonderful compactification is a ``nice" compactification of $G$ equipped with an action of $G\times G$, extending the natural $G\times G$ action on $G$ given by left and right multiplication. We will recall the construction of the wonderful compactification (especially in characteristic $p$) and its basic properties, primarily following \cite[\S6]{brion2007frobenius}. The essential idea is to embed $G$ into projective space $\pp(\textnormal{End}(M))$ for some suitable $\wt{G}$-module $M$, and then take the closure of $G$ inside. We recall another construction in Appendix~\ref{section:appendix:vinberg_monoid} using the Vinberg monoid.

First, we choose a regular weight $\lambda$ and a finite-dimensional $\wt{G}$-module $M$ satisfying the following properties:
\begin{itemize}
    \item The $\wt{T}$-eigenspace $M_\lambda$ of weight $\lambda$ is one-dimensional, and consists of $\wt{B}$-eigenvectors.
    \item All other weights of $M$ are $<\lambda$.
    \item For all $\alpha\in\Phi^+$, then $\mf{g}_{-\alpha}M_\lambda\ne 0$. In particular, the morphism
    \[ G/B\to \pp(M),\quad gB\mapsto g\cdot  M_\lambda\] 
    is a closed embedding.
    \item $M_{-\lambda}^*$ (i.e., the $\wt{T}$-eigenspace of weight $-\lambda$ in the dual module $M^*$) is a one-dimensional space of $B^-$-eigenvectors. In particular, the morphism
    \[ G/B^-\to \pp(M^*),\quad gB^-\mapsto g\cdot M_{-\lambda}^*\] is a closed embedding.
    \item The action of $\wt{G}$ on $\pp(M)$ factors through a faithful action of $G$.
\end{itemize}
By \cite[Lemma~6.1.1]{brion2007frobenius}, such modules exist for any regular weight $\lambda$. In particular, for $\lambda=(p-1)\rho$, we may take $M=\st$, the Steinberg module or more generally a Weyl module corresponding to any regular weight $\la$.

Now we fix $\lambda$ and $M$ as above. Consider the $\wt{G}\times \wt{G}$-module $\textnormal{End}(M)\cong M\otimes M^*$. Let $e\in\textnormal{End}(M)$ be the identity, with image $[e]\in \pp(\textnormal{End}(M))$. For the purposes of this construction, we'll denote 
\[ \pp\coloneqq \pp(\textnormal{End}(M)).\]
By \cite[Lemma~6.1.3]{brion2007frobenius}, the orbit
\[\wt{G}\times \wt{G}\cdot [e]\cong \wt{G}\times \wt{G}/((\wt{Z}\times \wt{Z})\textnormal{diag}(\wt{G}))\cong G.\]

Now take \[ X_G\coloneqq \ol{(\wt{G}\times \wt{G})\cdot [e]}\subseteq \pp,\]
the closure of this orbit inside $\pp$. 
\begin{definition}
We call $X_G$ the \textit{wonderful compactification of $G$}, and it is a projective compactification of $G$, equivariant with respect to $G\times G$.
\end{definition}
By \cite[Theorem~6.1.8(iv)]{brion2007frobenius}, $X_G$ is independent of the choices of $\lambda$ and $M$, in the sense that any choice of $\lambda$ and $M$ all yield isomorphic compactifications, and thus it makes sense to speak of ``the" wonderful compactification of $G$. When the group $G$ is clear, we will refer to $X_G$ by simply $X$.

\begin{notation}
Throughout, $X$ will denote the wonderful compactification of $G$. When necessary, we will specify $G$.
\end{notation}

\begin{example}
\label{example:construction_of_X_PSL2}
Let $G=\psl_2$ and $\wt{G}=\msf{SL}_2$. We can pick our regular weight to be $\omega$, the fundamental weight, and the corresponding irreducible representation is just $\textnormal{Sym}^1(\ff^2)=\ff^2$, the standard representation of $\msf{SL}_2$. Then we embed
\[ G=\psl_2\into \pp(\textnormal{End}(\ff^2))\cong \pp(\textnormal{Mat}_{2\times 2}(\ff))\cong \pp(\ff^4)\cong \pp^3.\]
Notice that $G$ is already a dense open subset of $\pp^3$: its complement is the codimension one closed subset $V(e_{11}e_{22}-e_{12}e_{21})$, hence $X=\ol{G}\cong\pp^3$. In other words, the wonderful compactification of $\psl_2$ is isomorphic to $\pp^3$.
\end{example}

\begin{example}
In the case of $G=\psl_n$, the wonderful compactification is the same as the older notion of ``space of complete collineations" of type $A_{n-1}$, see \cite{vainsencher1984complete}. In particular, for $G=\psl_3$, the wonderful compactification $X$ is isomorphic to the blowup of $\pp^8$ along the (image of the) Segre embedding of $\pp^2\times \pp^2$.
\end{example}

There is another method to construct the wonderful compactification via the Vinberg monoid. See Appendix~\ref{section:appendix:vinberg_monoid} for the construction of $X$ via Vinberg monoids. This approach has the advantage of understanding $X$ as a typical projective scheme, and understands line bundles by the graded rings and graded modules interpretation.

\subsection{Structure of $X$}
\label{subsection:structure_of_X}

Fix $\lambda,M$ as before. We have the identification 
\[ G\cong (\wt{G}\times \wt{G})\cdot [e]\subset X.\]
Let us also identify $T\subset G$ with
\[ T\cong (\wt{T}\times \wt{T})\cdot [e]\subset X.\]

\begin{notation}
We denote $\ol{T}$ to be the closure of $T$ in $X$, i.e. 
\[ \ol{T}\coloneqq \ol{(\wt{T}\times \wt{T})\cdot [e]}.\]
\end{notation}

Now let $\{m_i\}$ be a basis of $\wt{T}$-eigenvectors of $M$, with $\{m_i^*\}$ the dual basis of $M^*$. Then $m_i\otimes m_j^*$ forms a basis of $\wt{T}\times \wt{T}$-eigenvectors of $\textnormal{End}(M)$. We are particularly interested in the vector $m_\lambda\otimes m_\lambda^*$.

\begin{definition}
Define $\pp_0$ to be the complement in $\pp=\pp(\textnormal{End}(M))$ of $V(m_\lambda\otimes m_\lambda^*)$. Then define
\[\mathbf{X}_0\coloneqq X\cap \pp_0,\quad \ol{T_0}\coloneqq \ol{T}\cap \pp_0,\]
which are affine open subsets of $X$ and $\ol{T}$, respectively, and stable under $\wt{B}\times \wt{B}^-$ and $\ol{T}\times \ol{T}$, respectively.
\end{definition}

To study the structure of $X$, it is important to understand $\mathbf{X}_0$. We can map
\[ T\into \aa^\ell,\quad t\mapsto (\alpha_1(t^{-1}),\dots,\alpha_\ell(t^{-1})).\]
This map extends to an isomorphism
\[ \ol{T_0}\xr{\sim}\aa^\ell,\]
which is particularly useful to us.
\begin{notation}
Let $\gamma$ denote the inverse of the above map:
\[\gamma: \aa^\ell\xr{\sim} \ol{T_0}.\]
\end{notation}
More details on $\gamma$ can be found in \cite[Lemma~6.1.6]{brion2007frobenius}. However, what is particular convenient is that we may now use $\gamma$ to obtain the following $U\times U^-$-equivariant isomorphism:
\[ \wh{\gamma}:U\times \aa^\ell\times U^-\xr{\sim}\mathbf{X}_0,\quad (u,a,v)\mapsto u\cdot \gamma(a)\cdot v^{-1}.\]
In particular, $\mathbf{X}_0$ is isomorphic (as a scheme) to affine space. Furthermore, $X$ is covered by $\wt{G}\times \wt{G}$-translates of $\mathbf{X}_0$ (see \cite[Theorem~6.1.8(i)]{brion2007frobenius}), hence $X$ is smooth.

\subsection{Divisors and line bundles on the wonderful compactification}
\label{subsection:divisors_and_line_bundles_on_X}

Our goal in this subsection is to recall certain important prime divisors of $X$, along with the line bundles on $X$.

\begin{definition}
We define $\mathbf{X}_1,\mathbf{X}_2,\dots,\mathbf{X}_\ell$ to be the $\ell$ nonsingular prime divisors with normal crossings whose union is $X\setminus G$.
\end{definition}
In particular, $\mathbf{X}_i$ can be described explicitly. First, its restriction to $\ol{T_0}\cong \aa^\ell$ is given by $V(\alpha_i)$, i.e. the codimension $1$ subscheme where the $i$th coordinate of $\aa^\ell$ is set to zero. Then 
\[\mathbf{X}_i\cap \mathbf{X}_0= \wh{\gamma}(U\times V(\alpha_i)\times U^-).\]
The $G\times G$-orbits in $X$ are parametrized by $I\subseteq \{1,2,\dots,\ell\}$, so we will denote them by $\ooo_I$. Their closures are precisely of the form 
\[\ol{\ooo_I}=\mathbf{X}_I\coloneqq \bigcap_{i\in I}\mathbf{X}_i,\quad I\subset \{1,2,\dots,\ell\}.\]
In particular, there are exactly $2^\ell$ distinct $G\times G$-orbits in $X$. These orbits satisfy the containment relation
\[\ol{\ooo_I}\supseteq \ooo_J\iff I\subseteq J,\]
so in particular, $X$ contains a unique closed orbit.
\begin{definition}
We denote
\[ \mathbf{Y}\coloneqq \ooo_{\{1,2,\dots,\ell\}}=\mathbf{X}_1\cap\dotsm\cap\mathbf{X}_\ell\cong \wt{G}/\wt{B}\times \wt{G}/\wt{B},\]
the unique closed orbit of $X$ (see \cite[Theorem~6.1.8]{brion2007frobenius}).
\end{definition}

Now we turn to studying the codimension one closed subscheme $X\setminus \mathbf{X}_0$. It turns out that the irreducible components are precisely the prime divisors $\ol{Bs_iB^-}$ where $s_i\in W$ are the simple reflections, as in Notation~\ref{notation:s_i_simple_reflection}.

\begin{definition}
For $i=1,2,\dots,\ell$, define the following divisors:
\begin{align*}
    \mathbf{D}_i&\coloneqq \ol{B s_i w_0 B},\\
    \wt{\mathbf{D}_i}&\coloneqq (w_0,w_0)\cdot \mathbf{D}_i=\ol{B^- w_0 s_i B^-}.
\end{align*}
We call the $\mathbf{D}_i$ the \textit{Schubert divisors} and the $\wt{\mathbf{D}_i}$ the \textit{opposite Schubert divisors}.
\end{definition}
Note that $\mathbf{D}_i\sim \wt{\mathbf{D}_i}\sim \ol{B s_i B^-}$ are all linearly equivalent.
\begin{example}
Let $G=\psl_2$. Recall from Example~\ref{example:construction_of_X_PSL2}, that $X\cong \pp^3$, with the underlying vector space being $\ff\{e_{11},e_{12},e_{21},e_{22}\}$. Explicitly, we have 
\begin{align*}
    \mathbf{X}_1 &= V(\det)=V(e_{11}e_{22}-e_{12}e_{21}),\\
    \mathbf{D}_1 &= V(e_{21}),\\
    \wt{\mathbf{D}_1} &= V(e_{12}).
\end{align*}
\end{example}
\begin{definition}
We define 
\[ K_X\coloneqq -2\sum_{i=1}^{\ell}\mathbf{D}_i-\sum_{j=1}^{\ell}\mathbf{X}_j.\]
The divisor $K_X$ is a canonical divisor for $X$.
\end{definition}
\begin{notation}
By abuse of notation, we will also use $K_X$ to denote the class of $K_X$ in $\textnormal{Pic}(X)$, which we will later find to be $-2\rho-\sum_{\alpha\in\Delta}\alpha$.
\end{notation}

Now we recall the Picard group of $X$; see \cite[\S 6.1.B]{brion2007frobenius} for details. Note that the restriction of any line bundle $\lll$ to $\mathbf{Y}$ is isomorphic to a $\wt{G}\times \wt{G}$-linearly line bundle $\ooo_{\wt{G}/\wt{B}}(\lambda)\boxtimes \ooo_{\wt{G}/\wt{B}}(\mu)$ for $\lambda,\mu\in \Lambda$. It turns out that the restriction map $\textnormal{Pic}(X)\to\textnormal{Pic}(Y)$ is injective, and furthermore the image is precisely the classes
\[ \ooo_{\wt{G}/\wt{B}}(\lambda)\boxtimes \ooo_{\wt{G}/\wt{B}}(-w_0\lambda).\]
In other words, 
\[ \textnormal{Pic}(X)\cong \Lambda,\]
and the restriction map is described by
\[ \textnormal{res}:\Lambda\cong\textnormal{Pic}(X)\to \textnormal{Pic}(\mathbf{Y})\cong \Lambda\times \Lambda, \quad \lambda\mapsto (\lambda,-w_0\lambda).\]
\begin{remark}
Some authors define $\ooo_X(\lambda)$ to be the line bundle corresponding to the restriction map $\lambda\mapsto (-w_0\lambda,\lambda)$, and others define it as we do.
\end{remark}

The divisors we constructed above can be explicitly realized in the following way. The divisors $\mathbf{X}_i$ correspond to roots $\alpha_i\in\Lambda$, so that 
\[\textnormal{res }\ooo_X(\mathbf{X}_i)\cong \ooo_{\wt{G}/\wt{B}}(\alpha_i)\boxtimes \ooo_{\wt{G}/\wt{B}}(-w_0\alpha_i).\]
The divisors $\mathbf{D}_i$ (and $\wt{\mathbf{D}_i}$, since they are linearly equivalent) correspond to fundamental weights $\omega_i$, so that 
\[\textnormal{res }\ooo_X(\mathbf{D}_i) \cong \ooo_{\wt{G}/\wt{B}}(\omega_i)\boxtimes \ooo_{\wt{G}/\wt{B}}(-w_0\omega_i).\]

We will therefore make the following convention throughout the paper:
\begin{definition}
Define $\ooo_X(\lambda)$ to be the unique line bundle on $X$ which restricts to $\ooo_{\wt{G}/\wt{B}}(\lambda)\boxtimes \ooo_{\wt{G}/\wt{B}}(-w_0\lambda)$.
\end{definition}

We have (see for example \cite[Proposition~6.1.11]{brion2007frobenius}) that $\ooo_X(\lambda)$ is globally generated if and only if $\lambda$ is dominant, and ample if and only if $\lambda$ is regular. Furthermore, \cite[Theorem~3.2(i)]{stricland1987vanishing} shows that $\ooo_X(\lambda)$ has nonzero global sections if and only if $\lambda\ge \mu$ for some dominant weight $\mu$. See \S\ref{subsection:graded_module_vinberg_semigroup} for an interpretation of the invertible sheaves by graded modules; in particular, the filtration $F_{\le \lambda}$ of $\ff[\wt{G}]$ is identified with $\Gamma(X,\ooo_X(\lambda))$ and is discussed in more detail.

\section{Line bundles as direct summands of $\fr_*\lll$}
\label{section:line_bundles_as_direct_summands}
In this section, we give necessary and sufficient conditions for a line bundle on $X$ to be a direct summand of the Frobenius pushforward of another line bundle, which is the content of Theorem~\ref{thmA:line_bundles_direct_summands}. We also prove Theorem~\ref{thmB:PSL_n_high_multiplicity_of_line_bundle}, which allows us to give nontrivial lower bounds on the multiplicities of certain line bundles on $X_{\psl_n}$.

\subsection{Constraints on line bundles as direct summands}
\label{subsection:constraints_on_line_bundles_as_direct_summands}
We begin with the simple observation:
\begin{lemma}
If $\ooo_X(\mu)$ is a direct summand in $\fr_*\ooo_X(\lambda)$, then 
\begin{align*}
    \hom{\ooo_X}(\ooo_X(\mu),\fr_*\ooo_X(\lambda))&\ne 0,\\ 
    \hom{\ooo_X}(\fr_*\ooo_X(\lambda),\ooo_X(\mu))&\ne 0.
\end{align*}
\end{lemma}

From this, we first classify such Hom spaces.
\begin{definition}
\label{def:succeq}
Let $\succeq$ denote the partial ordering on $\Lambda$ such that $\gamma_1\succeq \gamma_2\iff \gamma_1-\gamma_2$ can be written as a nonnegative integer combination of fundamental weights and positive roots. In particular,
\[ \gamma_1\succeq \gamma_2 \iff \gamma_1-\gamma_2 \in \Lambda^+ +\rrr^+,\] and by \cite[Theorem~3.2(i)]{stricland1987vanishing},
\[\lambda\succeq 0\iff \Gamma(X,\ooo_X(\lambda))\ne 0.\] Define $\preceq$ similarly.
\end{definition}

It's clear that the relation $\succeq$ satisfies transitivity and reflexivity; it remains to check antisymmetry.

\begin{proposition}
\label{prop:when_is_succeq_relation}
The relation $\succeq$ is a partial order on $\Lambda$.
\end{proposition}
We present two proofs.
\begin{proof}
By \cite[Theorem~3.2(i)]{stricland1987vanishing},
$\Gamma(X,\ooo_X(\lambda))\ne 0\iff \lambda\succeq 0$.
Now suppose that $\lambda'\succeq \mu\succeq \lambda'$, and denote $\lambda\coloneqq \lambda'-\mu$.
Then 
$\Gamma(X,\ooo_X(\lambda))\ne 0,\quad \Gamma(X,\ooo_X(-\lambda))\ne 0$.
Now pick any nonzero $f\in\Gamma(X,\ooo_X(\lambda))$ and any nonzero $g\in\Gamma(X,\ooo_X(-\lambda))$. We then have that
\[ f\cdot g\in \Gamma(X,\ooo_X)\cong \ff\implies f\cdot g=1\]
after rescaling. But since $X$ is smooth projective and $\textnormal{Pic}(X)$ is torsion-free, we have that 
\[ 0= [\divv f] = [\divv_0 f]-[\divv_\infty g] =[\divv_0 f]-[\divv_0 g]\implies \lambda=-\lambda\implies \lambda =0\implies \lambda'=\mu.\]
\end{proof}
\begin{remark}
The exact same proof adapts more generally. Let $X$ be a smooth projective variety with torsion-free Picard group, and let $\succeq$ be any relation on $\textnormal{Pic}(X)$ which is transitive and reflexive, such that
\[ \Gamma(X,\ooo_X(\lambda))\ne 0\iff \lambda\succeq 0.\]
Then $\succeq$ satisfies the antisymmetry condition, i.e. $\succeq$ is a partial order on $\textnormal{Pic}(X)$.
\end{remark}
\begin{proof}
We may also prove this directly. The simple roots $\{\alpha_i\}$ form a basis of $\Lambda\otimes_\zz\rr$. Writing the $\omega_i$ in the basis of the $\{\alpha_i\}$ (with rational coefficients), the coefficients are nonnegative but not all nonzero, hence the sum of the coefficients (of either $\omega_i$ or $\alpha_i$) in the basis of $\{\alpha_i\}$ is positive, giving us a map $\varphi:(\Lambda,\succeq)\to (\qq,\ge)$ which sends $\lambda$ to the sum of its coefficients in the basis $\{\alpha_i\}$. If $\lambda_1\succeq \lambda_2$ and $\lambda_2\succeq \lambda_1$, then $0\ge\varphi(\lambda_1-\lambda_2)\ge 0\implies \varphi(\lambda_1-\lambda_2)=0\implies \lambda_1=\lambda_2$.
\end{proof}

We may now prove the necessary condition of Theorem~\ref{thmA:line_bundles_direct_summands}.
\begin{proposition}
\label{conditions_on_Homs}
Let $X$ be a wonderful compactification of a group $G$. Let $K_X=-2\rho-\sum_{\alpha\in \Delta}\alpha$ be (the class of) a canonical divisor for $X$. We have the following statements:
\begin{enumerate}
    \item $\hom{\ooo_X}(\ooo_X(\mu),\fr_*\ooo_X(\lambda))\ne 0\iff \lambda-p\mu\succeq 0$.
    \item $\hom{\ooo_X}(\fr_*\ooo_X(\lambda),\ooo_X(\mu))\ne 0\iff (1-p)K_X-(\lambda-p\mu)\succeq 0$.
\end{enumerate}
\end{proposition}
\begin{proof}
First, we make the identification 
\[\hom{\ooo_X}(\fff_1,\fff_2)\cong H^0(X,\ul{\textnormal{Hom}}(\fff_1,\fff_2))\cong H^0(X, \fff_1^\vee\otimes \fff_2).\]
In the first case, we have by the projection formula that 
\begin{align*}
    \hom{\ooo_X}(\ooo_X(\mu),\fr_*\ooo_X(\lambda))&=\Gamma(X,\ooo_X(-\mu)\otimes\fr_*\ooo_X(\lambda)),\\
    &=\Gamma(X,\fr_*(\ooo_X(\lambda)\otimes \fr^*\ooo_X(\mu))),\\
    &= \Gamma(X,\fr_*\ooo_X(\lambda-p\mu)),\\
    &= \Gamma(X,\ooo_X(\lambda-p\mu)).
\end{align*}
In the second case, we have the right adjoint by \ref{right_adjoint_to_fr_*}:
\begin{align*}
    \hom{\ooo_X}(\fr_*\ooo_X(\lambda),\ooo_X(\mu))&=\hom{\ooo_X}(\ooo_X(\lambda),\fr^*\ooo_X(\mu)\otimes \omega_X^{\otimes (1-p)}),\\
    &= \hom{\ooo_X}(\ooo_X(\lambda),\ooo_X(p\mu)\otimes \omega_X^{\otimes (1-p)}),\\
    &= \Gamma(X,\ooo_X(-\lambda)\otimes \ooo_X(p\mu)\otimes \omega_X^{\otimes (1-p)}),\\
    &= \Gamma(X, \ooo_X(-\lambda+p\mu+(1-p)K_X)).
\end{align*}
Finally, by \cite[Theorem~3.2(i)]{stricland1987vanishing} we conclude by noting that global sections of a line bundle $\ooo_X(\lambda)$ on $X$ are nonzero if and only if $\lambda\succeq 0$, and the conclusion follows.
\end{proof}
\begin{corollary}
\label{cor:condition_on_line_subbundles}
The only invertible sheaves $\ooo_X(\mu)$ which may appear as a direct summand in $\fr_*\ooo_X(\lambda)$ must satisfy the condition $(1-p)K_X\succeq \lambda-p\mu\succeq 0$.
\end{corollary}
\begin{proof}
If indeed $\ooo_X(\mu)$ is a summand of $\fr_*\ooo_X(\lambda)$, then there are nonzero maps $\ooo_X(\mu)\into \fr_*\ooo_X(\lambda)$ and $\fr_*\ooo_X(\lambda)\onto \ooo_X(\mu)$ corresponding to the canonical maps $A\into A\oplus B$ and $A\oplus B\onto A$.
\end{proof}

\begin{example}
Let $X$ be the wonderful compactification of $G=\psl_2$. Then the fundamental weight is $\omega$, and the simple root is $\alpha=2\omega$, and hence $K_X=-4\omega$, which is indeed the canonical divisor for $\pp^3\cong X$. In this case, the condition that $a\omega\succeq 0$ is precisely equal to $a\ge 0$. Now from Corollary~\ref{cor:condition_on_line_subbundles}, we find that the only line bundles $\ooo_{\pp^3}(k)$ which may appear as a direct summand in $\fr_*\ooo_{\pp^3}(n)$ must satisfy 
\[(1-p)(-4\omega)\succeq (n-pk)\omega\succeq 0\implies (1-p)(-4)\ge n-pk\ge 0,\] where $\ge$ is used as the ordinary ordering on $\zz$. This rearranges to the condition 
\[ \left\lceil \frac{4+n}{p}-4\right\rceil\le k\le \left\lfloor \frac{n}{p}\right\rfloor.\] In particular, for $n\equiv -1,-2,-3,-4\pmod{p}$, there are at most $4$ possibilities for $k$, which are $\lfloor n/p\rfloor -j$ for $j\in\{0,1,2,3\}$; for all other $n$, there are at most $5$ possibilities, which are $\lfloor n/p\rfloor -j$ for $j\in\{0,1,2,3,4\}$.
\end{example}

\begin{remark}
Fix a weight $\lambda\in\Lambda$. The set of weights $\mu\in\Lambda$ which satisfy the conditions of Corollary~\ref{cor:condition_on_line_subbundles} is finite and depends on $p$. However, for $p\gggg 0$ (in fact this bound is very small; for $G=\psl_3$, we need only $p>9$), the set of $\mu$ (depending on $\lambda$) stabilizes and is independent of $p$; furthermore, for any $\lambda$, the size of each set (of $\mu$) is bounded by some finite integer depending only on $G$ (or more accurately, $G$ defined over $\zz$), independent of $\lambda$ and $p$.
\end{remark}

\subsection{The possibilities on $X_{\psl_3}$}

Let $X$ be the wonderful compactification of $G=\psl_3$, and assume $p>3$. Then the fundamental weights are $\omega_1,\omega_2$, and the simple roots are $\alpha_1=2\omega_1-\omega_2$ and $\alpha_2=-\omega_1+2\omega_2$, and the only other positive root is $\alpha_1+\alpha_2$. It follows that $K_X=-3(\alpha_1+\alpha_2)=-3(\omega_1+\omega_2)$.
\begin{proposition}
\label{possibilities_on_X_PSL3}
Let $G=\psl_3$ and fix $\lambda\in\Lambda$. For each $\lambda$, let the number of distinct $\mu\in\Lambda$ for which $\lambda-p\mu\succeq 0$ and $(1-p)K_X-(\lambda-p\mu)\succeq 0$ be $n_\lambda$. Then for any $\lambda\in \Lambda$, we have $n_\lambda\le 27$, and for $p>9$, we additionally find that $n_\lambda\ge 21$.
\end{proposition}
\begin{proof}
Let $\lambda=\lambda_1\omega_1+\lambda_2\omega_2$, and $\mu=\mu_1\omega_1+\mu_2\omega_2$. First, we find that there must exist nonnegative integers $n_1,n_2$ such that 
\[ \vvec{\lambda_1-p\mu_1}{\lambda_2-p\mu_2}-n_1\vvec{2}{-1}-n_2\vvec{-1}{2}\ge \vvec{0}{0},\]
where $\ge$ denotes the usual (partial) ordering on $\zz^2$ (which in turn corresponds to $\succeq$ after identifying $\zz^2\xr{\sim} \Lambda$). For this, we have three cases:
\begin{enumerate}
    \item Both components of $\lambda-p\mu$ are already nonnegative, and hence we may take $n_1=n_2=0$.
    \item We have $\lambda_1-p\mu_1=x_1>0$ and $\lambda_2-p\mu_2=x_2<0$, but $|x_1|\ge 2|x_2|$.
    \item The reverse situation as above.
\end{enumerate}
The second condition implies the existence of nonnegative integers $m_1,m_2$ such that 
\[ \vvec{3(p-1)}{3(p-1)}\ge m_1\vvec{2}{-1}+m_2\vvec{-1}{2}+\vvec{\lambda_1-p\mu_1}{\lambda_2-p\mu_2}.\]
Therefore, let $x_i=\lambda_i-p\mu_i$. Then following the above cases:
\begin{enumerate}
    \item We have $0\le x_1,x_2$, and now we require that 
    \[ \vvec{3p-3-x_1}{3p-3-x_2}-m_1\vvec{2}{-1}-m_2\vvec{-1}{2}\ge 0.\] 
    Therefore, within the first quadrant, the possible $(x_1,x_2)$ points are bounded by the lines $x_2\le -\frac{1}{2}x_1+\frac{9(p-1)}{2}$ and $x_2\le -2 (x-\frac{9(p-1)}{2})$.
    \item We have $x_1>0$ and $x_2<0$, but $|x_1|\ge 2|x_2|$. In this case, the possible $(x_1,x_2)$ points are bounded by the lines $x_2\ge -\frac{1}{2}x_1$ and $x_2\le -2(x_1-\frac{9(p-1)}{2})$.
    \item In this case, we have $x_1<0$, $x_2>0$, and $2|x_1|\le |x_2|$. Now the possible $(x_1,x_2)$ points are bounded by the lines $x_2\ge -2x_1$ and $x_2\le -\frac{1}{2}x+\frac{9(p-1)}{2}$.
\end{enumerate}
Putting these together, we have the four lines bounding the region containing all possible points for $(x_1,x_2)$:
\begin{align*}
    x_2&\le -\frac{1}{2}x_1+\frac{9(p-1)}{2},\\
    x_2&\le -2 (x-\frac{9(p-1)}{2}),\\
    x_2&\ge -\frac{1}{2}x_1,\\
    x_2&\ge -2x_1.
\end{align*}

Now fix $\lambda$; each $\mu$ corresponds to a point $(x_1,x_2)$. Then the remaining choices of $\mu$ exactly correspond to $(x_1,x_2)\pm p\cdot (a_1,a_2)$ for integer $(a_1,a_2)$ for which this point still lies in the enclosed region. One can compute that the maximal value (for $p>9$, so that indeed $9p-9>8p$) of distinct $\mu$ is bounded above by $27$ and bounded below by $21$; two such $(x_1,x_2)$ values which indeed attain $27$ are $(-3(p-1),6(p-1))$, and $(6(p-1),-3(p-1))$ (i.e., the extreme edges), while two such $(x_1,x_2)$ values which attain $21$ are the $(p-1,p-1)$-shifts of those points, namely $(-2(p-1),5(p-1))$ and $(5(p-1),-2(p-1))$. For $p=5$ and $p=7$, we can compute by hand that there are less than $27$ such possibilities.
\end{proof}
\begin{corollary}
Let $G=\psl_3$. For any $\lambda\in\Lambda$, the number of distinct $\mu$ for which $\ooo_X(\mu)$ is a direct summand of $\fr_*\ooo_X(\lambda)$ is at most $27$.
\end{corollary}
\begin{remark}
In fact, the number of $(x_1,x_2)$ values (before choosing $\lambda$ and $\mu$) is precisely the lattice points in this region, which is given by Pick's formula:
\[ \# \{(x_1,x_2)\} = \textnormal{Area}+1+\frac{1}{2}\textnormal{Boundary points}.\]
We can compute the area by the Shoelace formula to obtain that the area is $27(p-1)^2$. On the other hand, we can compute the number of boundary points one by one. There are a total of $4\cdot 3(p-1)$ boundary points, hence the total number of lattice points is 
\[ \#\{(x_1,x_2)\} = 27(p-1)^2+1+6(p-1).\] Therefore, roughly speaking, we should expect at most $\frac{27(p-1)^2+1+6(p-1)}{p^2}\approx 27$ distinct values of $\mu$.
\end{remark}

\subsection{Frobenius splitting}
\label{subsection:frobenius_splitting}
In this subsection, we determine a class of line bundles which are \textit{Frobenius split} from the Frobenius pushforward of another line bundle.

\begin{definition}
\label{def:sub_divisor}
Let $D$ be an effective divisor in $\textnormal{Div}(X)$. Then an effective divisor $D'$ is an \textit{(effective) sub-divisor of} $D$ if $D-D'$ is still effective. In other words, the monoid generated by the prime divisors in $\textnormal{Div}(X)$ has a natural partial order, and the notion of an effective sub-divisor is exactly this partial order.
\end{definition}
We may now prove the sufficiency condition of Theorem~\ref{thmA:line_bundles_direct_summands}.
\begin{theorem}
\label{frobenius_splitting_as_lattice_points}
Let $\lambda,\mu\in\Lambda$ such that $\lambda-p\mu$ can be written in the form $\sum_{i=1}^{\ell}(a_i\omega_i+b_i\alpha_i)$ for $0\le a_i\le 2(p-1)$ and $0\le b_i\le p-1$. Then $\ooo_X(\mu)$ appears as a direct summand in $\fr_*\ooo_X(\lambda)$.
\end{theorem}
\begin{proof} 
By assumption, we have that $\lambda-p\mu\succeq 0$ and $K_X-(\lambda-p\mu)\succeq 0$. Therefore, by Proposition~\ref{conditions_on_Homs}, we have that 
\begin{itemize}
    \item $\hom{\ooo_X}(\ooo_X(\mu),\fr_*\ooo_X(\lambda))\ne 0$, and
    \item $\hom{\ooo_X}(\fr_*\ooo_X(\lambda),\ooo_X(\mu))\ne 0$. 
\end{itemize}
Now we want to find two maps such that we have the following commutative diagram:
\[
\begin{tikzcd}
\ooo_X(\mu)\arrow{r}\arrow{dr}{\id} & \fr_*\ooo_X(\lambda)\arrow{d}\\
& \ooo_X(\mu).
\end{tikzcd}
\]

By the projection formula, after tensoring by $\mathcal{O}_X(-\mu)$ this diagram becomes
\[
\begin{tikzcd}
\ooo_X\arrow{r}\arrow{dr}{\id} & \fr_*\ooo_X(\lambda-p\mu)\arrow{d}\\
& \ooo_X.
\end{tikzcd}
\]
Thus, we have a map 
\begin{equation}
\label{equation:1}
\begin{tikzcd}
\hom{\ooo_X}(\ooo_X,\fr_*\ooo_X(\lambda-p\mu))\otimes \hom{\ooo_X}(\fr_*\ooo_X(\lambda-p\mu),\ooo_X)\arrow{d}&\\
\hom{\ooo_X}(\ooo_X,\ooo_X)=\ff,
\end{tikzcd}
\end{equation}
and our goal is to find a some $f \otimes g$ that maps to a nonzero element of $\ff$. Our strategy is essentially the notion of Frobenius splitting relative to a divisor, as described in \cite[\S1.4]{brion2007frobenius}. Fix $D$ to be an effective divisor of the form 
\[D= \sum_{i=1}^{\ell}c_i\mathbf{D}_i+\wt{c_i}\wt{\mathbf{D}_i}+b_i\mathbf{X}_i,\] where $0\le c_i,\wt{c_i},b_i\le p-1$ and $[D]=\lambda-p\mu$; this is possible by the hypothesis on $\lambda-p\mu$ (we require that $0\le c_i+\wt{c_i}=a_i\le 2(p-1)$ and $0\le b_i\le p-1$). Fix the canonical section $\sigma_D$ of $\ooo_X(D)$ and regard it is a map $\ooo_X\to \ooo_X(D)$. We say that $X$ is \textit{Frobenius split relative to} $D$, or simply $D$-\textit{split}, if there exists an $\ooo_X$-linear map $\psi:\fr_*\ooo_X(D)\to \ooo_X$ such that the composition 
\[ \psi\circ\fr_*(\sigma_D)\in\hom{}(\fr_*\ooo_X,\ooo_X)\] is an $\ooo_X$-splitting. Now define the \textit{evaluation map} 
\[\epsilon_D: \ul{\textnormal{Hom}}_{\ooo_X}(\fr_*\ooo_X(D),\ooo_X)\to \ooo_X, \quad \psi\mapsto \psi(\fr_*(\sigma_D)).\] By \cite[Remark~1.4.2(i)]{brion2007frobenius}, $\psi$ is a $D$-splitting if and only if $\psi(\fr_*(\sigma_D))=1$, and this corresponds to the condition that $\ooo_X$ is a direct summand of $\fr_*\ooo_X(D)$. If we restrict the left factor of (\ref{equation:1}) to $\sigma_D\in\Gamma(X,\ooo_X(D))$, we recover the commutative diagram on \cite[pg.~39]{brion2007frobenius}:
\begin{equation}\label{equation:diagram_from_BK}
\begin{tikzcd}
\ul{\textnormal{Hom}}_{\ooo_X}(\fr_*\ooo_X(D),\ooo_X)\arrow{r}{\sim}\arrow{rd}{\epsilon_D} &\fr_*(\omega_X^{\otimes (1-p)}(-D))\arrow{d}{\tau\circ\fr_*(\sigma)}\\
 &\ooo_X.
\end{tikzcd}
\end{equation}

Next, by \cite[Exercise~III.6.10]{hartshorne2013algebraic}, we have the identification
\begin{align*}
    \ul{Hom}_{\ooo_X}(\fr_*(\ooo_X(\lambda-p\mu),\ooo_X))&\cong \fr_*\ul{Hom}_{\ooo_X}(\ooo_X(\lambda-p\mu), \fr^{!}\ooo_X),\\
    &\cong \fr_* \ul{Hom}_{\ooo_X}(\ooo_X(\lambda-p\mu),\omega_X^{\otimes(1-p)}),\\
    &\cong \fr_*( \omega_X^{\otimes(1-p)}\otimes \ooo_X(p\mu-\lambda)),\\
    &\cong \fr_*\ooo_X((1-p)K_X-(\lambda-p\mu)).
\end{align*}

Thus, (\ref{equation:1}) becomes
\begin{equation}\label{map_glob}
\begin{tikzcd}
 \Gamma(X,\fr_*\ooo_X((1-p)K_X-(\lambda-p\mu)))\arrow{d}{\epsilon_D}&\\
\Gamma(X,\ooo_X).
\end{tikzcd}
\end{equation}
By \cite[Proposition~1.3.7]{brion2007frobenius}, or by (\ref{equation:diagram_from_BK}), the map (\ref{map_glob}) is obtained as the composition 
\[
\begin{tikzcd}
\Gamma(X,\fr_*\ooo_X((1-p)K_X-(\lambda-p\mu)))\arrow{d}{-\circ \fr_*(\sigma_D)}&\\
\Gamma(X,\fr_*\omega_X^{\otimes(1-p)})\arrow{d}{\tau}&\\
\Gamma(X,\ooo_X),
\end{tikzcd}
\]
where the first map is the multiplication (or composition) with $\fr_*(\sigma_D)$ and the second map $\fr_*\omega_X^{\otimes(1-p)}\to \ooo_X$ is induced by the map $\wh{\tau}$, coming from the evaluation map $\epsilon:\ul{Hom}_{\ooo_X}(\fr_*\ooo_X,\ooo_X)\to \ooo_X$ which controls the Frobenius splitting of $X$ (see \cite[Theorem~1.3.8]{brion2007frobenius}).

One section of $\omega_X^{\otimes(1-p)}$ indeed gives a splitting (i.e. maps to $1 \in \ff=\Gamma(X,\mathcal{O}_X)$): this is precisely the section $\sigma$, described in \cite[Theorem~6.1.12]{brion2007frobenius}. We have that
\[\divv_0 \sigma = (p-1)\sum_{i=1}^{\ell}\left(\mathbf{X}_i+\mathbf{D}_i+\wt{\mathbf{D}_i}\right),\] so for convenience, make the following convention (which we will use in later parts of the paper as well):
\begin{notation}
\label{notation:tilde_K_X}
Denote \[ {\wt{K}_X}\coloneqq \sum_{i=1}^{\ell}\left(\mathbf{X}_i+\mathbf{D}_i+\wt{\mathbf{D}_i}\right).\]
\end{notation}
In particular, it suffices to show that there exists a global section on $\ooo_X((p-1){\wt{K}_X}-(\lambda-p\mu))$ which multiplies with $\sigma_D$ to $\sigma$, and then we will have shown that we have the desired splitting. The representability of $\sigma$ as a product of two sections is equivalent to a representation of $\divv_0  \sigma$ as a the sum of two effective divisors, each of which is the $\divv_0$ of the sections in question. Note that every effective sub-divisor (see Definition~\ref{def:sub_divisor}) of $\divv_0 \sigma$ is $\sum_{i=1}^{\ell}(c_i\mathbf{D}_i+\wt{c_i}\wt{\mathbf{D}_i}+b_i\mathbf{X}_i)$ for $0\le c_i,\wt{c_i},b_i\le p-1$. But note that by construction, $D$ is an effective sub-divisor of $(p-1){\wt{K}_X}$, and $\divv_0\sigma_D=D$; it follows that $(p-1){\wt{K}_X}-D$ is effective as well, with canonical section $\sigma_{(p-1){\wt{K}_X}-D}$. Therefore, we have that $\sigma_D$ and $\sigma_{(p-1){\wt{K}_X}-D}$ multiply to $\sigma$, and it follows that $\ooo_X(\mu)$ is indeed a direct summand of $\fr_*\ooo_X(\lambda)$.
\end{proof}

\begin{remark}
Let us give a quick recount of the section $\sigma$, described in \cite[Theorem~6.1.12]{brion2007frobenius}. Because the Steinberg module is self-dual, we have a map
\[ \xi:\st\otimes \st\into \Gamma(X,\ooo_X((p-1)\rho)).\]
Let $f\in \st$ be a highest weight vector. Then define $\tau_{(p-1)\rho}\coloneqq \xi(f\otimes f)$; this is the (unique, up to scalar) canonical section of $\ooo_X((p-1)\rho)$ which is a $B\times B$-eigenvector, as described in \cite[Proposition~6.1.11(iv)]{brion2007frobenius}. We have $\divv_0  \tau_{(p-1)\rho}=(p-1)\sum_{i=1}^{\ell}\mathbf{D}_i$. Now define $\tau_{-(p-1)\rho}\coloneqq (w_0,w_0)\cdot \tau_{(p-1)\rho}$. Then $\divv_0 \tau_{-(p-1)\rho}=(p-1)\sum_{i=1}^{\ell}\wt{\mathbf{D}_i}$. Now writing $\tau\coloneqq \tau_{-(p-1)\rho}\cdot\tau_{(p-1)\rho}$, we have $\tau\in\Gamma(X,\ooo_X{(2(p-1)\rho})$ and $\divv_0 \tau=(p-1)\sum_{i=1}^{\ell}(\mathbf{D}_i+\wt{\mathbf{D}_i})$. Now taking $\sigma_i$ the canonical sections of $\ooo_X(\mathbf{X}_i)$, we define
\[\sigma\coloneqq \tau\prod_{i=1}^{\ell} \sigma_i^{p-1}\in\Gamma(X,\omega^{\otimes(1-p)}),\quad \divv_0 \tau = (p-1)\sum_{i=1}^{\ell}(\mathbf{X}_i+\mathbf{D}_i+\wt{\mathbf{D}_i}).\]
\end{remark}

\begin{corollary}
\label{cor:min_max_mu_for_lambda}
Let $G=\psl_3$ and $p>9$. Then for each $\lambda$, Theorem~\ref{frobenius_splitting_as_lattice_points} guarantees between $14$ and $19$ distinct $\mu$ for which $\ooo_X(\mu)$ appears as a direct summand of $\fr_*\ooo_X(\lambda)$.
\end{corollary}
\begin{proof}
It's easily verified that this will not depend on $p$ for $p$ sufficiently large; the rest can be checked by hand.
\end{proof}

We have a few easy considerations regarding the upper bounds of the multiplicities.

\begin{proposition}
The multiplicity $m(\mu,\lambda)$ of $\ooo_X(\mu)$ as a direct summand in $\fr_*\ooo_X(\lambda)$ is easily bounded by the dimensions of global sections: we have
\[ m(\mu,\lambda)\le \min\left(\frac{\dim_{\ff} F_{\le \lambda}}{\dim_{\ff}F_{\le \mu}}, \dim_{\ff} F_{\le \lambda-p\mu}, \dim_{\ff} F_{\le (1-p)K_X-(\lambda-p\mu)}\right).\]
\end{proposition}
\begin{proof}
The first bound is clear: if we have $m(\mu,\lambda)$ copies of $\ooo_X(\mu)$ as direct summands inside $\fr_*\ooo_X(\lambda)$, then $F_{\le \mu}^{\oplus m(\mu,\lambda)}\into F_{\le\lambda}$. The next two bounds are clear from the fact that the dimensions of the vector spaces $\hom{\ooo_X}(\ooo_X(\mu),\fr_*\ooo_X(\lambda))$ and $\hom{\ooo_X}(\fr_*\ooo_X(\lambda),\ooo_X(\mu))$ are at least 
\[\dim \hom{\ooo_X}( \ooo_X(\mu),\ooo_X(\mu)^{\oplus m(\mu,\lambda)})=\dim \hom{\ooo_X}(\ooo_X(\mu)^{\oplus m(\mu,\lambda)}, \ooo_X(\mu))=m(\mu,\lambda).\]
\end{proof}
\begin{corollary}
\label{cor:easy_multiplicity_estimate_via_dim}
For each $\lambda$, there is exactly one copy of $\ooo_X(\lambda)$ as a direct summand in $\fr_*\ooo_X(p\lambda)$. In particular, there is exactly one copy of $\ooo_X$ as a direct summand in $\fr_*\ooo_X$ (which immediately follows from $X$ being Frobenius split).
\end{corollary}
\begin{proof}
The more general statement reduces to the statement about $\fr_*\ooo_X$ by projection formula:
\[ 
\ooo_X(\lambda)\otimes \fr_*\ooo_X \cong \fr_*\ooo_X(p\lambda).\]
Therefore it suffices to prove the statement for $\lambda=0$. In this case, by \cite[Theorem~6.1.12]{brion2007frobenius}, we have that $X$ is Frobenius split, hence we have at least one copy of $\ooo_X$ as a direct summand of $\fr_*\ooo_X$. By Corollary~\ref{cor:easy_multiplicity_estimate_via_dim}, the multiplicity is bounded above by $\dim F_{\le 0}=1$, hence is exactly $1$.
\end{proof}

\begin{corollary}
For any $\mu$, the multiplicity of $\ooo_X(\mu)$ in $\fr_*\ooo_X(\lambda)$ is exactly $1$, for the following $\lambda$:
\begin{itemize}
    \item $\lambda= (1-p)K_X+p\mu$
    \item $\lambda= (1-p)K_X+p\mu-\alpha_i$ for any simple root $\alpha_i$
    \item $\lambda = p\mu$
    \item $\lambda = p\mu +\alpha_i$ for any simple root $\alpha_i$.
\end{itemize}
\end{corollary}
\begin{proof}
We easily check that in all cases, $\lambda-p\mu$ falls into one of the following categories:
\begin{itemize}
    \item $(1-p)K_X$,
    \item $(1-p)K_X-\alpha_i$ for any simple root $\alpha_i$,
    \item $0$,
    \item $\alpha_i$ for any simple root $\alpha_i$.
\end{itemize}
In each case, $\lambda-p\mu$ can be written as $\sum m_i\omega_i+n_i\alpha_i$ with $0\le m_i\le 2(p-1)$ and $0\le n_i\le p-1$. It follows that the multiplicity $m(\mu,\lambda)\ge 1$. Now we apply Corollary~\ref{cor:easy_multiplicity_estimate_via_dim}, and note that $m(\mu,\lambda)$ is bounded by the following in each case:
\begin{itemize}
    \item $\dim_\ff F_{\le (1-p)K_X-(1-p)K_X}=\dim_{\ff} F_{\le 0}=1$,
    \item $\dim_\ff F_{\le (1-p)K_X-(1-p)K_X+\alpha_i}=\dim_{\ff} F_{\le \alpha_i}=1$,
    \item $\dim_\ff F_{\le 0}=1$,
    \item $\dim_\ff F_{\le \alpha_i}=1$.
\end{itemize}
It follows that $m(\mu,\lambda)=1$ in each of these cases.
\end{proof}

Notice that there are often a number of ways to choose an effective sub-divisor of $(p-1){\wt{K}_X}$ whose class is a fixed $\lambda\in \Lambda$. It is of interest to study how the remaining effective sub-divisors of $(p-1){\wt{K}_X}$ control the multiplicities $m(\mu,\lambda)$. 
For $G=\psl_n$, we are able to get an explicit lower bound, which is the content of Theorem~\ref{thmB:PSL_n_high_multiplicity_of_line_bundle}.

\begin{theorem}
\label{thm:psl_n_number_of_splittings}
Let $G=\psl_n$. Let $S(\lambda)$ denote the number of distinct effective subdivisors of $(p-1){\wt{K}_X}$ (see Notation~\ref{notation:tilde_K_X}) whose class is $\lambda\in \textnormal{Pic}(X)$. Then the multiplicity $m(\mu,\lambda)\ge S(\lambda-p\mu)$.
\end{theorem}
\begin{remark}
By applying generating functions, it's easy to see that $S(\lambda)$ is given by the formula
\[
[x^{\lambda}]\prod_{i=1}^{\ell}\left(\left(1+x^{L(\mathbf{D}_i)}+\dots+x^{L((p-1)\mathbf{D}_i)}\right)\left(1+\dots+x^{L((p-1)\wt{\mathbf{D}_i})}\right)\left(1+x^{L(\mathbf{X}_i)}+\dots+x^{L((p-1)\mathbf{X}_i)}\right)\right).
\]
(This is true in general, not just for $G=\psl_n$.) Note that this result is analogous to P. Achinger's formula for toric varieties, \cite{achinger2010note}. For example, in the case of projective space $\pp^m$, we have the formula 
$\fr_*\ooo_{\pp^m}(d)\cong \bigoplus_{e\in\zz}\ooo_{\pp^{m}}(e)^{\oplus m(e)}$, where
$m(e)=[x^{d-pe}](1+x+\dots+x^{p-1})^{m+1}=\sum_{i\ge 0}(-1)^i \binom{m+1}{i}\binom{d-pe+m-ip}{m}$ (and the binomials $\binom{a}{b}$ for $a<b$ are understood to be $0$).
\end{remark}
\begin{proof}
The strategy of the proof is as follows. For every effective subdivisor $D\subseteq (p-1){\wt{K}_X}$ mapping to $\lambda-p\mu$, we have a splitting by Theorem~\ref{frobenius_splitting_as_lattice_points}, inducing a map via
\[ \ooo_X(\mu)\xr{\sigma_D}\fr_*\ooo_X(\lambda)\xr{\sigma_{(p-1){\wt{K}_X}-D}} \ooo_X(\mu),\] whose composition is identity map. We will show that for every distinct (ordered pair) $D, D'\subseteq (p-1){\wt{K}_X}$ which map to $\lambda$, the composition
\[ \ooo_X(\mu)\xr{\sigma_D}\fr_*\ooo_X(\lambda)\xr{\sigma_{(p-1){\wt{K}_X}-D'}} \ooo_X(\mu)\] is the zero map. This will ensure that the splittings of $\ooo_X(D)$ and $\ooo_X(D')$ are mutually compatible, as splittings correspond to idempotents in the algebra $\hom{\ooo_X}(\ooo_X(\mu)^{\oplus m},\ooo_X(\mu)^{\oplus m})\cong \ff^{\oplus m}$ so the condition that the above composition is zero corresponds to the fact that the product of the corresponding idempotents is zero, hence are orthogonal splittings. In particular, by \cite[Theorem~1.3.8]{brion2007frobenius} (or \cite[Remark~1.4.2]{brion2007frobenius}), it suffices to show that the trace 
\[ \wh{\tau}(\sigma_{(p-1){\wt{K}_X}-D'}\circ \sigma_D)=\wh{\tau}(\sigma_{(p-1){\wt{K}_X}-D'}\cdot\sigma_D)=0.\]
Now, by \cite[Lemma~1.4.4(i)]{brion2007frobenius}, it suffices to check the trace on $\mathbf{X}_0$, the big open cell. (We refer the reader to \cite[\S6]{brion2007frobenius} or \S\ref{subsection:structure_of_X},\ref{subsection:divisors_and_line_bundles_on_X} to review the notation.)

By \cite[Proposition~6.1.7]{brion2007frobenius}, we have a $U\times U^-$-equivariant isomorphism
\[ \wh{\gamma}:U\times \aa^{n-1}\times U^-\xr{\sim} \mathbf{X}_0,\quad (u,a,v)\mapsto u\gamma(a)v^{-1},\]
where $\gamma$ is the isomorphism $\aa^{n-1}\xr{\sim}\ol{T_0}$ constructed in \cite[Lemma~6.1.6]{brion2007frobenius}. We have local coordinates $x_1,x_2,\dots,x_{n-1}$ for $\aa^{n-1}$, and coordinates $x_{ij}$ with $i<j$ for $U$, and coordinates $x_{ij}$ with $i>j$ for $U^-$ in the usual way: $x_{ij}$ denotes the $(i,j)$ entry of $U\into \msf{SL}_n\subset \msf{GL}_n$. In particular, the divisors $\mathbf{X}_k$ are precisely $V(x_k)$, as discussed in \cite[Theorem~6.1.8]{brion2007frobenius}. Now define
\[ \mathbf{m}\coloneqq \prod_{i\ne j}x_{ij}^{p-1}\prod_{k=1}^{n-1}x_k^{p-1},\] the monomial which controls the splitting of affine space, c.f. \cite[Example~1.3.1]{brion2007frobenius}. Now restricting to $T\subset \aa^{n-1}$ (i.e., when all $x_k\ne 0$), defining 
\[\mathbf{Y}_0\coloneqq \textnormal{image of }U\times T\times U^-,\]
we have the map given by
\[ \left((x_{ij})_{i<j}, (x_k)_{k\le n-1}, (x_{ij})_{i>j}\right)\mapsto \left[\begin{pmatrix}
1 & x_{12} & \dotsm & x_{1n}\\ 0 & 1 & \dotsm & x_{2n}\\
\vdots & \vdots & \ddots & \vdots\\
0 & 0 & \dotsm & 1
\end{pmatrix}
\begin{pmatrix}
1 & 0 & \dotsm & 0\\
0 & x_1 & \dotsm & 0\\
\vdots & \vdots & \ddots & 0\\
0 & 0 & \dotsm & \prod_{k=1}^{n-1}x_k
\end{pmatrix}
\begin{pmatrix}
1 &  0& \dotsm & 0\\ x_{21} & 1 & \dotsm & 0\\
\vdots & \vdots & \ddots & \vdots\\
x_{n1} & x_{n2} & \dotsm & 1
\end{pmatrix}\right],
\] sending the coordinates $x_{ij}$ and $x_{i}$ to the above product of three matrices in $\psl_n\cong \msf{PGL}_n$; the diagonal matrix in the $x_i$'s is due to the fact that the map $\gamma$ (see \S\ref{subsection:divisors_and_line_bundles_on_X}) sends $x_i=t_{i+1}/t_i$ for the diagonal matrix $\textnormal{diag}(1,t_1,t_2,\dots,t_{n-1})$. As this is completely symmetric, it suffices to work just with the entries above the diagonal. Suppose that the product of these matrices equals a matrix $(t_{ij})\in \msf{GL}_n$. Expanding, we find that, for $i<j$,
\[ t_{ij} = (x_{ij}+O)\prod_{k=1}^{j-1}x_k,\]
where $O$ denotes ``higher order terms": other monomials (in the $x_{ij}$ and $x_k$) not involving $x_{ij}$, of higher degree in the usual sense of degree of a monomial. By Gaussian elimination, we have that $\wt{\mathbf{D}_i}\cap \mathbf{Y}_0$ is given by the determinant of the $(n-i)\times (n-i)$ submatrix $(t_{r,i+r})_{r=1,2,\dots,n-i}$, in the top right, i.e. the determinant of
\[ \begin{pmatrix}
t_{1,i+1} & t_{1, i+2} & \dotsm & t_{1,n}\\
t_{2, i+1} & t_{2, i+2} & \dotsm & t_{2,n}\\
\vdots & \vdots &\ddots & \vdots\\
t_{n-i,i+1} & t_{n-i, i+2} & \dotsm & t_{n-i, n}
\end{pmatrix}\] (and similarly, $\mathbf{D}_i\cap \mathbf{Y}_0$ is the determinant of the $(n-i)\times (n-i)$ submatrix in the bottom right). In particular, using the fact that $x_k\ne 0$ for all $k=1,2,\dots,n-1$, we have the following formulas:
\begin{align*}
    \wt{\mathbf{D}_{n-1}} \cap \mathbf{Y}_0 &= V(x_{1n}),\\
    \wt{\mathbf{D}_{n-2}} \cap \mathbf{Y}_0 &= V(x_{1,n-1}x_{2,n}-x_{2,n-1}x_{1,n}),\\
    &\vdots \\
    \wt{\mathbf{D}_{k}} \cap \mathbf{Y}_0 &= V\left(\prod_{j-i=k}x_{i,j} +O\right),
\end{align*} where again, $O$ denotes ``higher order terms" (which are not important for the computation). We have similar formulas for $\mathbf{D}_i$:
\begin{align*}
    \mathbf{D}_{n-1} \cap \mathbf{Y}_0 &= V(x_{n1}),\\
    \mathbf{D}_{n-2} \cap \mathbf{Y}_0 &= V(x_{n,2}x_{n-1,1}-x_{n-1,2}x_{n,1}),\\
    &\vdots \\
    \mathbf{D}_{k} \cap \mathbf{Y}_0 &= V\left(\prod_{i-j=k}x_{i,j} +O\right),
\end{align*} using the action \[(w_0,w_0): x_{i,j}\mapsto x_{n+1-i,n+1-j}\] and the fact that $(w_0,w_0).\wt{\mathbf{D}_i}=\mathbf{D}_i$ (by definition, see \cite[Definition~6.1.10]{brion2007frobenius}). It's therefore clear that for each of the prime divisors $\delta$ described above, $\delta\cap \mathbf{Y}_0\ne \emptyset$, hence $\delta\cap\mathbf{Y}_0$ is indeed a divisor of $\mathbf{Y}_0$ (as $\delta\cap \mathbf{Y}_0$ is dense in $\delta$, hence has the same dimension, thus is codimension $1$ irreducible in $\mathbf{Y}_0$, which is dense in $X$). Since $\mathbf{Y}_0$ is dense in $\mathbf{X}_0$, which is isomorphic to an affine space (hence irreducible), it's clear that these polynomials are also the defining polynomials of each divisor in $\mathbf{X}_0$. In particular, up to scalar, we have that the restriction of the canonical section $\sigma_{\delta}$ is precisely the defining polynomial described above for each $\delta=\mathbf{D}_k,\wt{\mathbf{D}_k},\mathbf{X}_k$. By abuse of notation, for any effective divisor $D$, write $\sigma_D$ as the corresponding product of the defining polynomials. Examining the coefficient of $\mathbf{m}$ in the monomial expansion of \[ \prod_{\delta} \sigma_\delta,\] it becomes clear that the only way to obtain a monomial $\mathbf{m}$ is to choose only the ``lowest order terms" from each $\sigma_\delta$: thus, the ``higher order terms" described above via $O$ can be safely ignored. (If any higher order terms were used, then the total degree of that monomial is too large to be a scalar multiple of $\mathbf{m}$.)

Now choose two $D\sim D'$ with $D,D'\subseteq (p-1){\wt{K}_X}$. Then $D$ we have
\[ f\coloneqq \sigma_{D}\circ \sigma_{(p-1){\wt{K}_X}-D'}=\sigma_{D}\cdot\sigma_{(p-1){\wt{K}_X}-D'}=\frac{\sigma_D}{\sigma_{D'}}\cdot \mathbf{m}.\] Now for $D\ne D'$, we will have that the valuation of $D$ at some prime divisor $\delta$ is strictly greater than the valuation of $D'$ at $\delta$, and hence the exponent $e$ of the corresponding polynomial factor $\sigma_\delta$ in $f$ is between $p$ and $2p-2$, inclusive. 
In particular, if $\sigma_\delta$ is a monomial (i.e. $\delta$ is $\mathbf{D}_{n-1}$, or $\wt{\mathbf{D}_{n-1}}$, or any $\mathbf{X}_k$), then we immediately find that the coefficient of $\mathbf{m}$ in $f$ is zero, as the exponent of any variable appearing in that term is already strictly greater than $p-1$. If $\sigma_\delta$ is not a monomial, then the exponent $e$ is between $p$ and $2p-2$, but the contributing terms will not have the $p$th power of any monomial in $\sigma_\delta$ as a factor: as before, this would contribute a monomial with some variable raised to the power of $p>p-1$, hence not equal to $\mathbf{m}$. 
It follows that the contributing monomials from (the expansion of) $\sigma_\delta^e$ have coefficients which are multinomial coefficients $\binom{e}{m_1,m_2,\dots,m_r}$ where $m_1+m_2+\dots+m_r=e$ and all $m_i<p$. But then this coefficient is divisible by $p$, hence is zero. It follows that for all (ordered, but immediately implies unordered) pairs $(D,D')$ with $D\sim D'$ and $D,D'\subseteq (p-1){\wt{K}_X}$, the splittings induced by $D$ and $D'$ are indeed pairwise compatible, and therefore we obtain at least one copy of $\ooo_X(\mu)$ from $\fr_*\ooo_X(\lambda)$ for each distinct effective subdivisor of $(p-1){\wt{K}_X}$ whose class is $\lambda-p\mu$.
\end{proof}

\begin{example}
Let $G=\psl_3$ and fix $\lambda,\mu$ such that $\lambda-p\mu=6\omega_1+6\omega_2$ (here we necessarily assume $p\ge 7$). There are $460$ distinct effective sub-divisors of ${\wt{K}_X}$ whose class is $6\omega_1+6\omega_2$. Therefore, $m(\mu,\lambda)\ge 460$, i.e. the embedding $\ooo_X(\mu)^{\oplus 396}\into \fr_*\ooo_X(\lambda)$ can be splitted.
\end{example}

\begin{example}
Let $G=\psl_3$ and fix $\lambda,\mu$ such that $\lambda-p\mu=20\omega_1+22\omega_2$ (here we assume $p\ge 23$). There are $37290$ distinct effective sub-divisors of ${\wt{K}_X}$ whose class is $20\omega_1+22\omega_2$. Therefore, $m(\mu,\lambda)\ge 37290$, i.e. the embedding $\ooo_X(\mu)^{\oplus 37290}\into \fr_*\ooo_X(\lambda)$ can be splitted.
\end{example}

\begin{example}
Let $G=\psl_4$ and fix $\lambda,\mu$ such that $\lambda-p\mu=20\omega_1+21\omega_2+22\omega_3$ (here we assume $p\ge 23$). There are $14828077$ distinct effective sub-divisors of ${\wt{K}_X}$ whose class is $20\omega_1+21\omega_2+22\omega_3$. Therefore, $m(\mu,\lambda)\ge 14828077$, i.e. the embedding $\ooo_X(\mu)^{\oplus 14828077}\into \fr_*\ooo_X(\lambda)$ can be splitted.
\end{example}

\begin{example}
\label{ex:psl_3_example_local_coords}
Let $G=\psl_3$, and let us verify that Theorem~\ref{thm:psl_n_number_of_splittings} holds by explicitly checking the defining equations for each of the divisors (restricted to $\mathbf{X}_0$). Give $\aa^2$ the coordinates $x_1,x_2$, so the divisors have the form $\mathbf{X}_1 \cap \mathbf{X}_0  = V(x_1)$ and $\mathbf{X}_2 \cap \mathbf{X}_0 = V(x_2)$.
For $U$ and $U^-$, give them the coordinates
\[ 
\begin{pmatrix}
1 & a & b \\0&1&c\\0&0&1
\end{pmatrix},\quad 
\begin{pmatrix}
1& 0 & 0\\ d & 1&0\\e& f& 1
\end{pmatrix}.
\] Then noting that $G=\psl_3\cong \msf{PGL}_3$ and restricting to the (dense) open set 
$\mathbf{Y}_0\coloneqq U\times T\times U^-\subset \mathbf{X}_0$, we can expand to find the image inside $\mathbf{X}_0$:
\[ \begin{pmatrix}
1 & a & b \\0&1&c\\0&0&1
\end{pmatrix}
\begin{pmatrix} 1 & 0 & 0\\ 0 & x_1 & 0\\ 0 & 0 & x_1x_2\end{pmatrix}
\begin{pmatrix}
1& 0 & 0\\ d & 1&0\\e& f& 1
\end{pmatrix}
=\begin{pmatrix}
1 +adx_1+bex_1x_2 & ax_1+bfx_1x_2 & bx_1x_2\\ dx_1+ecx_1x_2 & x_1+fcx_1x_2 & cx_1x_2\\ ex_1x_2 & fx_1x_2 & x_1x_2
\end{pmatrix}.\]
We then consider the class of this matrix in $\psl_3\cong\msf{PGL}_3$. Recalling that $x_1,x_2\ne 0$ in $\mathbf{Y}_0$, we have that
\begin{align*}
    \mathbf{D}_1 \cap \mathbf{Y}_0 &= V(x_1^2x_2(df-e))=V(df-e),\\
    \wt{\mathbf{D}_1}\cap \mathbf{Y}_0 &= V(x_1^2x_2(ac-b))=V(ac-b),\\
    \mathbf{D}_2 \cap \mathbf{Y}_0 &= V(x_1x_2e)=V(e),\\
    \wt{\mathbf{D}_2}\cap \mathbf{Y}_0 &= V(x_1x_2b)=V(b).
\end{align*}
Since $\mathbf{Y}_0\subset \mathbf{X}_0$ is dense, it follows that each of these divisors are cut out by the same equations in $\mathbf{X}_0$.

Now we can explicitly check that the section $\sigma$ corresponds to
$\left(x_1x_2 b e (df-e)(ac-b)\right)^{p-1}$,
which we easily verify has coefficient $1$ for the monomial
\[ \mathbf{m}\coloneqq (x_1x_2 abcdef)^{p-1}\]
and has coefficient $0$ for all other monomials whose exponents differ from the aforementioned monomial by a multiple of $p$. Now for any $D\sim D'\subset (p-1){\wt{K}_X}$, corresponding to polynomials
\[ \sigma_D=x_1^{m_1}x_2^{m_2}b^{m_b}e^{m_e}(df-e)^{m_d}(ac-b)^{m_a},\quad \sigma_{D'}=x_1^{n_1}x_2^{n_2}b^{n_b}e^{n_e}(df-e)^{n_d}(ac-b)^{n_a}\] for $0\le m_i,n_i\le p-1$, the product $\sigma_D\cdot \sigma_{(p-1){\wt{K}_X}-D'}$ corresponds to the coefficient of $\mathbf{m}$ in the polynomial
\[ \sigma_D\cdot \frac{\mathbf{m}}{\sigma_{D'}}=x_1^{p-1+m_1-n_1}x_2^{p-1+m_2-n_2}b^{p-1+m_b-n_b}e^{p-1+m_e-n_e}(df-e)^{p-1+m_d-n_d}(ac-b)^{p-1+m_a-n_a}.\] But since $D\ne D'$ and $D\sim D'$, there must exist $m_i>n_i$, so then we have either monomial to a power at least $p$, for example $b^{p+\epsilon}$, which implies that the coefficient of $\mathbf{m}$ is $0$, or one of $(df-e)$ or $(ac-b)$ has exponent at least $p$. But in these cases, say $m_a-n_a=k$, then in the term $(ac-b)^{p+k-1}$
we require the term $ac$ exactly $p-1$ times due to the fact that there are no other appearances of the local coordinate $a$. It follows that the coefficient of $\mathbf{m}$ is multiplied by
$\binom{p-1+k}{p-1}=\frac{(p-1+k)!}{(p-1)!k!}=p\cdot -=0$, and thus the coefficient of $\mathbf{m}$ is zero. This implies that for any $D\sim D'$, the product $\sigma_D\cdot \sigma_{(p-1){\wt{K}_X}-D'}$ has trace $0$, hence the splittings corresponding to the ordered pair $(D,D')$ are mutually orthogonal. Since this holds for all $D,D'$, it follows that each distinct effective subdivisor $D\subseteq (p-1){\wt{K}_X}$ whose class is $\lambda-p\mu$, yields a distinct copy of $\ooo_X(\mu)$ as a summand inside $\fr_*\ooo_X(\lambda)$.
\end{example}
We pose two conjectures: that we should have analogous statements for any $G$, and that this lower bound is an equality.

\begin{conjecture}
Do the results of Theorem~\ref{thm:psl_n_number_of_splittings} hold for any (semisimple adjoint) $G$?
\end{conjecture}

\begin{conjecture}
Let $L:\textnormal{Div}(X)\to \textnormal{Pic}(X)$ be the natural map. For each $\lambda,\mu$ such that $\lambda-p\mu$ can be written as $\sum_{i=1}^{\ell}(a_i\omega_i+b_i\alpha_i)$, the multiplicity $m(\mu,\lambda)$ of $\ooo_X(\mu)$ in $\fr_*\ooo_X(\lambda)$ (as direct summands) is at least $S(\lambda-p\mu)$, due to Theorem~\ref{thm:psl_n_number_of_splittings}. Is it exactly equal?
\end{conjecture}

\section{Vector subbundles via irreducible representations}
\label{section:vector_subbundles_via_irreps}

Recall that $\wt{G}$ is the simply connected cover of a connected semisimple adjoint algebraic group $G$ over $\ff$, and $X$ denotes the wonderful compactification of $G$. In this section, we can use the $G_1\times G_1$-action on $\fr_*\lll$ to obtain embeddings of vector bundles.

\begin{theorem}
\label{thm:embedding_of_L_otimes_L}
Let $\lambda$ be a dominant weight minimal in its linkage class, and let $L_\lambda$ be the irreducible $G_1\times G_1$-representation corresponding to $\lambda$. Then there exists an embedding of $G_1\times G_1$-equivariant vector bundles $\psi:L_\lambda\otimes_\ff L_{-w_0\lambda}\otimes_\ff\ooo_X\into \fr_*\ooo_X(\lambda)$.
\end{theorem}
\begin{remark}
To be minimal in its linkage class, it is sufficient for $\lambda$ to lie in the fundamental alcove, despite the fact that there will in general be other weights linked to $\lambda$ in the fundamental alcove (but they will all be minimal), see \cite[5.2]{verma1975role}. Although $(p-1)\rho$ does not lie in the fundamental alcove, it has no linked weights less than it which are dominant: the next smallest is $-\rho$, which is not dominant.
\end{remark}
\begin{proof}
First, we show that 
\[L_\lambda\otimes_\ff L_{-w_0\lambda}\into F_{\le \lambda}=\Gamma(X,\ooo_X(\lambda)).\]
We first apply the central idempotent $\pi_\lambda$ to $F_{\le \lambda}$ (see \S\ref{subsection:representation_theory_of_G_1} and \S\ref{subsection:appendix:theory_of_idempotents} for a discussion on $\pi_\lambda$). By \cite[Lemma~II.4.15]{jantzen2003representations}, $\pi_\lambda F_{\le \lambda}$ has a submodule isomorphic to $M_\lambda\otimes M_{-w_0\lambda}$, since $\lambda$ is minimal in its linkage class. Now it suffices to check that $M_\lambda\otimes M_{-w_0\lambda}= L_\lambda\otimes L_{-w_0\lambda}$ for $\lambda$ minimal in its linkage class, which is true by \cite[5.2]{verma1975role} (also stated in \cite[\S4.1]{humphreys2006ordinary}). Now let $\psi$ be the map induced by the inclusions
\[ L_\lambda\otimes L_{-w_0\lambda}\into \pi_\lambda F_{\le \lambda}\into F_{\le \lambda},\]
corresponding to the desired map 
\[ L_\lambda\otimes_{\ff}L_{-w_0 \lambda}\otimes_{\ff}\ooo_X\to\fr_*\ooo_X(\lambda).\]
By Lemma~\ref{lemma_haboush_surjection_implies_nonzero}, it suffices to check that we have a surjection on the map induced by adjunction
\[ \wt{\psi}:L_\lambda\otimes L_{-w_0\lambda}\otimes \ooo_X\to \ooo_X(\lambda),\]
which is a $\wt{G}\times \wt{G}$-equivariant map. Letting $Y=\wt{G}/\wt{B}\times \wt{G}/\wt{B}$ be the unique closed orbit, we have that $\wt{\psi}|_Y$ restricts to a nonzero map at every fiber of $Y$, as $\ooo_X(\lambda)|_Y\cong \ooo_Y(\lambda)\boxtimes \ooo_Y(-w_0\lambda)$, and then $\wt{\psi}|_Y$ is induced by the map
\[ L_\lambda\otimes L_{-w_0\lambda}\xr{\sim} M_\lambda\otimes M_{-w_0\lambda}\] on global sections. If $\wt{\psi}$ were zero at the fiber of some $x\in X$, then 
\[\wt{\psi}|_{(\wt{G}\times \wt{G}).x}=0\implies \wt{\psi}|_{\ol{(\wt{G}\times \wt{G}).x}}=0\] by the $\wt{G}\times \wt{G}$-equivariant structure of $\wt{\psi}$. Since $\ol{(\wt{G}\times \wt{G}).x}\supset Y$, this contradicts the fact that $\wt{\psi}|_Y$ restricts to a nonzero map, hence $\wt{\psi}$ is nonzero at every fiber. But since the target of $\wt{\psi}$ is a line bundle, it must be surjective at every fiber. It follows that $\psi$ is nonzero at every fiber, and being a $G_1\times G_1$-equivariant map, must be injective at every fiber, hence a $G_1\times G_1$-equivariant embedding of vector bundles.
\end{proof}
\begin{remark}
For any dominant $\lambda\in \Lambda$, there exists some $N$ such that $\lambda$ satisfies Theorem~\ref{thm:embedding_of_L_otimes_L} for all $p\ge N$: it only remains to see that for all $p\gggg 0$, $\lambda$ will lie in the fundamental alcove. (For $G=\psl_n$, writing $\lambda=\sum_{i=1}^{\ell} a_i\omega_i$ with $a_i\ge 0$ and $\omega_i$ the fundamental weights, it suffices to take $N=\sum_i a_i+n-1$.)
\end{remark}
\begin{remark}
This proof is an adaptation and slight generalization of \cite[Theorem~2.1]{haboush1980short}.
\end{remark}

\begin{corollary}
\label{cor:embedding_of_St_otimes_St}
Let $\psi:\st\otimes_\ff  \st \otimes_\ff\ooo_X\to \fr_*\ooo_X((p-1)\rho)$ be the map induced by $\st\otimes_\ff \st \into F_{\le (p-1)\rho}=\Gamma(X,\ooo_X((p-1)\rho))$. Then $\psi$ is an embedding of $G_1\times G_1$-equivariant vector bundles.
\end{corollary}
\begin{proof}
Apply Theorem~\ref{thm:embedding_of_L_otimes_L} to $\lambda=(p-1)\rho$, which is clearly minimal in its linkage class.
\end{proof}
\begin{remark}
In Theorem~\ref{thm:determine_line_bundles_for_p-1_rho_block}, we show that this embedding can be splitted.
\end{remark}

\section{Splitting via idempotents}
\label{section:splitting_via_idempotents}

Let $\lll$ be a line bundle on $X$. Our goal in this section is to decompose $\fr_* \lll$ as a direct sum of vector (sub)bundles, by multiplying (on the left and right) by idempotents in $\uzero$, whose action arises from the $G_1\times G_1$-action on $\lll$. We will prove Theorem~\ref{thmC:decomp_of_frob_push_as_vector_subbundles} and Theorem~\ref{thmD:st_st_component_description}.

\subsection{The main strategy}
\label{subsection:the_main_strategy}
Consider a line bundle $\lll$ on $X$. By \cite[Theorem~7.2]{dolgachev2003lectures}, we have a $\wt{G}\times \wt{G}$-equivariant structure on $\lll$. The crucial point, which we will use \textbf{constantly}, is that we can use idempotents in $\uuu_0(\mf{g})$ to split any $\fr_*\lll$ into a direct sum of vector bundles, whose ranks are controlled by the action of idempotents on $\uuu_0(\mf{g})$.
\begin{proposition}
\label{prop:independence_of_line_bundles}
Let $\mathscr{S}=\{e_i\}_I$ be any system of pairwise orthogonal idempotents which sum to $1$, in $\uuu_0(\mf{g})$. Then we obtain the decomposition
\[ \fr_*\lll= \bigoplus_{(e_i,e_j)\in \mathscr{S}\times \mathscr{S}} e_i  (\fr_*\lll) e_j,\]
where the $e_i (\fr_* \lll) e_j$ are vector bundles. Furthermore,
\[\textnormal{rk }e_i  (\fr_*\lll) e_j=\dim_\ff e_i\uuu_0(\mf{g})e_j,\] and in particular are independent of the choice of line bundle $\lll$.
\end{proposition}
\begin{proof}
First, we note that any such collection of idempotents determines a splitting of $\fr_*\lll$ into a direct sum. We know that $\fr_*\lll$ is locally free. Direct summands of locally free sheaves are locally summands of a free module, i.e. projective modules, which are locally free, hence direct summands of locally free sheaves are again locally free.

Now notice that the ranks of summands of the decomposition of $\fr_*\lll$ can be deduced from the dimensions of summands of the decomposition of any fiber of $\fr_*\lll$ induced by the $G_1\times G_1$-action. In particular, we can consider the fiber at $1\in G\subset X$. Since $\fr$ is a flat finite morphism, we have (since $\fr$ is finite, therefore affine) the following natural isomorphism of $G_1\times G_1$-modules:
\[ (\fr_*\lll)_1\cong H^0(\fr^{-1}(1),\lll)=\Gamma(G_1,\lll |_{G_1}). \]
Therefore, the ranks of the summands induced by the $G_1\times G_1$ action on $\fr_*\lll$ can be deduced from the dimensions of the vector spaces of the summands induced by the $G_1\times G_1$-action on $\Gamma(G_1,\lll|_{G_1})$. Since the problem is local, we may restrict to $G\subset X$ and assume $\lll$ is a line bundle on $G$. Consider the following commutative diagram:
\[
\begin{tikzcd}
G_1\arrow{rd}{\iota_G} \arrow{d}{\iota_{\wt{G}}} &\\
\wt{G}\arrow{r}{\pi}&G.\\
\end{tikzcd}
\]
Since $\wt{G}$ is simply connected, all line bundles on $\wt{G}$ are isomorphic to $\ooo_{\wt{G}}$ as $\wt{G}\times \wt{G}$-equivariant sheaves (with the standard action), and in particular must have the standard $G_1\times G_1$-action. Therefore, $\pi^*\lll\cong \ooo_{\wt{G}}$ as $G_1\times G_1$-equivariant sheaves (with the standard action). Now note that
\[ \lll|_{G_1} \cong \iota^*_{G}\lll\cong (\iota_{\wt{G}}\circ\pi)^*\lll\cong \iota_{\wt{G}}^*(\pi^*\lll)\cong \iota_{\wt{G}}^*\ooo_{\wt{G}}\cong \ooo_{G_1}.\]
In particular, we have that $\iota_G^*\lll \cong \iota_{\wt{G}}^*\pi^*\lll\cong \ooo_{G_1}$ as $G_1\times G_1$-equivariant sheaves. Therefore, it suffices to describe the $G_1\times G_1$-action on $\ooo_{G_1}$, which is equivalent to describing the $G_1\times G_1$-action on $\ff[G_1]\cong \uuu_0(\mf{g})^*$. The discussion in \cite[I.8.6]{jantzen2003representations} shows that the left and right (standard) representations of $G_1$ on $\ff[G_1]\cong \uuu_0(\mf{g})^*$ induce the standard left and right representations of $G_1$ on $\ff[G_1]^*\cong \uuu_0(\mf{g})$. Hence this is equivalent to the action of $\uuu_0(\mf{g})\otimes \uuu_0(\mf{g})$ on $\uuu_0(\mf{g})^*$. By Lemma~\ref{lemma:u_0*=u_0}, this is equivalent to studying $\uuu_0(\mf{g})$ under the usual $\uuu_0(\mf{g})\otimes \uuu_0(\mf{g})$-action. It follows that the rank of the corresponding vector bundles induced by left and right action by idempotents is exactly the dimension of $e_i\uuu_0(\mf{g})e_j$. 
\end{proof}
\begin{remark}
Proposition~\ref{prop:independence_of_line_bundles} applies to any smooth $\wt{G}\times \wt{G}$-variety $X$ which contains $G$ (having the standard action of $\wt{G}\times \wt{G}$) as an open subset - in particular, as in the proof, it works for $G$. More generally, the result also holds for wonderful varieties, which include wonderful compactifications; see \cite{pezzini2018lectures} for an introduction to wonderful varieties.
\end{remark}

We first analyze the case where the idempotents are central.

\subsection{Central idempotents}
\label{subsection:Central_idempotents}
Recall the linkage principle~\ref{linkage_principle_U(g)}: we have the decomposition
\[\msf{Rep}~G_1 =\bigoplus_{\lambda\in \Lambda_p/(W,\cdot)}\msf{Rep}_\lambda(G_1),\] where we identify $\Lambda_p$ with $\Lambda/p\Lambda$ (and thus the action of $(W,\cdot)$ is identified with the action of $(W,\cdot)$ on $\Lambda/p\Lambda$). Corresponding to this, we have the existence of central idempotents $\{\pi_\lambda\}_{\lambda\in \Lambda_p /(W,\cdot)}$ inside the reduced enveloping algebra $\mcal{U}_0(\mf{g})$. As a result, we have the following decomposition:

\begin{lemma}
\label{lemma:central_idempotents_decomp_into_vector_bundles}
For any invertible sheaf $\lll$ on $X$, we have 
\[ \fr_* \lll=\bigoplus_{\lambda\in\Lambda_p /(W,\cdot)} \eee_{\lambda,\lll}\] as $(\ooo_X, G_1\times G_1)$-modules, where $\eee_{\lambda,\lll}=\pi_\lambda \fr_*\lll$ are vector subbundles of $\fr_*\lll$, and their ranks are precisely the dimension of $A_{\lambda}$ in the block decomposition 
\[\uuu_0(\mf{g})=\bigoplus_{\lambda\in \Lambda_p /(W,\cdot)}A_\lambda.\]
\end{lemma}
\begin{proof}
Specialize Proposition~\ref{prop:independence_of_line_bundles} to the collection $\{\pi_\lambda\}$ of central idempotents. This gives the decomposition as $\ooo_X$-modules; the fact that $G_1\times G_1$-action is preserved is due to the fact that all idempotents are central (see Remark~\ref{remark:central_idempotents_preserves_module}).
\end{proof}

\begin{notation}
Let $W(\lambda)$ denote the stabilizer of $\lambda\in\Lambda/p\Lambda $ in $(W,\cdot)$.
\end{notation}
\begin{proposition}
\label{prop:size_of_blocks}
Suppose $\mf{g}$ is simple and of classical type. We have that
\[ \uuu_0(\mf{g})\cong \bigoplus_{\lambda\in \Lambda_p /(W,\cdot)}A_\lambda,\]
where 
\[\dim A_{\lambda}=|W/W(\lambda)|\cdot p^{2\dim \wt{G}/\wt{B}}.\]
\end{proposition}
\begin{proof}
By \cite[\S5, page~41]{feldvoss1996homological} and \cite{haboush1980short}, 
\[
\frac{\dim A_\lambda}{|W\cdot \lambda|}=\frac{\dim A_{(p-1)\rho}}{|W\cdot (p-1)\rho|}=p^{2\dim \wt{G}/\wt{B}}.
\]
On the other hand, $|W\cdot \lambda|=|W/W(\lambda)|$.
\end{proof}
\begin{corollary}
\label{cor:sl_n_blocks_dimensions}
Let $G=\psl_n$ and $p\nmid n$. Then $\Lambda\cong \zz^n/\zz\cdot (1,1,\dots,1)$ (in the basis of $\varepsilon_i$), and hence $\Lambda/p\Lambda \cong \ff_p^n/\ff_p\cdot (1,1,\dots,1)$, both with the induced action of $W=S_n$. Then for $\lambda\in\ff_p^n$ for which $\lambda+\rho$ contains $k$ distinct entries with multiplicities $n_1,\dots,n_k$, 
\[\dim A_\lambda=\frac{n!}{n_1!n_2!\dotsm n_k!}p^{2\dim \wt{G}/\wt{B}}=\frac{n!}{n_1!n_2!\dots n_k!}p^{n(n-1)}.\]
\end{corollary}
\begin{proof}
Applying Proposition~\ref{prop:size_of_blocks} to $G=\psl_n$, it suffices to compute $W(\lambda)$. We may assume $\lambda+\rho=(a_1,a_2,\dots,a_k)$ where the $a_i$ are pairwise distinct and $a_i$ appears with multiplicity $n_i$, as it suffices to simply apply a permutation $\sigma\in W\cong S_n$ to $\lambda$ to reach this form, and such stabilizer subgroups are conjugate to each other, hence of the same size. Since shifting by $\rho$ essentially does not change the problem, We'll identify the $(W,\cdot)$ on $\Lambda/p\Lambda$ with the $W$ action on $\Lambda/p\Lambda$ by shifting by $\rho$ for the rest of this proof. Now $\sigma\in W$ stabilizes $\lambda$ iff $\sigma\lambda-\lambda=c\cdot (1,1,\dots,1)$. If $c\ne 0$, then this implies that we have $a_1+c$ is equal to some other $a_i$, let's say $a_2$. Similarly, $a_2+c$ is equal to some other $a_j$, which cannot be $a_1$: otherwise, $a_1+2c=a_1\implies c=0$. Continuing, we find that we can reorder the $a_i$ such that $a_i=a_{i-1}+c$, hence we find that $a_1+kc=a_1\implies kc=0$, and further that all $n_i$ are equal. But this is impossible, as $c\ne 0$ and $p\nmid n\implies p\nmid k\implies k\ne 0$. It follows that $c=0$, and thus $\sigma$ stabilizes $\lambda$ iff $\sigma\lambda=\lambda$ as vectors in $\ff_p^n$.
\end{proof}

\begin{remark}
Note that the condition that $p\nmid n$ is crucial. For $p=n$, consider the weight $\lambda=(0,1,2,\dots,n-1)=(0,1,2,\dots,p-1)$. Then consider any cyclic shift of order $p$: for example, shifting all by one to the right gives $\lambda'=(1,2,\dots,p-1,0)$. But $\lambda-\lambda'=(1,1,\dots,1)$, which implies that this shift is indeed in the stabilizer, whereas if $p\nmid n$, such a $\lambda$ would have trivial stabilizer. 
In fact, the only time such errors occur are when $n_1=n_2=\dots=n_k$ and $k=p$ (i.e., we require $p$ to be a ``bad prime," in the terminology of \cite{brown2001ramification}): $\{a_1,\dots,a_p\}=\{0,1,\dots,p-1\}$, each occuring with multiplicity $n/p$.

However, even when $p|n$, we still obtain lower bounds on the stabilizer of $\lambda$, which translates into upper bounds on the dimension of $A_\lambda$.
\end{remark}

For convenience, we'll make the following convention.
\begin{definition}
\label{def:type_of_vector}
For a vector $v\in \ff_p^n$, call its \textbf{type} $(n_1,n_2,\dots,n_k)$, where $v$ has $n_i$ entries which are equal to some $a_i\in \ff_p$, and the $a_i$ are pairwise distinct.
\end{definition}

Then, Corollary~\ref{cor:sl_n_blocks_dimensions} is reworded into the following:
\begin{corollary}
\label{cor:sl_n_blocks_dimensions_WITH_TYPES}
Let $G=\psl_n$ and $p\nmid n$. For $\lambda\in \Lambda$, write $\lambda=\sum_i b_i\varepsilon_i$, and suppose $(b_1,\dots,b_n)$ has type $(n_1,\dots,n_k)$ (as in Definition~\ref{def:type_of_vector}). Then for $\lambda\in \Lambda_p $, we have
\[\dim A_\lambda = \frac{n!}{n_1!n_2!\dotsm n_k!}p^{n(n-1)}.\]
\end{corollary}
\begin{remark}
Alternatively, suppose $\lambda=\sum_{i=1}^{n-1} c_i \omega_i$. Then as $\omega_i=\sum_{j\le i}\varepsilon_j$, and we have the $\rho$-shifted action of $W$, we obtain the vector \[(c_1+\dots+c_{n-1}+n-1,c_2+\dots+c_{n-1}+n-2,\dots,c_{n-1}+1,0),\] written in the basis of $\varepsilon_j$.
\end{remark}

\begin{proposition}
\label{prop:splitting_of_Os_from_p-1_rho}
We have that $\st\otimes \st\otimes \ooo_X$ is a direct summand inside of $\fr_*\ooo_X((p-1)\rho)$. Furthermore, $\pi_{(p-1)\rho}\fr_*\ooo_X((p-1)\rho)\cong \st\otimes \st\otimes \ooo_X$ (in the notation of Lemma~\ref{lemma:central_idempotents_decomp_into_vector_bundles}).
\end{proposition}
\begin{proof}
The block of $\uuu_0(\mf{g})$ containing the irreducible representation $L_{(p-1)\rho}=\st$ contains no other irreducible representations. From \cite[\S3.19]{brown2001ramification}, this block is isomorphic to $\textnormal{End}(\st)$, and hence has dimension $\dim \st\otimes \st$. On the other hand, by Corollary~\ref{cor:embedding_of_St_otimes_St}, we have that $\st\otimes \st\otimes \ooo_X\subseteq \pi_{(p-1)\rho}\fr_*\ooo_X((p-1)\rho)$. Now by comparing ranks, we find that we have an embedding of a vector bundle of rank $\dim \st\otimes \st$ into another of the same rank, hence they must be equal.
\end{proof}
\begin{corollary}
The line bundle $\ooo_X(\lambda)$ has multiplicity at least $\dim(\st\otimes \st)=(\dim \st)^2=p^{2\dim \wt{G}/\wt{B}}$ as a direct summand of $\fr_*\ooo_X(p\lambda+(p-1)\rho)$.
\end{corollary}
\begin{proof}
The case of $\lambda=0$ is immediately deduced from the previous proposition. The general case follows from projection formula.
\end{proof}
\begin{remark}
In Theorem~\ref{thm:determine_line_bundles_for_p-1_rho_block}, we show that for \textit{any} line bundle, we have a similar behavior where the block associated to $(p-1)\rho$ splits into line bundles:
\[ \pi_{(p-1)\rho}\fr_*\lll\cong \st\otimes\st\otimes \lll'.\]
Furthermore, we determine $\lll'$ based on $\lll$. In view of this, Proposition~\ref{prop:splitting_of_Os_from_p-1_rho} is an immediate corollary.
\end{remark}

\subsection{Primitive idempotents}
\label{subsection:primitive_idempotents}
Now, our goal is to apply Proposition~\ref{prop:independence_of_line_bundles} to a system of primitive idempotents. Unlike in \S\ref{subsection:Central_idempotents}, these idempotents will no longer be central, as we decompose the central idempotents into pairwise orthogonal primitive idempotents. To start, we'll construct the idempotents.

\begin{lemma}
\label{lemma:primitive_idempotents_U_0}
There exist pairwise orthogonal primitive idempotents $\{e_\lambda^i\}$ of $\uuu_0(\mf{g})$ which sum to $1\in\uuu_0(\mf{g})$, indexed by $(\lambda,i)$ where $\lambda\in\Lambda_p $ and $i=1,2,\dots,\dim L_\lambda$, where $L_\lambda$ is the irreducible $G_1$-representation of weight $\lambda$. The idempotents corresponding to $\lambda$ are associated to the matrix algebra $\textnormal{End}(L_\lambda)$ in $\uuu_0(\mf{g})/\textnormal{rad}$.
\end{lemma}
\begin{proof}
We apply Lemma~\ref{lemma:pairwise_orthogonal_primitive_idempotents} with $A=\uzero$. In particular, we have $\ll=\Lambda_p$, and the irreducibles are precisely $L_\lambda$.
\end{proof}

\begin{notation}
Denote by $\mathscr{E}$ the set of idempotents constructed in Lemma~\ref{lemma:primitive_idempotents_U_0}, namely the set of $e_\lambda^i$ for $\lambda\in\Lambda_p $ and $i=1,2,\dots,\dim L_\lambda$.
\end{notation}

\begin{notation}
For $\lambda\in\Lambda_p $, we denote $P_\lambda\coloneqq P_{L_\lambda}$, the projective cover of $L_\lambda$.
\end{notation}

The main point of our primitive idempotents is to decompose $\fr_*\lll$, or equivalently $\uuu_0(\mf{g})$:
\begin{proposition}
\label{prop:decomp_of_fr_*L_by_two_sided_idempotents}
Let $\lll$ be any line bundle on $X$. We have the decomposition of $\fr_*\lll$ into vector bundles
\[ \fr_*\lll= \bigoplus_{(e_\mu^i, e_\lambda^j)\in\mathscr{E}\times \mathscr{E}} e_\mu^i(\fr_*\lll)  e_\lambda^j,\] where 
\[\textnormal{rk }e_\mu^i(\fr_*\lll)e_\lambda^j=  [L_\mu: P_\lambda].\]
\end{proposition}
\begin{proof}
The decomposition of $\fr_*\lll$ follows immediately from applying Proposition~\ref{prop:independence_of_line_bundles} to the set of idempotents $\mathscr{E}$. The statement of the rank follows from applying Proposition~\ref{prop:appendix:dimension_of_terms} for $A=\uzero$ and $\ll=\mathscr{E}$.
\end{proof}

We now turn our attention to computing the ranks of the vector bundles, or equivalently, the values $[L_\mu:P_\lambda]$.

\subsection{Computation of dimensions}
\label{subsection:computation_of_dimensions}

Our goal now is to compute the multiplicities $[L_\mu:P_\lambda]$.

\begin{notation}
Let $\Delta_\lambda\coloneqq \uzero\otimes_{\uuu_0(\mf{b})} \ff_{\lambda}$ denote the baby Verma module associated to weight $\lambda$, as described in \cite{ciappara2021lectures}. (They are also described in \cite{humphreys1971modular} and \cite{humphreys2006ordinary}, but they are denoted by $Z_\lambda$ in these references.)
\end{notation}

\begin{definition}
Let $K_0(\msf{Rep}~\uzero)$ denote the Grothendieck group of the category $\msf{Rep}~\uzero$.
\end{definition}

\begin{definition}
\label{definition:d_lambda}
Following \cite{humphreys2006ordinary}, define $d_\lambda$ to be the multiplicity of $L_\lambda$ as a composition factor of $\Delta_\lambda$ in $K_0(\msf{Rep}~\uzero)$, i.e. 
\[d_\lambda \coloneqq [L_\lambda : \Delta_\lambda].\]
\end{definition}

\begin{notation}
\label{notation:a_lambda}
We denote 
\[a_\lambda\coloneqq |W|/|W(\lambda)|,\]
the size of the linkage class of $\lambda$.
\end{notation}
\begin{notation}
We write $\mu\sim \lambda$ to express that $\mu,\lambda$ belong to the same linkage class, i.e. that there exists some $w\in W$ such that $w\cdot \mu=\lambda$.
\end{notation}

Now we are able to express the values $[L_\mu:P_\lambda]$, and arrive at our main result concerning the decomposition of $\fr_*\lll$ into a direct sum of vector bundles, and prove Theorem~\ref{thmC:decomp_of_frob_push_as_vector_subbundles}.
\begin{theorem}
\label{thm:decomp_line_bundle_into_vector_bundles}
Let $\lll$ be a line bundle on $X$. We have the following decomposition of $\fr_*\lll$ into vector subbundles:
\[ \fr_*\lll=\bigoplus_{(e_\mu^i,e_\lambda^j)\in\mathscr{E}\times \mathscr{E}} e_\mu^i (\fr_*\lll) e_\lambda^j,\]
where 
\[ \textnormal{rk }e_\mu^i (\fr_*\lll) e_\lambda^j=
\begin{cases}
a_\lambda \cdot d_\lambda \cdot d_\mu & \mu\sim \lambda,\\
0 & \mu\not\sim \lambda.
\end{cases}\]
In other words, we have an abstract decomposition

\begin{equation}
\label{eq:decomp_frob_push_into_vector_subbundles}
\fr_*\lll \cong\bigoplus_{\lambda\in \Lambda_p }\bigoplus_{\mu\sim \lambda} \bigoplus_{\substack{1\le i\le \dim L_\mu\\ 1\le j\le \dim L_\lambda}}\fff_{\mu,\lambda}^{i,j}, \tag{$\star$}
\end{equation}

where $\fff_{\mu,\lambda}^{i,j}$ is a vector bundle of rank $a_\lambda d_\lambda d_\mu$, and these can be chosen to be vector subbundles (so that the isomorphism is in fact an equality).

In particular, the ranks of the summands are uniformly bounded by $(\max_{\lambda} d_\lambda)^2\cdot |W|$, and for $p\gggg 0$ this is independent of $p$ (see Theorem~\ref{thm:behavior_of_d_lambda}).
\end{theorem}
We remark that for fixed $\lambda,\mu\in\Lambda$, the values of $a_\lambda,d_\lambda,d_\mu$ are independent of $p$ for $p\gggg 0$ (see Theorem~\ref{thm:behavior_of_d_lambda} and \cite{andersen1994representations}), exhibiting a notion of ``$p$-uniformity," which is discussed in \cite[\S1.2, page~3]{raedschelders2019frobenius}.
\begin{proof}
From Proposition~\ref{prop:decomp_of_fr_*L_by_two_sided_idempotents}, we have the decomposition of $\fr_*\lll$ into the above vector subbundles, and we have that
\[ \textnormal{rk }e_\mu^i (\fr_*\lll) e_\lambda^j=[L_\mu:P_\lambda].\]
From \cite[Theorem,~\S5.4]{humphreys2006ordinary}, we have the following equalities in $K_0(\msf{Rep}~\uzero)$:
\[ [P_\lambda] = \sum_{\mu\sim \lambda}d_\lambda [\Delta_\mu] = a_\lambda d_\lambda [\Delta_\lambda]=a_\lambda d_\lambda \sum_{\mu\sim \lambda}d_\mu [L_\mu].\]
From this, we note that for $\mu\not\sim \lambda$, then $[P_\lambda]$ does not have $[L_\mu]$ as a composition factor, while for $\mu\sim \lambda$, each $[L_\mu]$ will appear with multiplicity $a_\lambda d_\lambda d_\mu$. Therefore we find that
\[
\dim e_\mu^i \uzero e_\lambda^j = [L_\mu:P_\lambda] = 
\begin{cases}
a_\lambda \cdot d_\lambda \cdot d_\mu & \mu\sim \lambda,\\
0 & \mu\not\sim \lambda,
\end{cases}
\]
and the theorem follows.
\end{proof}

Note that the ranks are often \textit{not} $1$; in fact, as we will see, the only case is when $\mu=\lambda=(p-1)\rho$. However, \cite[Theorem~1]{achinger2015characterization} shows that smooth toric varieties are completely characterized, among all projective connected schemes (over an algebraically closed field of characteristic $p$), by the property that the Frobenius pushforward of any line bundle decomposes into a direct sum of line bundles. In particular, for $X$ the wonderful compactification and $\lll$ a line bundle on $X$, then $\fr_*\lll$ does \textit{not} decompose into a direct sum of line bundles (except for the one case where $X$ is a toric variety, namely $G=\psl_2$).
\begin{remark}
\label{remark:compute_dim_L_lambda}
In general, computing $a_\lambda$ is straightforward and there exist explicit formulas. It remains to compute $d_\lambda$ and $\dim L_\lambda$. The values of $d_\lambda$ are of interest, but not completely understood in full generality, as discussed in \cite{humphreys2006ordinary}. See Appendix~\ref{section:appendix:decomposition_numbers} for a complete description of the ranks of the summands in certain small cases. We also discuss general properties of the decomposition numbers $d_\lambda$, and describe the general algorithm to compute the $d_\lambda$ in \S\ref{subsection:compute_d_lambda} for $p\gggg 0$. The dimensions of $L_\lambda$ (which tell us how many of each type of summand there are) can be computed generally in a very similar fashion (also for $p\gggg 0$); see Remark~\ref{KL_dim}.
\end{remark}

Of particular interest is determining which summands of $(\star)$ are line bundles.
\begin{corollary}
\label{cor:steinberg_splits_into_line_bundles_general}
When $\lambda=(p-1)\rho$, we have that the $\fff_{\mu,\lambda}^{i,j}$ (in $(\star)$, in the notation of Theorem~\ref{thm:decomp_line_bundle_into_vector_bundles}) are line bundles whenever $\mu\sim \lambda$. Furthermore, every $\fff_{\mu,\lambda}^{i,j}$ which is a line bundle satisfies $\mu\sim\lambda=(p-1)\rho$.
\end{corollary}
\begin{proof}
For $\lambda=(p-1)\rho$, there are no other $\mu$ in its linkage class: $\mu\sim\lambda\implies \mu=\lambda$. It follows that we have 
\[(\dim L_\lambda)(\dim L_\lambda)=(\dim \st)(\dim \st)=\dim \st\otimes \st\] sheaves $\fff_{\mu,\lambda}^{i,j}$ corresponding to the right idempotent $\lambda=(p-1)\rho$, and their ranks are 
\[a_{(p-1)\rho}d_{(p-1)\rho}^2=d_{(p-1)\rho}^2.\]
But by \cite[\S9.2]{humphreys2006ordinary} and the fact that $(p-1)\rho$ is maximal in its linkage class (tautologically, since it is the only member of its linkage class), we have $d_{(p-1)\rho}=1$. It follows that 
\[\textnormal{rk }\fff_{\mu,\lambda}^{i,j}=1\text{ for }\mu\sim \lambda=(p-1)\rho.\]

The converse follows immediately from noting that if the rank is $1$, then $a_\lambda=1$, and the only weight in $\Lambda_p$ which is stable under all of $(W,\cdot)$ is $(p-1)\rho$.
\end{proof}

We are, of course, interested in what these line bundles are; it turns out that we can describe them explicitly, which is the content of Theorem~\ref{thmD:st_st_component_description}.

\begin{theorem}
\label{thm:determine_line_bundles_for_p-1_rho_block}
Let $\lambda\in \Lambda$, and let $\mu$ be maximal with respect to $\succeq$ such that $\lambda-p\mu\ge (p-1)\rho$. Then 
\[\pi_{(p-1)\rho}\fr_*\ooo_X(\lambda)\cong \st\otimes  \st \otimes \ooo_X(\mu).\]
In particular, in the notation of Theorem~\ref{thm:decomp_line_bundle_into_vector_bundles},
\[\fff_{(p-1)\rho, (p-1)\rho}^{i,j}\cong \ooo_X(\mu).\]
\end{theorem}
\begin{proof}
Recall from \S\ref{subsection:appendix:vinberg_monoid} that 
\[R\coloneqq \ff[\wh{X}]=\bigoplus_{\lambda\in \Lambda}t^{\lambda}F_{\le \lambda},\]
and let $M$ be the $\Lambda$-graded $R$-module associated to the coherent sheaf $\eee=\pi_{(p-1)\rho}\fr_*\ooo_X(\lambda)$, as discussed in \S\ref{subsection:graded_module_vinberg_semigroup}. By Corollary~\ref{cor:steinberg_splits_into_line_bundles_general}, $\eee$ is isomorphic to a direct sum of line bundles. Therefore, by Corollary~\ref{cor:line_bundles_induce_free_A_module}, $M$ is a free $R$-module, clearly of rank $\dim\st\otimes  \st $ (by checking the rank of $\eee$). It follows that 
\[ M\cong \bigoplus_{i=1}^{\dim \st\otimes  \st } R(\lambda_i)\] as an $R$-module. Pick $\gamma$ to be any maximal $\lambda_i$ with respect to $\succeq$, and pick $v$ to be any nonzero vector in the lowest nonzero graded component of $R(\gamma)$, namely in $R(\gamma)_{-\gamma}\cong \ff$. Then 
\[\uzero\otimes \uzero.v\supseteq \st\otimes  \st ,\]
as $v\ne 0$ and $v$ lies in the $(p-1)\rho$ block. Since $\gamma$ is minimal and $\succeq$ is a partial order on $\Lambda$ (by Proposition~\ref{prop:when_is_succeq_relation}), it follows that
\[ R(\lambda_i)_{-\gamma}=
\begin{cases}
\ff  & \lambda_i=\gamma,\\
0 & \lambda_i\ne \gamma.
\end{cases}\]
By counting dimensions, we have
\[ \dim \st\otimes  \st  \ge \#\{\lambda_i \mid \lambda_i=\gamma\} \ge \dim \uzero\otimes \uzero.v\ge \dim\st\otimes  \st ,\]
with equality if and only if $\lambda_i=\gamma$ for all $\lambda_i$, i.e. that 
\[ M\cong R(\gamma)\otimes \st\otimes  \st \implies \pi_{(p-1)\rho}\fr_*\lll \cong \st\otimes  \st \otimes \ooo_X(\gamma).\]

It only remains to see that $\gamma=\mu$, where $\mu$ is maximal with respect to $\succeq$ such that $\lambda-p\mu\ge (p-1)\rho$. We showed that
\[ M\cong R(\gamma)\otimes \st\otimes \st ,\]
so in particular
\[ M_{-\gamma}\cong \st\otimes \st = \pi_{(p-1)\rho}(\fr_* R(\lambda))_{-\gamma}\cong \pi_{(p-1)\rho}R_{\lambda-p\gamma}.\]
In other words, $\gamma$ is maximal such that $\pi_{(p-1)\rho}R_{\lambda-p\gamma}\ne 0$, which is the same as $\lambda-p\gamma\ge (p-1)\rho$. But this is exactly the definition of $\mu$, and since $\succeq$ is a partial order (by Proposition~\ref{prop:when_is_succeq_relation}), we conclude that $\gamma=\mu$.
\end{proof}
\begin{remark}
Proposition~\ref{prop:splitting_of_Os_from_p-1_rho} is an immediate corollary. Since $0$ is the maximal weight $\mu$ (with respect to $\succeq$) such that $(p-1)\rho-\mu\ge (p-1)\rho$, we deduce the statement of Proposition~\ref{prop:splitting_of_Os_from_p-1_rho}:
\[ \pi_{(p-1)\rho}\fr_*\ooo_X((p-1)\rho)\cong \st\otimes  \st \otimes \ooo_X.\]
\end{remark}

\begin{example}
Let $G=\psl_3$. Then 
\[ \pi_{(p-1)\rho}\fr_*\ooo_X \cong \st\otimes  \st \otimes \ooo_X(-\rho),\] where $\rho=\alpha_1+\alpha_2=\omega_1+\omega_2$.
\end{example}

\begin{example}
Let $G=\psl_4$. Then
\[ \pi_{(p-1)\rho}\fr_*\ooo_X \cong \st\otimes  \st \otimes \ooo_X(-2\alpha_1-2\alpha_2-2\alpha_3)=\st\otimes  \st \otimes \ooo_X(-2\omega_1-2\omega_3).\]
\end{example}

\subsection{Concrete Applications}
\label{subsection:concrete_applications}

In this subsection, we apply Theorem~\ref{thm:decomp_line_bundle_into_vector_bundles} to certain examples (namely, the root systems in \S\ref{section:appendix:decomposition_numbers}) by computing the ranks of the vector subbundles in ($\star$). In each of the following cases of $G$, the values of $d_\lambda,d_\mu$ are known, and for any particular choice of $\lambda$, the value $a_\lambda$ is easy to compute.

The case of $X_{\psl_2}$ is omitted: as noted before, $X_{\psl_2}\cong \pp^3$, which is a toric variety and is completely known due to \cite{thomason1987algebraic}, \cite{bogvad1998splitting}, and \cite{achinger2010note}.

\begin{theorem}
\label{thm:ranks_of_subbundles_PSL3}
Let $G=\psl_3$. Then for $\lambda,\mu \in\Lambda_p$, the possible ranks of the vector subbundles in $(\star)$ are as follows:
\[ \textnormal{rk } \fff_{\mu,\lambda}^{i,j}=a_\lambda \cdot d_\lambda \cdot d_\mu \in \{1, 3, 6, 12, 24\},\]
and $0$ for $\mu\not\sim \lambda$. The $a_\lambda, d_\lambda,d_\mu$ are described in \S\ref{subsection:appendix:d_lambda_for_A_2}.
\end{theorem}
\begin{proof}
See \S\ref{subsection:appendix:d_lambda_for_A_2}.
\end{proof}

\begin{remark}
In \cite{braden1967restricted}, the dimensions of the $L_\lambda$ are computed, giving the number of components $\fff_{\mu,\lambda}^{i,j}$ associated to each pair of weights $(\mu,\lambda)$. Note that the $p$-regular weights precisely correspond to the last case, while the boundaries (i.e., the boundaries of the top alcove, along with the wall common to both alcoves) precisely correspond to the $a_\lambda=3$ case, and the extreme point is $\lambda=(p-1)\rho$.
\end{remark}

\begin{theorem}
\label{thm:ranks_of_subbundles_PSL4}
Let $G=\psl_4$. For $\lll$ a line bundle on $X$, the ranks of the vector subbundles in $(\star)$ are
\begin{align*}
    \textnormal{rk } \fff_{\mu,\lambda}^{i,j}=a_\lambda \cdot d_\lambda \cdot d_\mu \in \{&1, 2, 3, 4, 6, 8, 9, 11, 12, 16, 18, 22, 24, 33, 36, 44, 48, 54, 66,72, 88, 96, 108, 121,\\ & 132, 144, 198, 216, 264, 288, 396, 432, 484, 528, 726, 792, 864, 1452, 1584, 2904\},
\end{align*} 
for $\mu\sim\lambda$ and $0$ otherwise. The $a_\lambda,d_\lambda,d_\mu$ are described in \S\ref{subsection:appendix:A3}.

\end{theorem}
\begin{proof}
See \S\ref{subsection:appendix:A3}.
\end{proof}

\begin{theorem}
Let $G=\msf{PSO}_5$. For $\lambda\in\Lambda_p$, we have $a_\lambda\in\{1,2,4,8\}$, and $d_\lambda$ are described in Proposition~\ref{prop:d_lambda_for_B_2}. As a result, the ranks of the vector subbundles in $(\star)$ must lie within the following finite set:
\[ \textnormal{rk } \fff_{\mu,\lambda}^{i,j}=a_\lambda \cdot d_\lambda \cdot d_\mu \in \{1, 2, 3, 4, 6, 8, 9, 12, 16, 18, 24, 32, 36, 48, 64, 72, 96, 128\}\]
for $\mu\sim\lambda$ and $0$ otherwise.
\end{theorem}
\begin{proof}
See \S\ref{subsection:appendix:d_lambda_for_B_2}.
\end{proof}
\begin{theorem}
Let $G$ be the semisimple adjoint group corresponding to the root system $G_2$. Then for $\lambda\in\Lambda_p$, we have $a_\lambda\in\{1, 2, 3, 4, 6, 12\}$ and $d_\lambda$ described in \S\ref{subsection:appendix:G2}. As a result, the ranks of the vector subbundles in $(\star)$ must lie within the following finite set:
\begin{align*}
    \textnormal{rk } \fff_{\mu,\lambda}^{i,j}=a_\lambda \cdot d_\lambda \cdot d_\mu \in \{&1, 2, 3, 4, 5, 6, 8, 9, 10, 12, 15, 16, 17, 18, 20, 24, 25, 27, 29, 30, 32, 34, 36, 40, 45,\\
    &48, 50, 51, 54, 58, 60, 64, 68, 72, 75, 80, 85, 87, 90, 96, 100, 102, 108, 116, 120, \\ &128,
    136, 144, 145, 150, 153, 160, 162, 170, 174, 180, 192, 204, 216, 232, 240, 255, \\& 256,
    261, 270, 272, 288, 289, 290, 300, 306, 320, 324, 340, 348, 360, 384, 408, 432, \\&435,
    464, 480, 493, 510, 512, 522, 540, 544, 576, 578, 580, 612, 648, 696, 720, 768, \\&816,
    841, 864, 867, 870, 918, 928, 960, 972, 986, 1020, 1024, 1044, 1080, 1088, \\&1152,
    1156, 1224, 1296, 1392, 1479, 1536, 1566, 1632, 1682, 1728, 1734, 1740, \\ &1836, 1856,
    1944, 1972, 2088, 2304, 2448, 2523, 2592, 2784, 2958, 3072, 3132, \\ &3264, 3364, 3456,
    3468, 3672, 3888, 4176, 5046, 5568, 5916, 6264, 10092\}
\end{align*}
for $\mu\sim\lambda$ and $0$ otherwise.
\end{theorem}
\begin{proof}
See \S\ref{subsection:appendix:G2}.
\end{proof}

\section{The class in $K$-theory}
\label{section:class_in_K_theory}
Let $\lll$ be a line bundle on $X$, a wonderful compactification. In this section, we compute the class of $\fr^*\fr_*\lll$ in the Grothendieck group of coherent sheaves on $X$ using localization and also compute the Chern character $\operatorname{ch}\fr_*\lll$ using the Grothendieck-Riemann-Roch formula. 
\begin{definition}
\label{definition:K_0(X)}
Let $K_0(X)$ denote the Grothendieck group of the category of coherent sheaves on $X$, and we set $K_0(X)_{\mathbb{Q}}\coloneqq K_0(X) \otimes_{\mathbb{Z}} \mathbb{Q}$. Let $\kt$ denote the Grothendieck group of the category of $\wt{T}\times \wt{T}$-equivariant coherent sheaves on $X$, and we set $\ktq \coloneqq\kt \otimes_{\mathbb{Z}} \mathbb{Q}$.
\end{definition}

\begin{remark}
Since $X$ is smooth, $\kt$ identifies with the Grothendieck group of the category of $\widetilde{T} \times \widetilde{T}$-equivariant vector bundles on $X$ (see \cite[Theorem 5.7]{thomason1987algebraic}). Similarly, $K_0(X)$ identifies with the Grothendieck group of the category of vector bundles on $X$.
\end{remark}

Note that there is a natural map $\kt\to K_0(X)$ by forgetting the $\wt{T}\times \wt{T}$-equivariant structure. Let $A^*(X)$ be the Chow ring of $X$ (see for example \cite[Section 8.3]{fulton2013intersection}) and set $A^*(X,\mathbb{Q}) \coloneqq A^*(X) \otimes_{\mathbb{Z}} \mathbb{Q}$. Recall the Chern character $\operatorname{ch}\colon K_0(X)_{\mathbb{Q}} \rightarrow A^*(X,\mathbb{Q})$ (see for example \cite[Sections 15.1 and 18]{fulton2013intersection}).

\begin{remark}
It follows from \cite[Example 15.2.16]{fulton2013intersection} that $\operatorname{ch}$ is an isomorphism.
\end{remark}

Let $\int\colon A^*(X,\mathbb{Q}) \rightarrow A^*(\operatorname{pt},\mathbb{Q})=\mathbb{Q}$ be the pushforward homomorphism. 
The Grothendieck–Riemann–\newline
Roch theorem implies (see for example \cite[Section 15.2]{fulton2013intersection}) that for every coherent sheaf $\mathcal{F}$ on $X$ we have 
\begin{equation*}
\chi(\mathcal{F})=\int \operatorname{ch}(\mathcal{F}) \cdot \operatorname{td}_X,    
\end{equation*}
where $\operatorname{td}_X$ is the Todd genus of the (tangent bundle of) $X$.

Consider the natural embedding $\iota\colon X^{\widetilde{T} \times \widetilde{T}} \subset X$. It induces the pullback homomorphism 
\begin{equation*}
\iota^*\colon K_0^{\widetilde{T} \times \widetilde{T}}(X) \to K^{\widetilde{T} \times \widetilde{T}}_0(X^{\widetilde{T} \times \widetilde{T}})   
\end{equation*}                         
that becomes an isomorphism after tensoring by the field of fractions of $K_0^{\widetilde{T} \times \widetilde{T}}(\operatorname{pt})$ (see \cite{thomason1992formule} or \cite{takeda1994localization}).

It follows from \cite{bialynicki1973some} that $X$ has a stratification by $\widetilde{T} \times \widetilde{T}$-invariant locally closed subvarieties such that each of them is isomorphic to an affine space. Using the same argument as in \cite[Lemma 5.5.1]{chriss1997representation} we conclude that $K^{\widetilde{T} \times \widetilde{T}}_0(X)$ is a free module over $K_0^{\widetilde{T} \times \widetilde{T}}(\operatorname{pt})$, hence, $\iota^*$ is an embedding. 

\subsection{The Grothendieck group $K_0(X)$}
\label{subsection:grothendieck_group_K_0(X)}
The goal of this section is to describe the class $[\fr^*\fr_*\ooo_X(\lambda)] \in K_0(X)$. Since $\ooo_X(\lambda)$ is $\widetilde{T} \times \widetilde{T}$-equivariant, we can consider $[\fr^*\fr_*\ooo_X(\lambda)]$ as an element of $K^{\widetilde{T} \times \widetilde{T}}_0(X)$. We understand $\fr^*\fr_*\ooo_X(\lambda)$ as the same underlying sheaf of $\fr_*\ooo_X(\lambda)$, but identify $\fr_*\ooo_X(\lambda)$ with the twisted $\wt{T}\times\wt{T}$-action (coming from the standard $\wt{T}\times \wt{T}$-action on $\ooo_X(\lambda)$) rather than the standard action on $\fr_*\ooo_X(\lambda)$. Therefore the action on $\fr^*\fr_*\ooo_X(\lambda)$ is identified with the standard action on $\fr_*\ooo_X(\lambda)$; consequently it simplifies our work to consider $\fr^*\fr_*\ooo_X(\lambda)$. Since $\iota^*$ is an embedding (also an isomorphism after tensoring with $\textnormal{Frac}(K_0^{\wt{T}\times \wt{T}}(\textnormal{pt}))$), it is enough to describe the image $\iota^*[\fr^*\fr_*\ooo_X(\lambda)] \in K_0^{\widetilde{T} \times \widetilde{T}}(X^{\widetilde{T} \times \widetilde{T}})$. 
The subvariety $X^{\wt{T}\times \wt{T}}$ is precisely the finite set of points $W\times W\subset \wt{G}/\wt{B}\times \wt{G}/\wt{B}$, by \cite[Lemma~4.2]{evens2008wonderful} (under the identification $ W\times W\ni (w_1,w_2)\leftrightarrow (w_1,w_2)\cdot 1\in G\subset X$).
\begin{theorem}
\label{thm:compute_class_of_pushforward}
The class of $\fr^*\fr_*\ooo_X(\lambda)$ in $K_0^{\wt{T}\times \wt{T}}(X)$, pulled back to $K_0^{\wt{T}\times \wt{T}}(X^{\wt{T}\times \wt{T}})$, is given by
\[ \left[\fr^*\fr_*\ooo_X(\lambda)\bigg|_{X^{\wt{T}\times \wt{T}}}\right]=\left(\left[\fr^*\fr_*\ooo_X(\lambda)\bigg|_{(y,w)}\right]\right)_{(y,w)\in W\times W}\in K_0^{\wt{T}\times \wt{T}}(X^{\wt{T}\times \wt{T}})\cong K_0^{\wt{T}\times \wt{T}}(\textnormal{pt})^{\oplus |W|^2},\]
where we have the equality (in $K_0^{\wt{T}\times \wt{T}}(z)$)
\begin{align*}
    \Bigg[\fr^*\fr_*\ooo_X(\lambda)&\bigg|_{(y,w)}\Bigg] = [(-y(\lambda),w w_0(\lambda)]+\\
    &\sum_{\substack{0\le a_\gamma,b_\mu,c_{\alpha_i}<p \\ \gamma,\mu\le \Phi^+\\\alpha_i\in\Delta}} \left[(-y(\lambda),ww_0(\lambda))+\sum_{\substack{\gamma,\mu\in\Phi^+,\\ \alpha_i\in\Delta}}\left(a_\gamma y(\gamma)+c_{\alpha_i}y(\alpha_i),-b_\mu w(\mu)-c_{\alpha_i}w(\alpha_i)\right)\right],
\end{align*}
and the symbol $[(a,b)]$ for $a,b\in\Lambda^2$ indicates the one-dimensional $\wt{T}\times \wt{T}$-module of weight $(a,b)$.
\end{theorem}
\begin{proof}
Let $\iii_Y$ denote the ideal sheaf cutting out $Y\subset X$ for any closed subvariety $Y\subset X$. (In this case, $Y=X^{\wt{T}\times \wt{T}}$.) We have the short exact sequence
\[ 0\to \iii_Y \lll/\iii_Y^{(p)}\lll\to (\fr_*^X \lll)\bigg|_Y\to \fr_*^Y(\lll|_Y)\to 0,\]
where $\iii_Y^{(p)}$ is locally the image of $(f^p \mid f\in I)$. (The sheaf $\iii_Y$ is the ideal sheaf of the \textit{Frobenius neighborhood} of $Y$, see \cite{mirkovic2001geometric}). It suffices to check each point $(y,w)\in W\times W$ independently due to the isomorphism $ K_0^{\wt{T}\times \wt{T}}(\textnormal{pt})^{\oplus |W|^2}$. Fix $z=(y,w)\in W\times W$.
After applying $\fr^*$ (an exact functor), we obtain a copy of 
\[\left[\fr^*\fr_*\left(\ooo_X(\lambda)\big|_{z}\right)\right]=\left[\fr^*\fr_*\left(\ooo_{\wt{G}/\wt{B}}(\lambda)\boxtimes \ooo_{\wt{G}/\wt{B}}(-w_0\lambda)\bigg|_{(y,w)}\right)\right]=[\left(-y(\lambda),ww_0(\lambda)\right)].\]
We may assume $\lambda=0\implies \ooo_X(\lambda)=\ooo_X$, and then shift all weights by $(-y(\lambda),ww_0(\lambda))$. Now \'etale locally around $z$, we find coordinates $x_1,\dots,x_m$ where each $x_i$ is a $\wt{T}\times \wt{T}$-eigenvector. Then the term $\fr_*^Y(\lll|_Y)$ in the short exact sequence becomes
\[ (x_1,\dots,x_m)\ff[[x_1,\dots,x_m]]/(x_1^p,\dotsm,x_m^p)\ff[[x_1,\dots,x_m]]\cong \bigoplus_{0\le a_i<p}\ff \{x_1^{a_1}\dotsm x_m^{a_m}\}\] as $\wt{T}\times \wt{T}$-modules. The weights corresponding to $x_1^{a_1}\dotsm x_m^{a_m}$ are precisely the weights of the cotangent space $T_z^*(X)$ (due to applying $\fr^*$ to cancel the twisted action of $\fr_*$ on weights), which are the negative of the weights of the tangent space $T_z(X)$. The result follows from applying \cite[Lemma~4.4]{evens2008wonderful}.
\end{proof}
\begin{remark}
A similar approach can be carried out for the $\wt{T}$-fixed points, which yields a toric variety $Y\subset X$ (rather than a finite set of points). Note that the description of the (co)normal bundle of $Y$ in $X$ is given in \cite[Section 3.1]{brionjoshua2008equivariant}.
\end{remark}

\subsection{The Chern character}
\label{subsection:chern_character}
Recall that the Chern character induces the isomorphism
\[ \textnormal{ch}: K_0(X)_{\qq}\xr{\sim}A^*(X,\qq).\]
The Chern character of $\fr_*\lll$ can then reveal information about $\fr_*\lll$ itself (specifically, its class in rational $K$-theory).

\begin{definition}
Define $\psi^k$ to be the \textit{$k$th Adams operation}, as in \cite{pink2012adams}.
\end{definition}
The $p$th Adams operation is characterized by $\psi^p=\fr^*$. By \cite[Proposition~I.6.3]{fulton2013riemann}, $\psi^p$ sends an element $a_i$ of homogeneous degree $i$ in the Chow ring to $p^i a_i$. Consequently $\psi^p$ is invertible as an endomorphism of $A^*(X,\qq)$.

\begin{theorem}
\label{thm:chern_character}
Let $\lll$ be a line bundle on a smooth projective variety $X$ (in particular, a wonderful compactification) and let $\todd$ denote the Todd genus of $X$ in $A^*(X,\qq)$. Then
\[ \textnormal{ch}(\fr_*\lll)=p^{\dim X}\cdot \frac{(\psi^p)^{-1}(\textnormal{ch}(\lll)\cdot \todd)}{\todd}.\]
\end{theorem}
\begin{proof}
By the projection formula, we have that
\[ \chi(\fr_*\lll\otimes \fff)=\chi(\lll\otimes \fr^*\fff)\] for any vector bundle $\fff$. Applying Grothendieck-Riemann-Roch to both sides, we find that
\begin{equation}
\label{eq:groth_riemann_roch}
\int \textnormal{ch}(\fr_*\lll)\cdot \textnormal{ch}(\fff)\cdot \todd =\int \textnormal{ch}(\lll)\cdot \textnormal{ch}(\fr^*\fff)\cdot \todd.
\end{equation}
Noting that $\psi^p$ is invertible on $K_0(X)_{\qq}$ (and consequently on $A^*(X,\qq)$), the right hand side becomes
\begin{align*}
    \int \textnormal{ch}(\lll)\cdot \textnormal{ch}(\fr^*\fff)\cdot \todd&=\int \textnormal{ch}(\lll)\cdot \psi^p(\textnormal{ch}(\fff))\cdot \todd,\\
    &=\int (\psi^p)^{-1}\left(\psi^p(\textnormal{ch}(\lll)\cdot \todd)\cdot \textnormal{ch}(\fff)\right),\\
    &= p^{\dim X}\int (\psi^p)^{-1}(\textnormal{ch}(\lll)\cdot \todd)\cdot \textnormal{ch}(\fff),
\end{align*}
after applying the fact that $\int \psi^p(-)=p^{\dim X}\int -$ (since $\int \psi^p(x)=\pi_*(\fr_*(\fr^*x \cdot [X]))=\pi_*(x \cdot \fr_*[X])=p^{\operatorname{dim}X}\pi_*(x)$, where $x \in A^*(X,\mathbb{Q})$, $[X] \in A^*(X,\mathbb{Q})$ is the class that represents $X$ and $\pi\colon X \rightarrow {\operatorname{pt}}$ is the natural map).  Now comparing the two sides of (\ref{eq:groth_riemann_roch}) and letting $\fff$ run over all vector bundles, we conclude equality of integrands, hence 
\[\textnormal{ch}(\fr_*\lll)\cdot \todd=p^{\dim X}\cdot (\psi^p)^{-1}(\textnormal{ch}(\lll)\cdot \todd),\]
which yields the result.
\end{proof}

\begin{remark}
A similar computation can be carried out more explicitly for the $\wt{T}$-fixed points of $X$, which form a toric variety $Y$. By the projection formula, we have \[ \chi(\fr^*\fr_*\lll\otimes \fff)=\chi(\lll\otimes \fr^*\fr_*\fff).\]
We have the analogue of the Grothendieck-Riemann-Roch formula by \cite[Corollary~3.1]{edidin2000riemann} in the $\wt{G}$-equivariant Chow ring $A^*_{\wt{G}}(X,\qq)$ of $X$, and following the proof of Theorem~\ref{thm:chern_character} we find that
\[ \textnormal{ch}(\fr^*\fr_*\fff)=\psi^p\left(p^{\dim X}\frac{(\psi^p)^{-1}(\textnormal{ch}(\fff)\cdot \todd)}{\todd}\right)=p^{\dim X}\frac{\textnormal{ch}(\fff)\cdot \todd}{\psi^p(\todd)},\]
where $\todd$ is now the equivariant Todd genus. The Adams operations are still well-defined (see for example \cite{joshua2003k}) and functoriality of $\fr$ ensures that $\psi^p$ is still invertible (see \cite[Proposition~I.6.3]{fulton2013riemann}).
The pullback map induced from $Y\into X$ gives an isomorphism $r:A^*_{\wt{G}}(X,\qq)\xr{\sim}A^*_{\wt{T}}(Y,\qq)^W$ by \cite[Theorem~2.2.1]{brionjoshua2008equivariant}. The image of the $\wt{G}$-equivariant Todd genus $\textnormal{td}_X$ under $r$ is computed in the paper \cite{brionjoshua2008equivariant}; thus using that $r$ commutes with $\psi^p$, one can obtain an explicit formula for $\textnormal{ch}(\fr^*\fr_*\fff)$ as an element of $A^*_{\wt{T}}(Y,\qq)$.
\end{remark}

\appendix
\section{Vinberg monoid in characteristic $p$}
\label{section:appendix:vinberg_monoid}

The wonderful compactification has another realization through the Vinberg monoid that will be important to us. We will recall this construction following \cite{rittatore2001very}, \cite{anabalibanu2018part2}, \cite{brion2000large}.

\subsection{The filtration on $\ff[\wt{G}]$}

\begin{theorem}
\label{filtration_by_dominant_weights}
There exists a filtration $F_{\le \lambda}$ of $\wt{G}\times \wt{G}$-modules indexed by $\lambda\in\Lambda^+ $ such that $\ff[\wt{G}]=\colim_{\lambda\in\Lambda^+ }F_{\le \lambda}$ and $\Gamma(X,\ooo_X(\lambda))\cong \colim_{\lambda\ge \mu \in\Lambda^+}F_{\le \mu}$ as $\wt{G}\times \wt{G}$-modules.
\end{theorem}
\begin{proof}
We briefly outline the construction; for full proofs, see \cite{rittatore2001very} or \cite{anabalibanu2018part2}. The composition factors are isomorphic to $M_\lambda\otimes M_{-w_0\lambda}$ (see \cite[Proposition~4.20]{jantzen2003representations}). As a result, $F_{\le 0}=\ff\subset \ff[\wt{G}]$. We now proceed by induction. For any dominant $\lambda$, it is a lowest weight vector in the quotient $\ff[\wt{G}]/(\sum_{\mu<\lambda}F_{\le\mu})$, so $\ff[\wt{G}]/(\sum_{\mu<\lambda}F_{\le \mu})$ contains a submodule $M$ isomorphic to $M_\lambda\otimes M_{-w_0\lambda}$ by \cite[Lemma~II.4.15]{jantzen2003representations}. We then take $F_{\le \lambda}$ to be the preimage of $M$ inside $\ff[\wt{G}]$ (via the canonical projection).
\end{proof}

\begin{definition}
We define $F_{\le \lambda}$ to be the filtration described in Theorem~\ref{filtration_by_dominant_weights}. By abuse of notation, for any weight $\lambda$, we'll denote 
\[F_{\le \lambda}\coloneqq\colim_{\mu\in \Lambda^+, \hspace{2mm} \mu\le \lambda}F_{\le \mu},\] so that $\Gamma(X,\ooo_X(\lambda))\cong F_{\le \lambda}$ (for \textit{any} weight $\lambda$).
\end{definition}

\begin{remark}
\label{remark:filtration_splits_for_small_weights}
Unlike in the characteristic $0$ case, we do not get a direct sum, but rather a filtration. However, in some cases, the filtration $F$ splits (for specific $\lambda$). For any $G$ and dominant $\lambda$, picking $p$ sufficiently large, we have
\[ F_{\le \lambda}\cong \bigoplus_{\mu\le \lambda}M_\mu\otimes M_{-w_0\mu}.\]
For fixed dominant weight $\lambda$, we can always choose $p\gggg 0$ to make this true: for example, large enough that $\lambda$ lies in the fundamental alcove, hence all weights $\mu\le\lambda$ are the sole weight in their blocks, and the central idempotents splits the filtration into the direct sum of the composition factors $M_\mu\otimes M_{-w_0\mu}=L_{\mu}\otimes L_{-w_0\mu}$.
\end{remark}

\subsection{The Vinberg monoid}
\label{subsection:appendix:vinberg_monoid}
We briefly recount the Vinberg monoid and its relationship to the wonderful compactification. See \cite{rittatore2001very} and \cite{brion2000large} for proofs; for more detail in the case of characteristic $0$, see \cite{anabalibanu2018part2}. Let $G$ be a semisimple adjoint group, and let $X$ be the wonderful compactification of $G$.

\begin{definition}
Let $R=\bigoplus_{\lambda\in \Lambda}t^{\lambda}F_{\le \lambda}$ denote the Rees algebra of $\ff[\wt{G}]$ with the filtration $F$. Define the \textit{Vinberg monoid} $\wt{X}$ of $G$ to be $\spec R$.
\end{definition}

The Vinberg monoid comes equipped with a natural $\wt{G}\times \wt{G}$-action. In \cite{vinberg1995asymptotic},\cite{vinberg1995reductive}, Vinberg has defined a surjection $\wt{X}\to \aa^{\ell}$. We briefly recall this map. First, we define $T^+\coloneqq \spec \ff[t^{\alpha}\mid \alpha\in \Delta]\cong \aa^{\ell}$. We have the natural embedding $\ff[t^{\alpha}\mid \alpha\in\Delta]\into \ff[\wt{G}]$ given by sending $t^{\alpha}\mapsto t^{\alpha}\cdot 1\in t^{\alpha}F_{\le \alpha}$, yielding the $\wt{T}\times \wt{T}$-equivariant surjection $\Upsilon:\wt{X}\to \aa^{\ell}$.

\begin{definition}
We define $\wh{X}$ to be the unique dense open subvariety of $\wt{X}$ which meets each fiber of $\Upsilon$ in the open $\widetilde{G} \times \widetilde{G}$-orbit of that fiber.
\end{definition}

The variety $\wh{X}$ is normal and quasi-affine. The variety $\wt{X}$ is normal and the affine closure of $\wh{X}$ (see \cite[page~17]{brion2007total}). It turns out that $\wh{X}$ is the relative spec of the sheaf of $\ooo_X$-modules $\bigoplus_{\lambda\in\Lambda}\ooo_X(\lambda)$, related to the wonderful compactification $X$ by 
\[ X\cong \wh{X}/\wt{T}.\]

\subsection{Graded rings and modules}
\label{subsection:graded_module_vinberg_semigroup}

In fact, coherent sheaves are completely determined by a corresponding $\Lambda$-graded modules over $\ff[\wt{X}]$, which are equivalent to $\wt{T}$-equivariant sheaves. We have the canonical $\wt{T}$-equivariant map
\[ \pi: \wh{X}\onto X\] given by the quotient of the $\wt{T}$-action. Note that $\pi$ is $\wt{G}\times \wt{G}$-equivariant. More generally, assume we have an action of an algebraic group $H$ such that $\pi$ is $H$-equivariant. This induces the map 
\[ \Gamma(\wh{X},-)\circ \pi^*:\textnormal{Coh}^{H} (X)\xr{\sim}\textnormal{Coh}^{H\times \wt{T}}(\wh{X})\to \Lambda-\textnormal{graded }\ff[\wh{X}]-\textnormal{modules with }H\textnormal{-action}.\] In particular, this sends 
\[ \fff\mapsto \bigoplus_{\la\in\Lambda}t^{\lambda}\cdot \Gamma(X,\fff\otimes \ooo_X(\lambda)).\]

\begin{lemma}
\label{lemma:torsor_pullback}
Let $G$ be a group acting on a variety $Y$, with subgroup $H$, such that $\pi:Y\to X=Y/H$ is the map given by quotient by $H$-action. Then for any $G$-module $V$, we have that $\pi^* (Y\times^H V)$ is the total space of $\ooo_Y\times V$.
\end{lemma}
\begin{proof}
Standard.
\end{proof}

\begin{corollary}
\label{cor:line_bundles_induce_free_A_module}
For any $\wt{G}$-module $V$, then $\pi^* (\wh{X}\times^T V)$ is the total space of $\ooo_{\wh{X}}\otimes_\ff V$. In particular, for any line bundle $\lll$ on $X$, then $\pi^*\lll\cong \ooo_{\wh{X}}$, and hence the graded module associated to $\lll$ is a free $\ff[\wh{X}]$-module of rank $1$.
\end{corollary}
\begin{proof}
This is immediate from Lemma~\ref{lemma:torsor_pullback}, with $H=\wt{T}$ and the map $\pi:\wh{X}\to X=\wh{X}/T$. The second statement follows from the fact that the total space of $\ooo_X(\lambda)$ is $\wh{X}\times^{\wt{T}}\ff_{-\lambda}$.
\end{proof}

\section{Algebraic Tools}
\label{section:appendix:algebra}

In this section, we'll summarize the algebraic tools we need.
\subsection{Theory of idempotents}
\label{subsection:appendix:theory_of_idempotents}

Our main tool in \S\ref{section:splitting_via_idempotents} is the theory of idempotents. We review the theory of idempotents, following \cite[\S54-55]{curtis1966representation}. Fix $A$ a finite-dimensional associative $\ff$-algebra. We wish to study the (left) indecomposable modules over $A$ (all results which are true for left $A$-modules have obvious analogues for right $A$-modules as well).
\begin{definition}
Let $A$ be a left $A$-module in the natural way. Then $A$ decomposes into a direct sum of indecomposable $A$-modules: 
\[A\cong A_1\oplus \dotsm \oplus A_n.\]
The $A_i$ are called the (left) \textit{principal indecomposable modules}, or PIMs.
\end{definition}

It is immediate that every principal indecomposable module is projective. Some of the principal indecomposable modules may be isomorphic, but in any case, there are only a finite number of non-isomorphic principal indecomposable modules of $A$.

\begin{definition}
A (left) $A$-module $M$ is a \textit{projective indecomposable module} if it is projective and indecomposable.
\end{definition}

Projective indecomposable modules are just isomorphism classes of principal indecomposable modules. To understand this identification, we first need to understand how principal indecomposable modules arise. Their existence is controlled by the idempotents in $A$.

\begin{definition}
An \textit{idempotent} of $A$ is an element $e\in A$ such that $e^2=e$. Two idempotents $e,f\in A$ are \textit{orthogonal} if $ef=fe=0$. An idempotent is called \textit{primitive} if it cannot be written as the sum of two orthogonal idempotents.
\end{definition}

The following holds by \cite[Theorem~54.5]{curtis1966representation}: a left ideal $I$ of $A$ is a principal indecomposable module if and only if $I=Ae$ for some primitive idempotent $e$ in $A$. It follows that every principal idecomposable (left) $A$-module is projective indecomposable and vice versa. By \cite[Corollary~54.12]{curtis1966representation} we have the following bijection:
\[ \{\textnormal{irreducible }A-\textnormal{modules}\}\leftrightarrow \{\textnormal{projective indecomposable }A-\textnormal{modules}\}\leftrightarrow\{\textnormal{PIMs}\}/\sim,\]
with the identification given by $L\mapsto P_L$, the projective cover of $L$, and PIM stands for principal indecomposable $A$-modules. The projective cover satisfies the conditions that $\hom{}(P_L,L)=\ff$ and $\hom{}(P_L,L')=0$ for all $L'\ne L$.

As we mentioned, all results for left $A$-modules have obvious analogues for right $A$-modules, as the theory is essentially the same. But not only do the theories match, but \cite[Corollary~54.10]{curtis1966representation} states that for an idempotent $e\in A$, the left ideal $Ae$ is indecomposable if and only if the right ideal $eA$ is indecomposable. Therefore, the idempotents in $A$ completely control both the left and right indecomposable modules.

\begin{lemma}
\label{lemma:pairwise_orthogonal_primitive_idempotents}
Let $A$ be an associative $\ff$-algebra. Let $\ll$ denote the set of isomorphism classes of irreducible $A$-representations. Then there exist pairwise orthogonal primitive idempotents which sum to $1 \in A$, and they are indexed by $(L,i)$ with $L \in \mathbb{L}$ and $i=1,2,\ldots,\operatorname{dim}L$. 
\end{lemma}
\begin{proof}
The irreducible representations of $A$ are the same as the irreducible representations of $A/\textnormal{rad}$, as for any such simple module $L$, $\textnormal{rad}\cdot L=0$. Since $A/\textnormal{rad}$ is semisimple, the Wedderburn-Artin theorem implies that
\[ A/\textnormal{rad}\cong \bigoplus_{L\in\ll } \textnormal{End}(L).\]
Now,
\[\textnormal{End}(L)\cong L \otimes L^*,\]
hence decomposes into the direct sum of $\dim L$ copies of $L$ as irreducible $A$-modules, given by primitive idempotents $\wt{e}_{L}^1,\dots,\wt{e}_{L}^{\dim L}$ (see \cite[\S54-55]{curtis1966representation}), where there are $d=\dim L$ many primitive idempotents, as $\textnormal{End}(L)\cong \textnormal{Mat}_{d\times d}(\ff)$. Therefore, we obtain exactly $d$ such idempotents in $A/\textnormal{rad}$ by choosing $\wt{e}_L^i$ to be the matrix with a single $1$ at the $(i,i)$ entry; these are primitive because the corresponding $A/\textnormal{rad}$-module is isomorphic to $L$, which is simple. It is clear that their sum is $1\in A/\textnormal{rad}$. Now by \cite[Theorem~21.28]{lam1991first}, these lift to primitive idempotents $e_{L}^1,\dots,e_{L}^{\dim L}$ of $A$, and by \cite[Proposition~21.25]{lam1991first}, they lift compatibly to pairwise orthogonal idempotents, which sum to a lift of $1\in A/\textnormal{rad}$ in $A$: but we can choose this to be $1\in A$.
\end{proof}
\begin{notation}
Denote by $\mathscr{E}$, the set of idempotents $e_{L}^i$ in Lemma~\ref{lemma:pairwise_orthogonal_primitive_idempotents}, where $L\in\ll$ and $i=1,2,\dots,\dim L$.
\end{notation}

\begin{lemma}
\label{lemma:appendix:e_idempotent_=_projective_cover}
Let $L$ be an irreducible $A$-module. Then we have the isomorphism of left $A$-modules
\[A e_L^j\cong P_L,\]
where $P_L$ denotes the projective cover of $L$.
\end{lemma}
\begin{proof}
By \cite[Theorem~54.5]{curtis1966representation}, $Ae_L^j$ is a projective indecomposable (as projective indecomposables are precisely the principle indecomposables up to isomorphism). Therefore, it is isomorphic to $P_L$, the projective cover of $L$.
\end{proof}
\begin{corollary}
We have the isomorphism of left $A$-modules
\[A\cong \bigoplus_{L\in\ll } P_L^{\oplus \dim L}.\]
\end{corollary}
\begin{proof}
Applying Lemma~\ref{lemma:appendix:e_idempotent_=_projective_cover}, we have that
\[A= \bigoplus_{L\in\ll }\bigoplus_{j=1}^{\dim L}Ae_L^j\cong \bigoplus_{L\in\ll }P_L^{\oplus \dim L}\]
as left $A$-modules.
\end{proof}

The main point of our primitive idempotents is to decompose $A$:
\begin{proposition}
\label{prop:appendix:dimension_of_terms}
We have the decomposition (as an $\ff$-vector space)
\[ A= \bigoplus_{(e_{L'}^i,e_L^j)\in\mathscr{E}\times \mathscr{E}} e_{L'}^i A e_L^j.\]
The dimension of $e_{L'}^i A e_{L}^j$ is precisely $[L':P_L]$, the multiplicity of $L'$ in $P_L$ as a composition factor.
\end{proposition}
\begin{proof}
By Lemma~\ref{lemma:appendix:e_idempotent_=_projective_cover}, we have that $A e_L^j\cong P_L$ and $A e_{L'}^i \cong P_{L'}$. Now by \cite[Theorem~54.15]{curtis1966representation}, we have that
\[ \dim e_{L'}^i A e_L^j = \dim e_{L'}^i P_L = \dim \hom{A}(Ae_{L'}^i, P_L)=\dim\hom{A}(P_{L'},P_L)=[L':P_L].\]
\end{proof}

\begin{remark}
\label{remark:central_idempotents_preserves_module}
Sometimes it is advantageous \textit{not} to use primitive idempotents. If we instead use the collection of \textit{central} idempotents, this gives us the decomposition of $A$ into $A$-submodules, i.e. $A$ decomposes \textit{as an $\ff$-algebra} into the direct sum of subalgebras. In particular, applying the central idempotents to any module preserves the $A$-module structure. On the other hand, applying a left (resp. right) primitive idempotent to a module loses the left (resp. right) module structure, thus applying both left and right primitive idempotents loses the $A$-module structure on both sides entirely.
\end{remark}

\subsection{Assorted results}
\label{subsection:assorted_results}
In this subsection, we compile standard or known results which are used in the paper.

Our first result concerns symmetric algebras. A finite-dimensional associative $\ff$-algebra is called \textit{Frobenius} if there exists a nondegenerate bilinear form $\langle -,-\rangle$ such that $\langle ab,c\rangle=\langle a,bc\rangle$ for all $a,b,c,\in A$, equivalently, $A$ as a right $A$-module is isomorphic to its dual as a left $A$-module. A finite-dimensional associative $\ff$-algebra $A$ is called \textit{symmetric} if there exists a bilinear form which is both symmetric and makes $A$ a Frobenius algebra. (See \cite[\S R.2-R.3]{humphreys2006ordinary}.)
\begin{lemma}
\label{lemma:frobenius_symmetric_algebra_isomorphic_to_dual}
Let $A$ be a symmetric $\ff$-algebra. Then $A\cong A^*$ as $A\otimes A$-modules (with the standard left and right action).
\end{lemma}
\begin{proof}
Clear.
\end{proof}

\begin{proposition}[Exercise~III.6.10, \cite{hartshorne2013algebraic}]
\label{right_adjoint_to_fr_*}
Let $X$ be the wonderful compactification and $\fr:X\to X$ the Frobenius morphism. The right adjoint functor of $\fr_*$ is given by $\fr^!: \fff\mapsto \fr^*\fff\otimes \omega_X^{\otimes (1-p)}$, where $\omega_X$ is the canonical sheaf on $X$.
\end{proposition}

\begin{lemma}
\label{lemma_haboush_surjection_implies_nonzero}
Suppose we have an affine morphism $f:X\to Y$, and coherent sheaves $\eee_X$ on $X$ and $\eee_Y$ on $Y$. Suppose we have a surjective map of sheaves
\[ j: f^*\eee_Y\to \eee_X.\]
Then the corresponding adjoint map
\[\eee_Y\to f_*\eee_X\] 
is nonzero at every point.
\end{lemma}
\begin{proof}
The problem is local, so we may check this on affine charts. Let $\spec B\subset Y$ and $f^{-1}(\spec B)=\spec A\subset X$, with the induced map $f^\sharp:B\to A$. Let $\eee_Y|_{\spec B}=\wt{N}$ and $\eee_X|_{\spec A}=\wt{M}$. By hypothesis, we have a surjection
$\varphi: N\otimes_B A \onto M$,
which by adjunction gives us the map
$\psi:N\to M_B$,
where $M_B$ denotes $M$ as a $B$-module $M$. Let $[\mf{m}]\in\spec B$ be any closed point. It suffices to check that the induced map $\psi_{\mf{m}}':N/\mf{m}N\to M_B/\mf{m}M_B$ is nonzero, or equivalently that $\psi_{\mf{m}}:N\onto N/\mf{m}N\to M_B/\mf{m}M_B$
is nonzero. But if $\psi_{\mf{m}}$ were zero at some $[\mf{m}]$, then $\psi_{\mf {m}}(n)=\varphi(n\otimes 1)\in \mf{m}M$ for all  $n\in N$, implying that $\im(\varphi)\subset \mf{m}M$. Then any $\mf{n}\in\spec A$ which contains $f^\sharp(\mf{m})$ satisfies $\im\varphi\subset \mf{n}M\implies \varphi|_{[\mf{n}]}=0$, contradiction.
\end{proof}

\section{Decomposition numbers}
\label{section:appendix:decomposition_numbers}
In this section, we review general properties of the decomposition numbers $d_\lambda$ (defined in Definition~\ref{definition:d_lambda}), then give an overview of the general procedure to compute them for $p\gggg 0$ (as well as the closely related problem of the dimensions of $L_\lambda$). For ``small" root systems, the decomposition numbers are completely known, e.g. in \cite{humphreys2006ordinary}. In these cases, we compute the ranks of the summands $\fff_{\mu,\lambda}^{i,j}$ in ($\star$), Theorem~\ref{thm:decomp_line_bundle_into_vector_bundles}.

\subsection{Properties of $d_{\lambda}$}

Our discussion will be based on \cite[\S3.4]{humphreys2006ordinary}.

\begin{definition}
The affine Weyl group $W^{\textnormal{aff}}$ acts on $\Lambda\otimes_{\zz}\rr$ by translations and reflections by \textit{affine root hyperplanes}: hyperplanes of the form 
\[ \langle \lambda+\rho,\alpha^\vee\rangle=mp\]
for some $\alpha\in \Phi^+$ and $m\in\zz$. We define an \textit{open alcove} to be a connected component of $\Lambda\otimes_{\zz}\rr$ with the affine root hyperplanes removed. We define an \textit{alcove} to be the closure of an open alcove; note that alcoves are fundamental domains for $(W^{\textnormal{aff}},\cdot_p)$-action on $\Lambda\otimes_{\zz}\rr$.

The \textit{fundamental alcove} is the alcove defined by
\[ \{ \lambda\in\Lambda\otimes_{\zz}\rr\mid 0\le \langle \lambda+\rho,\alpha^\vee\rangle \le p \textnormal{ for all }\alpha\in\Phi^+\}.\]
\end{definition}
\begin{remark}
After shifting by $\rho$, the fundamental alcove consists of weights $\sum a_i\omega_i$ with $a_i\ge 0$ and another condition depending on the highest (co)root. For $\mf{sl}_n$, that condition is $\sum_{i=1}^{\ell} a_i\le p$.
\end{remark}

Now, $\Lambda_p$ (see Definition~\ref{definition:Lambda_p} lies in the union of a finite number (precisely, $|W|/[\Lambda : \rrr ]$, see \cite[\S3.4]{humphreys2006ordinary}) of alcoves. Following \cite{humphreys2006ordinary}, let us make a convention regarding the top and bottom alcoves.

\begin{notation}
The alcove intersecting $\Lambda_p$ is the \textit{top alcove} if it contains $(p-1)\rho$. The alcove intersecting $\Lambda_p$ is the \textit{bottom alcove} if it contains $-\rho$ (note that $-\rho\not\in \Lambda_p$.
\end{notation}

\begin{example}
For type $A_1$, corresponding to $\mf{g}=\mf{sl}_2$, we have the single alcove given by \[\{-\omega,0,\omega,\dots,(p-1)\omega\},\] where $\omega$ denotes the fundamental weight. In this case, 
\[\Lambda_p=\{0,\omega,\dots,(p-1)\omega\}\] lies in this single alcove, consisting of all weights except $-\omega=-\rho$.
\end{example}

\begin{example}
For type $A_2$, corresponding to $\mf{g}=\mf{sl}_3$, then $\Lambda_p$ lies in the union of two alcoves. In particular, $\Lambda_p$ consists of all of the union of these two alcoves except for the ``boundary" edges of the bottom alcove: specifically, the weights $a\omega_1+b\omega_2$ for $a=-1$ or $b=-1$, and $a,b\le p-1$. See \S\ref{subsection:appendix:d_lambda_for_A_2} for a diagram.
\end{example}

\begin{example}
For type $B_2$, corresponding to $\mf{g}=\mf{so}_5$, we have four alcoves covering $\Lambda_p$. See \S\ref{subsection:appendix:d_lambda_for_B_2} for a diagram.
\end{example}

For more complicated Lie algebras, it is convenient to depict the alcoves via graphs, as in \cite[\S3.4]{humphreys2006ordinary}. We represent each alcove by a vertex and the edges represent shared walls, which give the ordering of the alcoves from top to bottom.

\begin{example}
\label{example:diagram_of_alcoves_not_numbered}
The following diagram depicts the alcoves covering $\Lambda_p$, as discussed in \cite[\S3.4]{humphreys2006ordinary}.
\begin{center}
    \begin{tikzpicture}
    \node at (0,0) {\Large $A_1$};
    \draw[fill=black] (1,0) circle (3pt);

    \node at (0,-2) {\Large $A_2$};

    \draw[fill=black] (1,-2) circle (3pt);
    \draw[fill=black] (1,-3) circle (3pt);

    \draw[thick] (1,-2) -- (1,-3);

    \node at (0,-5) {\Large $A_3$};

    \draw[fill=black] (1,-5) circle (3pt);
    \draw[fill=black] (1,-6) circle (3pt);
    \draw[fill=black] ({1-sqrt(2)/2},{-6-sqrt(2)/2}) circle (3pt);
    \draw[fill=black] ({1+sqrt(2)/2},{-6-sqrt(2)/2}) circle (3pt);
    \draw[fill=black] (1,{-6-sqrt(2)}) circle (3pt);
    \draw[fill=black] (1,{-7-sqrt(2)}) circle (3pt);
    
    \draw[thick] (1,-5) -- (1,-6);
    \draw[thick] ({1-sqrt(2)/2},{-6-sqrt(2)/2}) -- (1,-6);
    \draw[thick] ({1+sqrt(2)/2},{-6-sqrt(2)/2}) -- (1,-6);
    \draw[thick] ({1-sqrt(2)/2},{-6-sqrt(2)/2}) -- (1,{-6-sqrt(2)});
    \draw[thick] ({1+sqrt(2)/2},{-6-sqrt(2)/2}) -- (1,{-6-sqrt(2)});
    \draw[thick] (1,{-7-sqrt(2)}) -- (1,{-6-sqrt(2)});

    \node at (5,-2) {\Large $B_2$};

    \draw[fill=black] (6,-2) circle (3pt);
    \draw[fill=black] (6,-3) circle (3pt);
    \draw[fill=black] (6,-4) circle (3pt);
    \draw[fill=black] (6,-5) circle (3pt);

    \draw[thick] (6,-2) -- (6,-3);
    \draw[thick] (6,-3) -- (6,-4);
    \draw[thick] (6,-4) -- (6,-5);

    \node at (10,0) {\Large $G_2$};

    \draw[fill=black] (11,0) circle (3pt);
    \draw[fill=black] (11,-1) circle (3pt);
    \draw[fill=black] (11,-2) circle (3pt);
    \draw[fill=black] (11,-3) circle (3pt);
    \draw[fill=black] (11,-4) circle (3pt);
    \draw[fill=black] ({11-sqrt(2)/2},{-4-sqrt(2)/2}) circle (3pt);
    \draw[fill=black] ({11+sqrt(2)/2},{-4-sqrt(2)/2}) circle (3pt);
    \draw[fill=black] (11,{-4-sqrt(2)}) circle (3pt);
    \draw[fill=black] (11,{-5-sqrt(2)}) circle (3pt);
    \draw[fill=black] (11,{-6-sqrt(2)}) circle (3pt);
    \draw[fill=black] (11,{-7-sqrt(2)}) circle (3pt);
    \draw[fill=black] (11,{-8-sqrt(2)}) circle (3pt);
    
    \draw[thick] (11,0) -- (11,-1);
    \draw[thick] (11,-1) -- (11,-2);
    \draw[thick] (11,-2) -- (11,-3);
    \draw[thick] (11,-3) -- (11,-4);
    \draw[thick] (11,-4) -- ({11-sqrt(2)/2},{-4-sqrt(2)/2});
    \draw [thick] (11,-4) -- ({11+sqrt(2)/2},{-4-sqrt(2)/2});
    \draw [thick] ({11-sqrt(2)/2},{-4-sqrt(2)/2}) -- (11,{-4-sqrt(2)});
    \draw [thick] ({11+sqrt(2)/2},{-4-sqrt(2)/2}) -- (11,{-4-sqrt(2)});
    \draw [thick] (11,{-4-sqrt(2)}) -- (11,{-5-sqrt(2)});
    \draw [thick] (11,{-5-sqrt(2)}) -- (11,{-6-sqrt(2)});
    \draw [thick] (11,{-6-sqrt(2)}) -- (11,{-7-sqrt(2)});
    \draw [thick] (11,{-7-sqrt(2)}) -- (11,{-8-sqrt(2)});

\end{tikzpicture}

\end{center}
\end{example}

\begin{definition}
A weight $\lambda\in\Lambda_p$ is called $p$\textit{-regular} if it lies in the interior of an alcove.
\end{definition}

\begin{proposition}
If $\lambda\in\Lambda_p$ is a $p$-regular weight, maximal in its linkage class (equivalent to the condition that $\lambda$ lies in the top alcove), then $d_\lambda=1$.
\end{proposition}
\begin{proof}
See \cite[\S9.2]{humphreys2006ordinary}.
\end{proof}

\begin{theorem}
\label{thm:behavior_of_d_lambda}
For fixed $p$, $p$-regular $\lambda\in \Lambda_p$, the value $d_\lambda$ depends only on the alcove that $\lambda$ lies in. For $\lambda\in\Lambda_p$ which are not $p$-regular, then $\lambda$ lies on a wall common to two alcoves, and $d_\lambda$ will be the decomposition number belonging to the lower alcove (in the partial ordering of alcoves discussed above). Furthermore, the $d_\lambda$ assigned to each alcove is independent of $p$ for $p\gggg 0$.
\end{theorem}
\begin{proof}
For the statement of fixed $p$, see \cite{humphreys2006ordinary} and \cite{jantzen1974charakterformel}; this essentially follows from the properties of translation functors. For the statement of independence of $p$, see \cite[Theorem~1]{andersen1994representations} (see also \S\ref{subsection:compute_d_lambda}).
\end{proof}

In \cite{jantzen1974charakterformel}, $\lambda$ is described as lying in the ``upper closure" of a unique alcove, from which $d_\lambda$ should be equal to its decomposition number. As a result, we'll make the following convention.
\begin{notation}
When we say that $\lambda\in\Lambda_p$ lies in some alcove, we mean the \textit{upper closure} of an alcove. This convention only affects those $\lambda$ lying on a wall common to two alcoves: in this case, $\lambda$ should be regarded as part of the lower alcove.
\end{notation}

As a result, we may label the alcoves from Example~\ref{example:diagram_of_alcoves_not_numbered} with the values of $d_\lambda$, with no confusion of which value that $d_\lambda$ should take for $\lambda$ lying on the wall between two alcoves.
\begin{example}
Extending Example~\ref{example:diagram_of_alcoves_not_numbered}, we can label each alcove with the decomposition number for all weights in the upper closure of each alcove (again, taken from \cite{humphreys2006ordinary}).
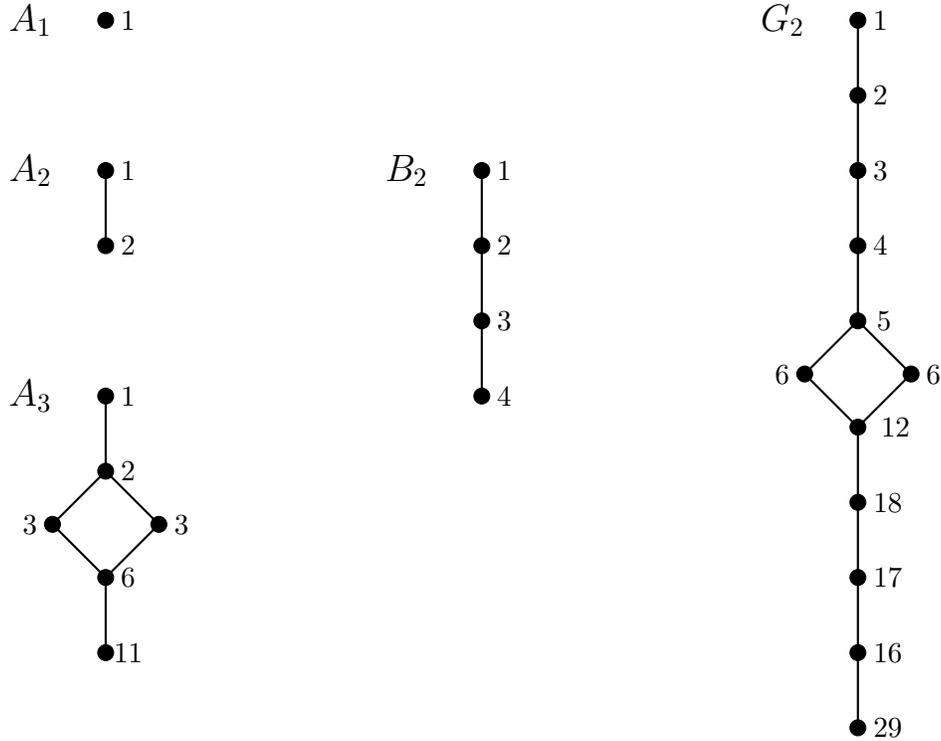
\begin{center}
    \begin{tikzpicture}
    \node at (0,0) {\Large $A_1$};
    \draw[fill=black] (1,0) circle (3pt);
    \node at (1.3,0) {1};
    
    \node at (0,-2) {\Large $A_2$};

    \draw[fill=black] (1,-2) circle (3pt);
    \draw[fill=black] (1,-3) circle (3pt);

    \draw[thick] (1,-2) -- (1,-3);

    \node at (1.3,-2) {1};
    \node at (1.3,-3) {2};
    
    \node at (0,-5) {\Large $A_3$};

    \draw[fill=black] (1,-5) circle (3pt);
    \draw[fill=black] (1,-6) circle (3pt);
    \draw[fill=black] ({1-sqrt(2)/2},{-6-sqrt(2)/2}) circle (3pt);
    \draw[fill=black] ({1+sqrt(2)/2},{-6-sqrt(2)/2}) circle (3pt);
    \draw[fill=black] (1,{-6-sqrt(2)}) circle (3pt);
    \draw[fill=black] (1,{-7-sqrt(2)}) circle (3pt);
    
    \draw[thick] (1,-5) -- (1,-6);
    \draw[thick] ({1-sqrt(2)/2},{-6-sqrt(2)/2}) -- (1,-6);
    \draw[thick] ({1+sqrt(2)/2},{-6-sqrt(2)/2}) -- (1,-6);
    \draw[thick] ({1-sqrt(2)/2},{-6-sqrt(2)/2}) -- (1,{-6-sqrt(2)});
    \draw[thick] ({1+sqrt(2)/2},{-6-sqrt(2)/2}) -- (1,{-6-sqrt(2)});
    \draw[thick] (1,{-7-sqrt(2)}) -- (1,{-6-sqrt(2)});

    \node at (1.3,-5) {1};
    \node at (1.3,-6) {2};
    \node at ({0.7-sqrt(2)/2},{-6-sqrt(2)/2}) {3};
    \node at ({1.3+sqrt(2)/2},{-6-sqrt(2)/2}) {3};
    \node at (1.3,{-6-sqrt(2)}) {6};
    \node at (1.3,{-7-sqrt(2)}) {11};
    
    \node at (5,-2) {\Large $B_2$};

    \draw[fill=black] (6,-2) circle (3pt);
    \draw[fill=black] (6,-3) circle (3pt);
    \draw[fill=black] (6,-4) circle (3pt);
    \draw[fill=black] (6,-5) circle (3pt);

    \draw[thick] (6,-2) -- (6,-3);
    \draw[thick] (6,-3) -- (6,-4);
    \draw[thick] (6,-4) -- (6,-5);
    
    \node at (6.3, -2) {1};
    \node at (6.3, -3) {2};
    \node at (6.3, -4) {3};
    \node at (6.3, -5) {4};
    
    \node at (10,0) {\Large $G_2$};

    \draw[fill=black] (11,0) circle (3pt);
    \draw[fill=black] (11,-1) circle (3pt);
    \draw[fill=black] (11,-2) circle (3pt);
    \draw[fill=black] (11,-3) circle (3pt);
    \draw[fill=black] (11,-4) circle (3pt);
    \draw[fill=black] ({11-sqrt(2)/2},{-4-sqrt(2)/2}) circle (3pt);
    \draw[fill=black] ({11+sqrt(2)/2},{-4-sqrt(2)/2}) circle (3pt);
    \draw[fill=black] (11,{-4-sqrt(2)}) circle (3pt);
    \draw[fill=black] (11,{-5-sqrt(2)}) circle (3pt);
    \draw[fill=black] (11,{-6-sqrt(2)}) circle (3pt);
    \draw[fill=black] (11,{-7-sqrt(2)}) circle (3pt);
    \draw[fill=black] (11,{-8-sqrt(2)}) circle (3pt);
    
    \draw[thick] (11,0) -- (11,-1);
    \draw[thick] (11,-1) -- (11,-2);
    \draw[thick] (11,-2) -- (11,-3);
    \draw[thick] (11,-3) -- (11,-4);
    \draw[thick] (11,-4) -- ({11-sqrt(2)/2},{-4-sqrt(2)/2});
    \draw [thick] (11,-4) -- ({11+sqrt(2)/2},{-4-sqrt(2)/2});
    \draw [thick] ({11-sqrt(2)/2},{-4-sqrt(2)/2}) -- (11,{-4-sqrt(2)});
    \draw [thick] ({11+sqrt(2)/2},{-4-sqrt(2)/2}) -- (11,{-4-sqrt(2)});
    \draw [thick] (11,{-4-sqrt(2)}) -- (11,{-5-sqrt(2)});
    \draw [thick] (11,{-5-sqrt(2)}) -- (11,{-6-sqrt(2)});
    \draw [thick] (11,{-6-sqrt(2)}) -- (11,{-7-sqrt(2)});
    \draw [thick] (11,{-7-sqrt(2)}) -- (11,{-8-sqrt(2)});

    \node at (11.3,0) {1};
    \node at (11.3,-1) {2};
    \node at (11.3,-2) {3};
    \node at (11.3,-3) {4};
    \node at (11.35,-4) {5};
    \node at ({10.7-sqrt(2)/2},{-4-sqrt(2)/2}) {6};
    \node at ({11.3+sqrt(2)/2},{-4-sqrt(2)/2}) {6};
    \node at (11.5,{-4-sqrt(2)}) {12};
    \node at (11.4,{-5-sqrt(2)}) {18};
    \node at (11.4,{-6-sqrt(2)}) {17};
    \node at (11.4,{-7-sqrt(2)}) {16};
    \node at (11.4,{-8-sqrt(2)}) {29};
\end{tikzpicture}

\end{center}
\end{example}

In \S\ref{subsection:compute_d_lambda}, we discuss the general method for computing $d_\lambda$. In ``small" root systems, the decomposition numbers of $d_\lambda$ are completely known (for example, $A_1,A_2,A_3,B_2,G_2$ in \cite{humphreys2006ordinary}), and we describe the consequences in \S\ref{subsection:concrete_applications}.

\subsection{Computation of $d_\lambda$ in general}
\label{subsection:compute_d_lambda}

Let $\mathcal{H}^{\textnormal{aff}}=\mathcal{H}(W^{\textnormal{aff}})$ be the affine Hecke algebra over $\BZ[v^{\pm 1}]$. Recall that this is the algebra with generators $T_{s_i}$, $i=0,1,\ldots,r$ that satisfy relations 
\begin{equation*}
(T_s+1)(T_s-q),
T_{x}T_{y}~\text{for}~\ell(xy)=\ell(x)+\ell(y),
\end{equation*}
here $q=v^2$.
Recall that we have a stratification of $\La$ into alcoves for $W^{\textnormal{aff}}$. Let ${\mathbf{M}}$ be $\BZ[v^{\pm 1}]$-module consisting of formal linear combinations of alcoves. The support of an element $m=\sum_{A}m_A A \in {\mathbf{M}}$ is the set $\operatorname{supp}(m)\coloneqq\{A\,|\, m_A \neq 0\}$. Let ${\mathbf{M}}_c \subset {\mathbf{M}}$ be the set of all $m \in {\mathbf{M}}$ such that $\operatorname{supp}(m)$ is finite. In \cite[3.2]{lusztig1997periodic} (see also \cite[Lemma 9,3]{lusztig1999bases}) the $\mathcal{H}^{\textnormal{aff}}$-module structure on ${\mathbf{M}}_c$ is defined. Let $\leq$ be the partial order on the set of alcoves defined in \cite[1.3]{lusztig1998bases}.
Let ${\mathbf{M}}_{\geq}$ be the set of all $m \in {\mathbf{M}}$ such that $\operatorname{supp}(m)$ is bounded below under $\leq$. We also set $\mathfrak{m}:=\{m=\sum_A m_A A \in {\mathbf{M}}\,|\,m_A \in \BZ[v^{-1}]\}$, $\mathfrak{m}_{\geq}:=\mathfrak{m} \cap M_{\geq}$. In \cite[Lemma 9.16]{lusztig1999bases} certain involution ${\mathbf{b}}'\colon {\mathbf{M}}_{\geq } \rightarrow {\mathbf{M}}_{\geq}$ is defined. It then follows from \cite[Section 9.17]{lusztig1999bases} that for every alcove $B$ there exists the unique element $B_{\geq} \in {\mathfrak{m}}_{\geq}$ such that $B_{\geq}-B \in v^{-1}\mathfrak{m}_{\geq}$ and such that ${\mathbf{b}}'(B_{\geq})=B_{\geq}$.
We can decompose 
\begin{equation*}
B_{\geq } = \sum_{B \leq A} \pi_{B,A} A,    
\end{equation*}
where $\pi_{B,A} \in v^{-1}\BZ[v^{-1}]$ for $B<A$ and $\pi_{B,B}=1$.  
Consider now the category of $(\mathcal{U}_0(\mathfrak{g}),\wt{T})$-modules (see for example \cite[14.1]{lusztig1998bases} or \cite{andersen1994representations}). Blocks of this category are parametrized  by $W^{\textnormal{aff}}$-orbits on $\La$ (see for example \cite{andersen1994representations}). The regular block is in bijection with the set of all alcoves. For an alcove $A$ let $\Delta_A$ be the corresponding baby Verma module and $L_A$ be the corresponding simple $(\mathcal{U}_0(\mathfrak{g}),\wt{T})$-module.
The following result was a conjecture formulated by Lusztig in \cite[17.3]{lusztig1999bases} and holds for $p \gggg 0$ by the results of \cite{bezrukavnikov2013representations}.
\begin{equation}\label{mult_vs_periodic}
[L_B:\Delta_A]=\pi_{B,A}(-1).    
\end{equation}
Forgetting the $\wt{T}$-action and noting that $L_B$ becomes a simple $\mathcal{U}_0(\mathfrak{g})$-module and $\Delta_A$ becomes a baby Verma for $\mathcal{U}_0(\mathfrak{g})$ we can extract numbers $d_\la$ from (\ref{mult_vs_periodic}).

\begin{remark}\label{KL_dim} Using a similar approach one can compute multiplicities of $L_\la$ in Weyl modules for very large $p$ (see \cite[\S11.2]{ciappara2021lectures} for the detailed list of references) and in particular obtain dimensions of $L_\la$ since dimensions of Weyl modules are given by the Weyl character formula. First, we note that $L_\lambda$ is again an irreducible module for $\wt{G}$. Now consider the Weyl module $W_\lambda\in \msf{Rep}~\wt{G} $. We can compute $\dim W_\lambda$ using the Weyl character formula. On the other hand, we have that
\[ [L_\lambda]=\sum c_{\lambda \mu}[W_\mu]\in K_0(\msf{Rep}~\wt{G}),\]
where the $c_{\lambda \mu}$ are read off from Kazhdan-Lusztig polynomials corresponding to the group $W^{\textnormal{aff}}$  (see \cite{kazhdan1979representations}). It follows that
\[\dim L_\lambda = \sum c_{\lambda\mu}\cdot \dim W_\mu.\]
\end{remark}

\subsection{The decomposition numbers for $A_1$}
\label{subsection:appendix:A1}

The root system $A_1$ corresponds to the connected semisimple adjoint group $G=\psl_2$. Recall that there is a single positive (hence simple) root $\alpha$, as shown below.
\begin{center}
    \begin{tikzpicture}
    \foreach\ang in {180, 360}{
     \draw[->,thick] (0,0) -- (\ang:2cm);
    }
    \draw[fill=black] (0,0) circle (2pt);
    \node at (2.3, 0) {$\alpha$};
    \node[anchor=north] at (-3,1) {\Large $A_{1}$};
\end{tikzpicture}

\end{center}
The single fundamental weight $\omega$ satisfies $\alpha=2\omega$. We have that $\Lambda=\zz\omega$, which we will identify freely with $\zz$. We have \[\Lambda_p=\{0,\omega,\dots,(p-1)\omega\},\]
contained within the (closure of the) fundamental alcove, which consists of $\{-\omega,0,\dots,(p-1)\omega\}$.

It is easy to check that for all $\lambda\in \Lambda_p$, we have $W_\lambda=L_\lambda$, and hence 
\begin{lemma}
For $G=\psl_2$ and $\lambda\in\Lambda_p$, we have $d_\lambda=1$.
\end{lemma}

The linkage classes in $\Lambda_p$ are $\{n,p-n-2\}$ for $n=0,1,\dots,p-2$ and $\{p-1\}$.

As a result:
\begin{theorem}
\label{thm:ranks_of_subbundles_PSL2}
Let $G=\psl_2$ and $\lambda\in\Lambda_p$. Then $d_\lambda=1$ for all $\lambda\in\Lambda_p$, and 
\[a_\lambda=
\begin{cases}
1 & \lambda=p-1,\\
2 & \lambda=0,1,\dots,p-2.
\end{cases}\]
As a result, the possible ranks of the vector subbundles in $(\star)$, described in Theorem~\ref{thm:decomp_line_bundle_into_vector_bundles} are as follows:
\[ \textnormal{rk } e_\lambda^i \fr_*\lll e_\mu^j=a_\lambda \cdot d_\lambda \cdot d_\mu\cdot \delta_{\mu\sim \lambda} = \begin{cases}
1 & \lambda=\mu = p-1,\\
2 & \lambda+\mu = p-2\textnormal{ or }\lambda=\mu\ne p-1,\\
0 & \textnormal{else}.
\end{cases}\]
\end{theorem}
\begin{remark}
This is weaker than known results (but generalizes well to more complicated root systems). Since $X_{\psl_2}\cong \pp^3$ is a toric variety, the Frobenius pushforward of any line bundle on $X$ decomposes into the direct sum of line bundles, and we even know exactly which line bundles appear and their multiplicities (see \cite{achinger2010note}).
\end{remark}

\subsection{The decomposition numbers for $A_2$}
\label{subsection:appendix:d_lambda_for_A_2}

The root system $A_2$ corresponds to the Lie algebra $\mf{sl}_3$. This root system has $6$ roots, with two simple roots $\{\alpha,\beta\}$, as shown below.

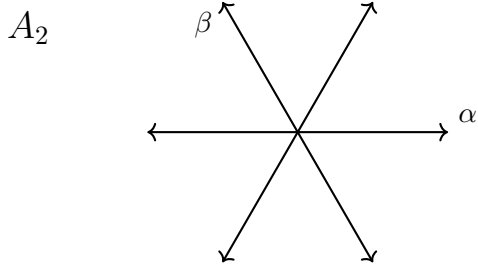
\begin{center}
    \begin{tikzpicture}
    \foreach\ang in {60,120,...,360}{
     \draw[->, thick] (0,0) -- (\ang:2cm);
    }
    \node[anchor=south west] at (2,0) {$\alpha$};
    \node[anchor=north east] at (-1, {sqrt(3)}) {$\beta$};
    \node[anchor=north west] at (-4,{sqrt(3)}) {\Large $A_{2}$};
\end{tikzpicture}
\end{center}

The $p$-restricted weights lie in the union of two alcoves: namely, it is the intersection of the two alcoves shown in the following diagram (represented by black triangles) with the region of dominant weights (indicated by the first quadrant with respect to the blue $\omega_1$ and $\omega_2$-axes).

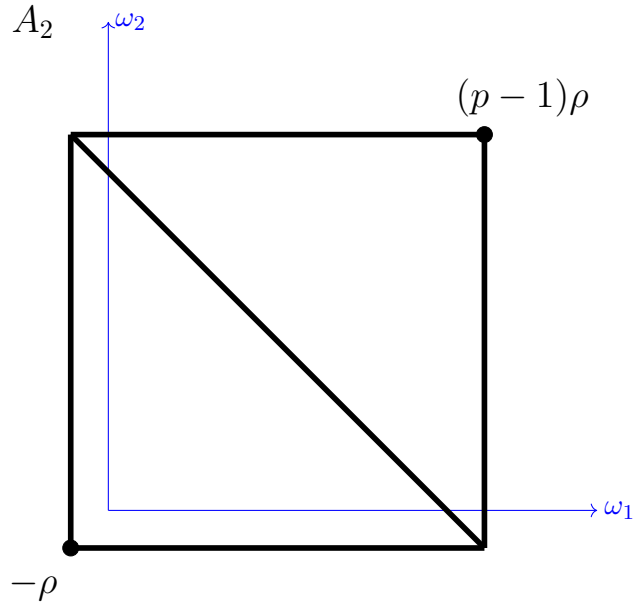
\begin{center}
    \begin{tikzpicture}
    \node at (-1,6.5) {\Large $A_2$};

    \draw[fill=black] (5,5) circle (3pt);
    \draw[fill=black] (-0.5,-0.5) circle (3pt);

    \draw[->, draw=blue] (0,0) -- (0,6.5);
    \draw[->, draw=blue] (0,0) -- (6.5,0);
    \draw[line width = 0.75mm] (-0.5, -0.5) -- (-0.5, 5);
    \draw[line width = 0.75mm] (-0.5, -0.5) -- (5, -0.5);
    \draw[line width = 0.75mm] (-0.5, 5) -- (5, 5);
    \draw[line width = 0.75mm] (5, 5) -- (5, -0.5);
    \draw[line width = 0.75mm] (-0.5, 5) -- (5, -0.5);

    \node[text=blue] at (0.3, 6.5) {$\omega_2$};
    \node[text=blue] at (6.8, 0) {$\omega_1$};
    \node at (-1, -1) {\Large $-\rho$};
    \node at (5.5, 5.5) {\Large $(p-1)\rho$};
\end{tikzpicture}

\end{center}

\begin{proposition}
For type $A_2$, we have that $d_\lambda=1$ if it lies in the top alcove, and $d_\lambda=2$ if it lies in the (upper closure of the) bottom alcove.
\end{proposition}
\begin{proof}
We sketch the proof, presented in \cite[\S5.2]{humphreys2006ordinary} (originally from \cite{verma1975role}).

First, we have a surjection
\[ \Delta_\lambda\onto W_\lambda.\]

But since 
\[ \mu\sim \lambda \implies [Z_\mu] = [Z_\lambda]\in K_0(\msf{Rep}~\uzero),\]
we have that
$[Z_\lambda]$ contains the terms in $\sum_{\mu\sim \lambda}[W_\mu]$ at least once.
From the Weyl character formula,
\[ \sum_{\mu\sim \lambda}\dim W_\mu = \dim \Delta_\lambda\implies [Z_\lambda]=\sum_{\mu\sim\lambda}[W_\mu].\]

Now from \cite{braden1967restricted}, we have that 
\[ [W_\lambda] = 
\begin{cases}
[M_\lambda] + [M_{w_0\cdot \lambda}] & \lambda\textnormal{ in top alcove},\\
[M_\lambda] & \lambda \textnormal{ in bottom alcove}.
\end{cases}\]
(Note that $w_0$ denotes the word of longest length in the Weyl group, with $w_0\cdot\lambda=w_0(\lambda+\rho)-\rho$.)
It follows that for $\lambda$ in the top alcove, we just have $d_\lambda=1$, and for $\lambda$ in the bottom alcove, then $w_0^{-1}\cdot\lambda$ is in the top alcove, and thus
\[ d_\lambda = [M_\lambda : \Delta_\lambda] = [M_\lambda : \Delta_{w_0^{-1}\cdot\lambda}] = 2,\]
where one comes from the summand 
\[[W_{w_0^{-1}\cdot\lambda}]=[M_{w_0^{-1}\cdot\lambda}]+[M_{\lambda}]\] and another comes from the summand
\[ [W_{w_0.w_0^{-1}\cdot\lambda}]=[W_{\lambda}]=[M_\lambda].\]
\end{proof}

The values of $d_\lambda$ can be summarized succintly in the following diagram.

\begin{center}
    \begin{tikzpicture}
    \node at (-2,6) {\Large $A_2$};
    \draw[fill=black] (5,5) circle (3pt);
    \draw[fill=black] (-0.5,-0.5) circle (3pt);

    \draw[->, draw=blue] (0,0) -- (0,6.5);
    \draw[->, draw=blue] (0,0) -- (6.5,0);
    \draw[loosely dotted, line width = 0.5mm] (-0.5, -0.5) -- (-0.5, 5);
    \draw[loosely dotted, line width = 0.5mm] (-0.5, -0.5) -- (5, -0.5);
    \draw[line width = 0.55mm] (-0.5, 5) -- (5, 5);
    \draw[line width = 0.55mm] (5, 5) -- (5, -0.5);
    \draw[draw = red, line width = 0.75mm] (0, 5) -- (5, 5) -- (5, 0) -- (0, 5);
    \draw[draw = orange, line width = 0.75mm] (0, 0) -- (0, 4.5) -- (4.5, 0) -- (0, 0);

    \node[text=blue] at (0.3, 6.5) {$\omega_2$};
    \node[text=blue] at (6.8, 0) {$\omega_1$};
    \node at (-1, -1) {\Large $-\rho$};
    \node at (5.5, 5.5) {\Large $(p-1)\rho$};
    \node[text=red] at (3.5, 3.5) {\Large $\mathbf{1}$};
    \node[text=orange] at (1.2, 1.2) {\Large $\mathbf{2}$};
\end{tikzpicture}
\end{center}

\begin{example}
From \cite[page~26]{humphreys2006ordinary} and \cite[Table~5]{humphreys1973some}, we find the values of $d_\lambda$ for $A_2$ and $p=5$. In the following diagram, we see that for all $\lambda$ within the region enclosed by the red triangle, then $d_\lambda=1$. For the remaining $\lambda\in\Lambda_p$, they lie in the region enclosed by the orange triangle, and $d_\lambda=2$. The red triangle is therefore the ``top alcove" and the orange triangle is the intersection of the ``bottom alcove" with $\Lambda_p$.

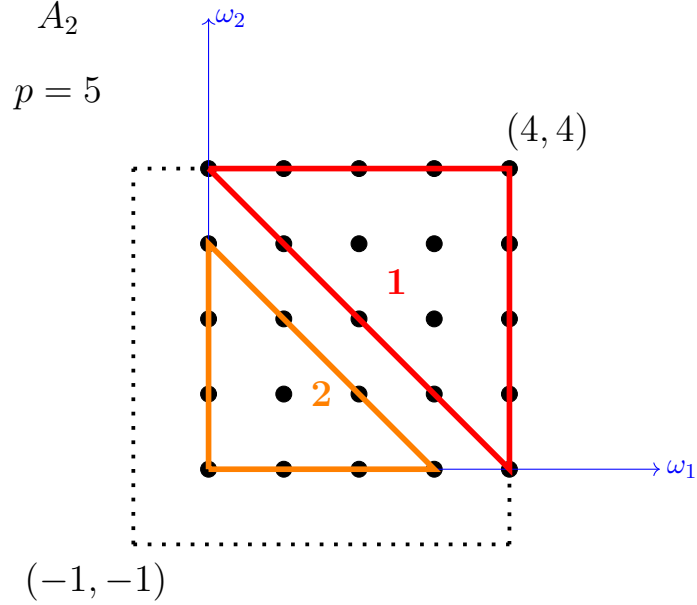
\begin{center}
    \begin{tikzpicture}
    \node at (-2,6) {\Large $A_2$};
    \node at (-2,5) {\Large $p=5$};

    \draw[fill=black] (0,0) circle (3pt);
    \draw[fill=black] (1,0) circle (3pt);
    \draw[fill=black] (2,0) circle (3pt);
    \draw[fill=black] (3,0) circle (3pt);
    \draw[fill=black] (4,0) circle (3pt);
    \draw[fill=black] (0,1) circle (3pt);
    \draw[fill=black] (1,1) circle (3pt);
    \draw[fill=black] (2,1) circle (3pt);
    \draw[fill=black] (3,1) circle (3pt);
    \draw[fill=black] (4,1) circle (3pt);
    \draw[fill=black] (0,2) circle (3pt);
    \draw[fill=black] (1,2) circle (3pt);
    \draw[fill=black] (2,2) circle (3pt);
    \draw[fill=black] (3,2) circle (3pt);
    \draw[fill=black] (4,2) circle (3pt);
    \draw[fill=black] (0,3) circle (3pt);
    \draw[fill=black] (1,3) circle (3pt);
    \draw[fill=black] (2,3) circle (3pt);
    \draw[fill=black] (3,3) circle (3pt);
    \draw[fill=black] (4,3) circle (3pt);
    \draw[fill=black] (0,4) circle (3pt);
    \draw[fill=black] (1,4) circle (3pt);
    \draw[fill=black] (2,4) circle (3pt);
    \draw[fill=black] (3,4) circle (3pt);
    \draw[fill=black] (4,4) circle (3pt);

    \draw[->, draw=blue] (0,0) -- (0,6);
    \draw[->, draw=blue] (0,0) -- (6,0);
    \draw[loosely dotted, line width = 0.5mm] (-1, -1) -- (-1, 4);
    \draw[loosely dotted, line width = 0.5mm] (-1, -1) -- (4, -1);
    \draw[loosely dotted, line width = 0.5mm] (-1, 4) -- (4, 4);
    \draw[loosely dotted, line width = 0.5mm] (4, 4) -- (4, -1);
    \draw[draw = red, line width = 0.75mm] (0, 4) -- (4, 4) -- (4, 0) -- (0, 4);
    \draw[draw = orange, line width = 0.75mm] (0, 0) -- (0, 3) -- (3, 0) -- (0, 0);

    \node[text=blue] at (0.3, 6) {$\omega_2$};
    \node[text=blue] at (6.3, 0) {$\omega_1$};
    \node at (-1.5, -1.5) {\Large $(-1, -1)$};
    \node at (4.5, 4.5) {\Large $(4,4)$};
    \node[text=red] at (2.5, 2.5) {\Large $\mathbf{1}$};
    \node[text=orange] at (1.5, 1) {\Large $\mathbf{2}$};
\end{tikzpicture}

\end{center}
\end{example}
For every $\lambda$ except for $\lambda=(p-1)\rho$, there exists $\mu\sim \lambda$ in both the top alcove and the bottom alcove. The $a_\lambda$ (from Definition~\ref{notation:a_lambda}) can be described as follows. Let $\lambda=a_1\omega_1+a_2\omega_2$, for $0\le a_1,a_2\le p-1$, and let $(a_1+a_2+2,a_2+1,0)\in \ff_p^3$ have type $\overrightarrow{\mathbf{n}}$ (see Definition~\ref{def:type_of_vector}). Then we have the following cases for $a_\lambda$.
\[a_\lambda = \begin{cases}
1 & \overrightarrow{\mathbf{n}}=(3),\\
3 & \overrightarrow{\mathbf{n}}=(2,1),\\
6 & \overrightarrow{\mathbf{n}}=(1,1,1).
\end{cases}\]
We remark that we can also describe it fairly explicitly: let $\lambda = a_1\omega_1+a_2\omega_2$. Then
\[ 
\begin{cases}
(a_1,a_2)=(p-1,p-1)&\implies a_\lambda = 1,\\
(a_1, a_2)=(a,p-2-a)&\implies a_\lambda = 3,\\
(a_1,a_2)=(a,p-1),\quad a\ne p-1&\implies a_\lambda = 3,\\
(a_1,a_2)=(p-1,a),\quad a\ne p-1&\implies a_\lambda = 3,\\
(a_1,a_2)=(a,b),\quad a+b\ne p-2, \quad a,b<p-1 &\implies a_\lambda =6.
\end{cases}
\]

\subsection{The decomposition numbers for $A_3$}
\label{subsection:appendix:A3}
The root system $A_3$ corresponds to the Lie algebra $\mf{sl}_4$. This root system has $3$ simple roots. The $p$-restricted weights $\Lambda_p$ are covered by six alcoves, with decomposition numbers shown below (data from \cite{humphreys2006ordinary}).

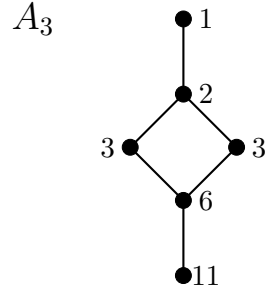
\begin{center}
    \begin{tikzpicture}
    \node at (0,0) {\Large $A_3$};

    \draw[fill=black] (2,0) circle (3pt);
    \draw[fill=black] (2,-1) circle (3pt);
    \draw[fill=black] ({2-sqrt(2)/2},{-1-sqrt(2)/2}) circle (3pt);
    \draw[fill=black] ({2+sqrt(2)/2},{-1-sqrt(2)/2}) circle (3pt);
    \draw[fill=black] (2,{-1-sqrt(2)}) circle (3pt);
    \draw[fill=black] (2,{-2-sqrt(2)}) circle (3pt);
    
    \draw[thick] (2,0) -- (2,-1);
    \draw[thick] ({2-sqrt(2)/2},{-1-sqrt(2)/2}) -- (2,-1);
    \draw[thick] ({2+sqrt(2)/2},{-1-sqrt(2)/2}) -- (2,-1);
    \draw[thick] ({2-sqrt(2)/2},{-1-sqrt(2)/2}) -- (2,{-1-sqrt(2)});
    \draw[thick] ({2+sqrt(2)/2},{-1-sqrt(2)/2}) -- (2,{-1-sqrt(2)});
    \draw[thick] (2,{-2-sqrt(2)}) -- (2,{-1-sqrt(2)});

    \node at (2.3,0) {1};
    \node at (2.3,-1) {2};
    \node at ({1.7-sqrt(2)/2},{-1-sqrt(2)/2}) {3};
    \node at ({2.3+sqrt(2)/2},{-1-sqrt(2)/2}) {3};
    \node at (2.3,{-1-sqrt(2)}) {6};
    \node at (2.3,{-2-sqrt(2)}) {11};
\end{tikzpicture}

\end{center}

Therefore, for $\lambda\in \Lambda_p$, we have
\[d_\lambda = 
\begin{cases}
1 & \lambda \textnormal{ lies in the top alcove},\\
2 & \lambda \textnormal{ lies in the second alcove (from the top)},\\
3 & \lambda \textnormal{ lies in one of the two middle alcoves},\\
6 & \lambda \textnormal{ lies in the second alcove from the bottom},\\
11 & \lambda \textnormal{ lies in the bottom alcove}.
\end{cases}\]
Furthermore, write $\lambda=a_1\omega_1+a_2\omega_2+a_3\omega_3$. Then suppose the vector $(a_1+a_2+a_3+3,a_2+a_3+2,a_3+1,0)\in \ff_p^4$ has type $\overrightarrow{\mathbf{n}}$ (see Definition~\ref{def:type_of_vector}). Then
\[
a_\lambda=\begin{cases}
1 & \overrightarrow{\mathbf{n}}=(4),\\
4 & \overrightarrow{\mathbf{n}}=(3,1),\\
6 & \overrightarrow{\mathbf{n}}=(2,2),\\
12 & \overrightarrow{\mathbf{n}}=(2,1,1),\\
24 & \overrightarrow{\mathbf{n}}=(1,1,1,1).
\end{cases}
\]

\subsection{The decomposition numbers for $B_2$}
\label{subsection:appendix:d_lambda_for_B_2}

Let us first briefly review the root system $B_2$, corresponding to the simple Lie algebra $\mf{so}_5$. The simple roots $\{\alpha,\beta\}$ are given by $\alpha=\varepsilon_1$ and $\beta=\varepsilon_2-\varepsilon_1$, and the positive roots are
\[ \{\alpha,\beta,\alpha+\beta,2\alpha+\beta\},\]
as shown below.

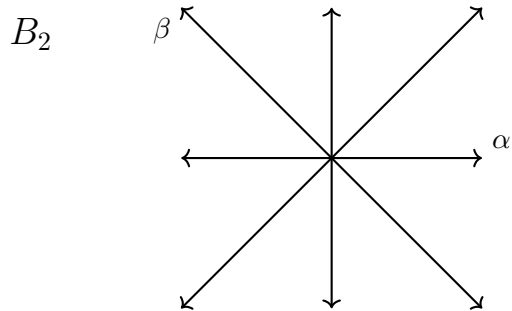
\begin{center}
    \begin{tikzpicture}
    \foreach\ang in {90,180,...,360}{
     \draw[->,thick] (0,0) -- (\ang:2cm);
    }
    \foreach\ang in {45,135,...,335}{
     \draw[->,thick] (0,0) -- (\ang:2.83cm);
    }
    \node[anchor=south west] at (2,0) {$\alpha$};
    \node[anchor=north east] at (-2, 2) {$\beta$};
    \node[anchor=north] at (-4,2) {\Large $B_{2}$};
\end{tikzpicture}

\end{center}

The fundamental roots are $\omega_1=\frac{1}{2}\varepsilon_1+\frac{1}{2}\varepsilon_2$, and $\omega_2=\varepsilon_2$. The Weyl group is $W=S_2\ltimes (\zz/2\zz)^2$.

We have $\Lambda_p=\{a\alpha+b\beta\mid 0\le a,b<p\}$, and $\Lambda_p$ is covered by four alcoves, as shown below. The dotted lines indicate the alcoves, labeled by $a,b,c,d$; the solid black line indicates the region bounded by $\Lambda_p$, and the blue axes are the $\omega_1,\omega_2$-axes.

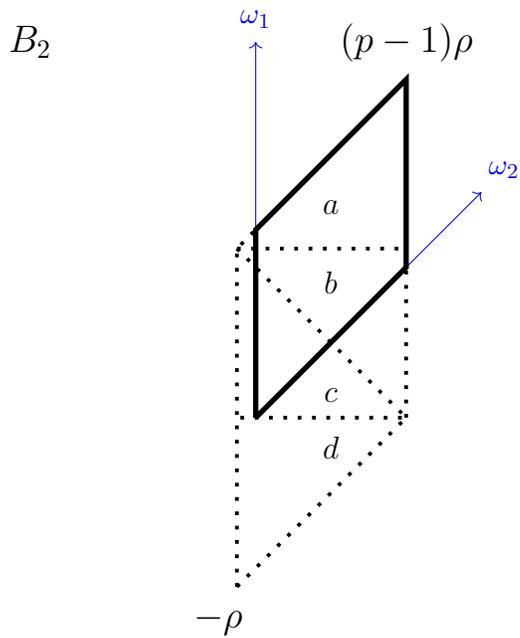
\begin{center}
    \begin{tikzpicture}
    \node at (-3,5) {\Large $B_2$};

    \draw[->, draw=blue] (0,0) -- (0,5);
    \draw[->, draw=blue] (0,0) -- (3,3);

    \draw[line width = 0.7mm] (0,0) -- (0,2.5) -- (2,4.5) -- (2,2) -- (0,0);
    \draw[loosely dotted, line width = 0.5mm] (-0.25,-2.25) -- (-0.25, 2.25) -- (0,2.5);
    \draw[loosely dotted, line width = 0.5mm] (-0.25,-2.25) -- (2,0) -- (-0.25,0);
    \draw[loosely dotted, line width = 0.5mm] (-0.25,2.25) -- (2,2.25);
    \draw[loosely dotted, line width = 0.5mm] (-0.25, 2.25) -- (2,0);
    \draw[loosely dotted, line width = 0.5mm] (2,0) -- (2,2);


    \node at (-0.5, -2.7) {\Large $-\rho$};
    \node at (2, 5) {\Large $(p-1)\rho$};
    \node[text=blue] at (0,5.3) {$\omega_1$};
    \node[text=blue] at (3.3, 3.3) {$\omega_2$};
    
    \node at (1,2.8) {\large $a$};
    \node at (1,1.8) {\large $b$};
    \node at (1,0.3) {\large $c$};
    \node at (1,-0.4) {\large $d$};
\end{tikzpicture}
\end{center}

\begin{proposition}
\label{prop:d_lambda_for_B_2}
The alcoves covering $\Lambda_p$ of $B_2$ can be represented as follows.
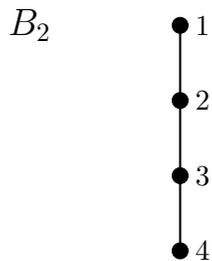
\begin{center}
    \begin{tikzpicture}
    \node at (0,0) {\Large $B_2$};

    \draw[fill=black] (2,0) circle (3pt);
    \draw[fill=black] (2,-1) circle (3pt);
    \draw[fill=black] (2,-2) circle (3pt);
    \draw[fill=black] (2,-3) circle (3pt);

    \draw[thick] (2,0) -- (2,-1);
    \draw[thick] (2,-1) -- (2,-2);
    \draw[thick] (2,-2) -- (2,-3);

    \node at (2.3, 0) {$1$};
    \node at (2.3, -1) {$2$};
    \node at (2.3, -2) {$3$};
    \node at (2.3, -3) {$4$};
\end{tikzpicture}
\end{center} Then
\[d_\lambda=\begin{cases}
1 & \lambda\textnormal{ lies in the top alcove},\\
2 & \lambda\textnormal{ lies in the second alcove},\\
3 & \lambda\textnormal{ lies in the third alcove},\\
4& \lambda\textnormal{ lies in the bottom alcove}.
\end{cases}\]
\end{proposition}
\begin{proof}
See \cite[\S5.2]{humphreys2006ordinary}.
\end{proof}
\begin{example}
From \cite[page~27]{humphreys2006ordinary} and \cite[Table~7]{humphreys1973some}, we obtain the $d_\lambda$ for $\lambda\in\Lambda_p$ for $p=3$. The weights in red have $d_\lambda=1$, the weights in orange have $d_\lambda=2$, the weights in brown have $d_\lambda=3$, and the weights in green have $d_\lambda=4$. The alcove boundaries are shown with dotted lines.
\begin{center}
    \begin{tikzpicture}
    \node at (-3,6) {\Large $B_2$};
    \node at (-3,5) {\Large $p=3$};
    
    \draw[fill=brown] (1,1) circle (3pt);
    \draw[fill=orange] (2,2) circle (3pt);
    \draw[fill=brown] (0,2) circle (3pt);
    \draw[fill=orange] (1,3) circle (3pt);
    \draw[fill=red] (2,4) circle (3pt);
    \draw[fill=red] (0,4) circle (3pt);
    \draw[fill=red] (1,5) circle (3pt);
    \draw[fill=red] (2,6) circle (3pt);

    \draw[->, draw=blue] (0,0) -- (0,6);
    \draw[->, draw=blue] (0,0) -- (3,3);

    \draw[thick] (0,0) -- (0,4) -- (2,6) -- (2,2) -- (0,0);
    \draw[loosely dotted, line width = 0.5mm] (-1,-3) -- (-1, 3) -- (0,4);
    \draw[loosely dotted, line width = 0.5mm] (-1,-3) -- (2,0) -- (-1,0);
    \draw[loosely dotted, line width = 0.5mm] (-1,3) -- (2,3);
    \draw[loosely dotted, line width = 0.5mm] (-1, 3) -- (2,0);
    \draw[loosely dotted, line width = 0.5mm] (2,0) -- (2,2);

    \draw[draw=red, line width = 0.7mm] (0,4) -- (2,4) -- (2,6) -- (0,4);
    \draw[draw=orange, line width = 0.7mm] (2,2) -- (1,3);
    \draw[draw=brown, line width = 0.7mm] (1,1) -- (0,2);

    \draw[fill=green, draw=green] (0,0) circle (3pt);

    \node at (0,-0.3) {$(0,0)$};
    \node at (1.6, 1) {$(0,1)$};
    \node at (2.6, 2) {$(0,2)$};
    \node at (-0.6, 2) {$(1,0)$};
    \node at (1, 3.3) {$(1,1)$};
    \node at (2.6, 4) {$(1,2)$};
    \node at (-0.6, 4) {$(2,0)$};
    \node at (0.4, 5) {$(2,1)$};

    \node[text=red] at (1.5, 4.5) {\LARGE {$\mathbf{1}$}};
    \node[text=orange] at (1.2, 2.2) {\LARGE {$\mathbf{2}$}};
    \node[text=brown] at (0.4, 1.2) {\LARGE {$\mathbf{3}$}};
    \node[text=green] at (0.5, -0.1) {\LARGE {$\mathbf{4}$}};

    \node at (-1, -3.5) {\Large $(-1, -1)$};
    \node at (2, 6.5) {\Large $(2,2)$};
    \node[text=blue] at (0,6.3) {$\omega_1$};
    \node[text=blue] at (3.3, 3.3) {$\omega_2$};
\end{tikzpicture}

\end{center}

\end{example}

\begin{remark}
In \cite{braden1967certain} (summarized in \cite{braden1967restricted}), the dimensions of the $L_\lambda$ are computed, giving the number of components $\fff_{\mu,\lambda}^{i,j}$ associated to each pair of weights $(\mu,\lambda)$.
\end{remark}

\subsection{The decomposition numbers for $G_2$}
\label{subsection:appendix:G2}
The root system $G_2$ corresponds to the Lie algebra $\mf{g}_2$, and has two simple roots, with six positive roots, as shown below. The Weyl group is $W\cong D_6\cong \zz/6\zz\ltimes \zz/2\zz$.

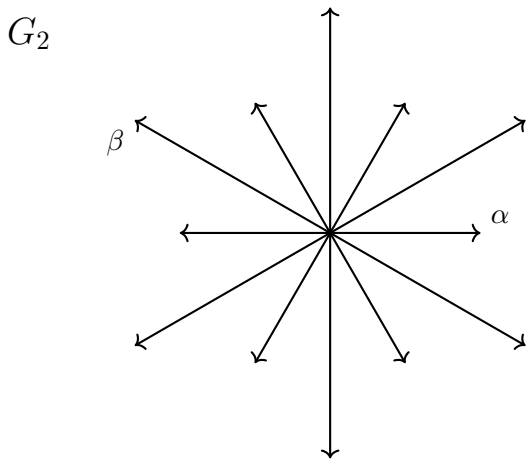
\begin{center}
    \begin{tikzpicture}
    \foreach\ang in {60,120,...,360}{
     \draw[->,thick] (0,0) -- (\ang:2cm);
    }
    \foreach\ang in {30,90,...,330}{
     \draw[->,thick] (0,0) -- (\ang:3cm);
    }
    \node[anchor=south west] at (2,0) {$\alpha$};
    \node[anchor=north east] at ({-sqrt(3)*3/2},3/2) {$\beta$};
    \node[anchor=north] at (-4,3) {\Large $G_{2}$};
\end{tikzpicture}
\end{center}

The $p$-restricted weights $\Lambda_p$ are covered by twelve alcoves, and their decomposition numbers are shown below (data from \cite{humphreys2006ordinary}).

\begin{center}
    \begin{tikzpicture}
    \node at (0,0) {\Large $G_2$};

    \draw[fill=black] (2,0) circle (3pt);
    \draw[fill=black] (2,-1) circle (3pt);
    \draw[fill=black] (2,-2) circle (3pt);
    \draw[fill=black] (2,-3) circle (3pt);
    \draw[fill=black] (2,-4) circle (3pt);
    \draw[fill=black] ({2-sqrt(2)/2},{-4-sqrt(2)/2}) circle (3pt);
    \draw[fill=black] ({2+sqrt(2)/2},{-4-sqrt(2)/2}) circle (3pt);
    \draw[fill=black] (2,{-4-sqrt(2)}) circle (3pt);
    \draw[fill=black] (2,{-5-sqrt(2)}) circle (3pt);
    \draw[fill=black] (2,{-6-sqrt(2)}) circle (3pt);
    \draw[fill=black] (2,{-7-sqrt(2)}) circle (3pt);
    \draw[fill=black] (2,{-8-sqrt(2)}) circle (3pt);
    
    \draw[thick] (2,0) -- (2,-1);
    \draw[thick] (2,-1) -- (2,-2);
    \draw[thick] (2,-2) -- (2,-3);
    \draw[thick] (2,-3) -- (2,-4);
    \draw[thick] (2,-4) -- ({2-sqrt(2)/2},{-4-sqrt(2)/2});
    \draw [thick] (2,-4) -- ({2+sqrt(2)/2},{-4-sqrt(2)/2});
    \draw [thick] ({2-sqrt(2)/2},{-4-sqrt(2)/2}) -- (2,{-4-sqrt(2)});
    \draw [thick] ({2+sqrt(2)/2},{-4-sqrt(2)/2}) -- (2,{-4-sqrt(2)});
    \draw [thick] (2,{-4-sqrt(2)}) -- (2,{-5-sqrt(2)});
    \draw [thick] (2,{-5-sqrt(2)}) -- (2,{-6-sqrt(2)});
    \draw [thick] (2,{-6-sqrt(2)}) -- (2,{-7-sqrt(2)});
    \draw [thick] (2,{-7-sqrt(2)}) -- (2,{-8-sqrt(2)});

    \node at (2.3,0) {1};
    \node at (2.3,-1) {2};
    \node at (2.3,-2) {3};
    \node at (2.3,-3) {4};
    \node at (2.35,-4) {5};
    \node at ({1.7-sqrt(2)/2},{-4-sqrt(2)/2}) {6};
    \node at ({2.3+sqrt(2)/2},{-4-sqrt(2)/2}) {6};
    \node at (2.5,{-4-sqrt(2)}) {12};
    \node at (2.4,{-5-sqrt(2)}) {18};
    \node at (2.4,{-6-sqrt(2)}) {17};
    \node at (2.4,{-7-sqrt(2)}) {16};
    \node at (2.4,{-8-sqrt(2)}) {29};
\end{tikzpicture}
\end{center}

\frenchspacing
\bibliographystyle{alpha}
\bibliography{bibl}

\end{document}